%% file: ms.tex
\documentclass{amsart}
\pdfoutput=1
\usepackage[backend=biber, backref, style=alphabetic]{biblatex}
\usepackage{amsmath, amsthm, amsfonts, amssymb}
\usepackage[mathcal]{euscript}
\usepackage{tikz-cd}
\usepackage{mathrsfs}
\usepackage{hyperref}
\usepackage{color}
\usepackage[capitalize,nameinlink,noabbrev]{cleveref}
\usepackage[T1]{fontenc}
\input{macros.tex}

\numberwithin{equation}{theorem}

\title{Berkovich log discrepancies in positive characteristic}
\author{Eric Canton}
\thanks{Partially supported by DMS \#1606414}
\address{
  Department of Mathematics\\
  University of Michigan\\
  Ann Arbor, MI 48109
}
\email{ecanton@umich.edu}

\begin{document}
\begin{abstract}
  We introduce and study a log discrepancy function on the space of semivaluations 
  centered on an integral noetherian scheme of positive characteristic. Our definition 
  shares many properties with the analogue in characteristic zero; we prove that if log 
  resolutions exist, then our definition agrees with previous approaches to 
  log discrepancies of semivaluations that these resolutions. We then apply 
  this log discrepancy to a variety of topics in singularity theory over fields 
  of positive characteristic. Strong $F$-regularity and sharp $F$-purity of Cartier 
  subalgebras are detected using positivity and non-negativity of log discrepancies 
  of semivaluations, just as Kawamata log terminal and log canonical singularities 
  are defined using divisorial log discrepancies, making precise a long-standing 
  heuristic. We prove, in positive characteristic, several theorems of 
  Jonsson and Musta\c{t}\u{a} in characteristic zero regarding log canonical 
  thresholds of graded sequences of ideals. Along the way, we give a valuation-theoretic 
  proof that asymptotic multiplier ideals are coherent on strongly $F$-regular schemes.
\end{abstract}
\maketitle
\tableofcontents

\input{zNewIntro}

\input{zNotation}

\input{zLog-discrepancies-positive-characteristic}

\input{zMainTheorem}

\input{zF-singularities}
\input{zLSC}

\input{zLCT-of-graded-sequences}

\printbibliography

\end{document}

%% file: macros.tex
\newcommand{\im}{\mathrm{im}}

\newcommand{\isom}{\cong}
\newcommand{\into}{\hookrightarrow}

\newcommand{\inv}{^{-1}}

\newcommand{\comment}[1]{}
\newcommand{\ev}{\mathrm{ev}}

\newcommand{\tensor}{\otimes}

\DeclareMathOperator{\Spec}{Spec}

\newcommand{\A}{\mathbb{A}}       
\newcommand{\Bl}{\mathrm{Bl}}     
\newcommand{\V}{\mathbb{V}}       

\renewcommand{\div}{\text{div}}

\newcommand{\shC}{\mathcal{C}}
\newcommand{\shD}{\mathcal{D}}
\newcommand{\shE}{\mathcal{E}}
\newcommand{\shF}{\mathcal{F}}
\newcommand{\shG}{\mathcal{G}}
\newcommand{\shH}{\mathcal{H}}

\newcommand{\shI}{\mathcal{I}}
\newcommand{\shJ}{\mathcal{J}}

\newcommand{\shL}{\mathcal{L}}

\newcommand{\shO}{\mathcal{O}}

\newcommand{\shR}{\mathcal{R}}

\newcommand{\shHom}{{\mathcal{H}om}}

\newcommand{\mld}{\mathrm{mld}}
\newcommand{\lct}{\mathrm{lct}}

\newcommand{\D}{\mathfrak{D}}
\newcommand{\ord}{\mathrm{ord}}
\newcommand{\ratrk}{\mathrm{ratrk}}

\newcommand{\triv}{\mathrm{triv}}
\newcommand{\val}{\mathrm{val}}
\newcommand{\Val}{\mathrm{Val}}

\newcommand{\ldb}{[\![}   
\newcommand{\rdb}{]\!]}
\newcommand{\lbrak}{\{\!\!\{}
\newcommand{\rbrak}{\}\!\!\}}

\DeclareMathOperator{\trdeg}{tr.deg}

\newcommand{\m}{\mathfrak{m}}
\newcommand{\n}{\mathfrak{n}}
\newcommand{\p}{\mathfrak{p}}
\newcommand{\q}{\mathfrak{q}}
\newcommand{\idealr}{\mathfrak{r}}

\newcommand{\ideala}{\mathfrak{a}}
\newcommand{\idealb}{\mathfrak{b}}
\newcommand{\idealc}{\mathfrak{c}}

\newcommand{\fpow}[1]{\ensuremath ^{[p^{#1}]}}

\newcommand{\Frac}{\mathrm{Frac}}

\DeclareMathOperator{\Hom}{Hom}

\mathchardef\mhyph="2D 

\newcommand{\mb}{\mathbb}
\newcommand{\mc}{\mathcal}

\newcommand{\mf}{\mathfrak}

\newcommand{\N}{\mb{N}}
\newcommand{\Q}{\mb{Q}}
\newcommand{\R}{\mb{R}}
\newcommand{\Z}{\mb{Z}}
\newcommand{\C}{\mb{C}}
\newcommand{\bF}{\mb{F}}
\newcommand{\e}{\varepsilon}

\newtheorem{theorem}{Theorem}[section]
\newtheorem*{theorem*}{Theorem}
\newtheorem*{mainthm}{Main Theorem}
\newtheorem*{theoremA}{Theorem A}
\newtheorem*{theoremB}{Theorem B}
\newtheorem*{theoremC}{Theorem C}
\newtheorem*{theoremD}{Theorem D}
\newtheorem{corollary}[theorem]{Corollary}
\newtheorem*{corollary*}{Corollary}
\newtheorem{proposition}[theorem]{Proposition}
\newtheorem{lemma}[theorem]{Lemma}

\theoremstyle{remark}
\newtheorem{question}{Question}[theorem]

\theoremstyle{definition}

\newtheorem{remark}[theorem]{Remark}
\newtheorem{definition}[theorem]{Definition}
\newtheorem{example}[theorem]{Example}
\newtheorem{conjecture}[theorem]{Conjecture}
\newtheorem*{conjecture*}{Conjecture}

\newenvironment{pfClaim}{\,{\em Proof of claim:}}{\hfill$\boxtimes$\newline\newline}

%% file: zNewIntro.tex
\section{Introduction}

One of the fundamental ways to study singularities of 
normal varieties of dimension at least 3
is through the {\em log discrepancies} of real 
valuations on the variety. Log discrepancies were discovered
by Mori's school in the 1980s as part of the development of 
the minimal model program; they have since found applications 
across algebraic geometry and commutative algebra. 

During the 1990s and early 2000s, researchers began to realize
deep connections between classes of singularities defined 
using log discrepancies and resolutions of singularities
(e.g.: rational, log canonical, and Kawamata log terminal) 
and singularities appearing in tight closure 
(e.g.: F-rational, F-pure, and F-regular, resp.) 
\cite{Fedder83, MehtaSrinivasFPureSurface, MehtaSrinivasRatImpliesFRat, 
SmithFRatImpliesRat, HaraWatanabe, HaraYoshida}. 
The groundbreaking connection was made by Hara and Watanabe in 
\cite{HaraWatanabe}, where the authors showed that splittings 
of the Frobenius morphism on normal $\Q$-Gorenstein varieties
can be converted into divisors giving log canonical pairs. 
We build on this connection, extending many of their ideas 
to the setting of normal $F$-finite schemes, and beyond. 

\subsection{Log discrepancies}

Let $X$ be a normal variety over an algebraically closed field $k$. 
Log discrepancies were classically defined only for 
divisorial valuations, meaning those of the form $c\,\ord_E$ for 
a real number $c > 0$ and a prime divisor $E \subset Y$ 
on a normal variety with a proper birational morphism $\pi: Y \to X$.


Extending the log discrepancy function to the space $\Val_X$ 
of valuations centered on $X$ goes back at least to 
Favre and Jonsson's {\em Valuative Tree}, \cite{FavreJonsson}, 
for smooth complex surfaces, where they called it {\em thinness}.  
Numerous groups of authors developed the theory more generally
in higher dimensions and on singular varieties 
\cite{dFH, JonssonMustata, BdFFU, BoucksomFavreJonssonNonArchMAEquation, BlumThesis}. 
The approach taken 
depends on using resolutions of singularities to find subspaces 
of $\Val_X$ that have simple, cone-like structures, and so
is presently unavailable in positive characteristics and dimensions
greater than three. 

%
%
%
%

The goals of this paper are to 
extend log discrepancies to $\Val_X$ in characteristic $p > 0$ 
making systematic use of $p^{-e}$-linear maps, and to show
that this is a good extension by demonstrating a number of
properties enjoyed by log-resolution-based extensions
to $\Val_X$ are also enjoyed by our definition. 
In fact, we prove that the two approaches yield the {\em same} log discrepancy
function should one have log resolutions (e.g. on surfaces, and on 
3-dimensional varieties when $p \ge 5$); see \cref{agreement with JM}. 
This is our main theorem, and should be compared with 
Mauri, Mazzon, and Stevenson's comparison of log discrepancies
in characteristic zero with Temkin's canonical metrics, and Temkin's
comparision with Musta\c{t}\u{a} and Nicaise's weight metrics
\cite{MauriMazzonStevenson, TemkinMetrization, MustataNicaiseWeightFunctions}.
It is interesting to observe that Brezner and Temkin's different function 
for a finite morphism between Berkovich curves is closely related to our definition
when applied to the Frobenius morphism (if this is finite) \cite{BreznerTemkin}.

\subsection{Statement of results}
By a pair $(X, \Delta)$ we mean $X$ is a normal variety
over an algebraically closed field and $\Delta \ge 0$ is
a $\Q$-Weil divisor on $X$. 
The starting point for this article is the following 
observation, due to Cascini, Musta\c{t}\u{a}, and Schwede, 
shared with the present author in private correspondence. 

\begin{proposition}[{\Cref{CMS3}}]\label{CMS1}
  Let $(X, \Delta)$ be a pair over a field with
  characteristic $p > 0$. Assume $(1-p^e)(K_X + \Delta)$
  is an integral Cartier divisor for some $e > 0$, and let
  $\psi_\Delta: \shO_X((p^e-1)\Delta) \to \shO_X$ 
  be the associated $p^{-e}$-linear map. 
  For every divisor $E$ over $X$, one has
  \[
    A_{(X, \Delta)}(E) = 
    \sup\{ (p^e-1)\inv \ord_E(f) \,:\, f \in k(X), \psi_\Delta(f) = 1\}. 
  \]
\end{proposition}

This expression is 
quite similar to the definition of $A_{X}(E)$ in 
the approaches to log discrepancies on non-$\Q$-Gorenstein varieties
of \cite{dFH,BdFFU}. We model our definition of log discrepancy
on this proposition, incorporating also a supremum over
$e \ge 1$ with $(1-p^e)(K_X + \Delta)$ Cartier.
The main result of this paper is that this extension to arbitrary 
valuations matches Jonsson and Musta\c{t}\u{a}'s in the 
presence of log resolutions.  

\begin{mainthm}[\Cref{agreement with JM}]
  Let $(X, \Delta)$ be a pair over a field with
  characteristic $p > 0$. Suppose $(1-p^e)(K_X + \Delta)$
  is an integral Cartier divisor for some $e > 0$, 
  and suppose log resolutions exist 
  for varieties over $k$ of dimension $\dim(X)$. 
  For valuations $v \in \Val_X$, 
  denote by $A_{(X, \Delta)}(v)$ the log discrepancy defined 
  as in \cite{JonssonMustata}, and by $A(v; \shC^X \cdot \Delta)$ 
  our log discrepancy from \Cref{Section: Log discrepancy}. Then 
  $A_{(X, \Delta)}(v) = A(v; \shC^X \cdot \Delta)$ 
  for all $v \in \Val_X$.
\end{mainthm}

This theorem proves that our approach is the 
``correct'' definition of log discrepancies on $\Val_X$, cf.
\cite{MauriMazzonStevenson,MustataNicaiseWeightFunctions, TemkinMetrization,
BreznerTemkin}. Because of this theorem, for the rest of this 
introduction we write $A_X$ for the quantity that would be written 
$A(-; \shC^X)$ in the notation of \S \ref{Section: Log discrepancy}. 

As one might hope, we can extend Hara and Watanabe's result
that $\Q$-Gorenstein normal varieties that are
$F$-pure (resp. $F$-regular) are 
log canonical (resp. Kawamata log terminal; klt). 
We no longer need the normal nor $\Q$-Gorenstein assumptions, 
understanding log canonical (resp. klt) 
to mean $A_X(E) \ge 0$ (resp. $> 0$) 
for all divisors $E$ over $X$. In fact, we can 
say much more using our approach. The new
tool is our extension of log discrepancies 
to the space $X^\beth$ of {\em semi}valuations
on $X$, a compactification of $\Val_X$ common
to non-archimedean geometry. 

\begin{theoremA}[\cref{characterization of ShFP and SFR,HW corollary}, 
  cf. {\cite[Theorem 3.3]{HaraWatanabe}}]
  Let $X$ be an $F$-finite integral scheme. 
  \begin{enumerate}
    \item If $X$ is $F$-pure, then $X$ is log canonical. 
    \item If $X$ is $F$-regular, then $X$ is klt. 
    \item Conversely, if $A_X(\xi) \ge 0$ (resp. $> 0$)
      for all $\xi \in X^\beth$ besides the trivial
      valuation on $X$, then $X$ is $F$-pure
      (resp. $F$-regular). 
  \end{enumerate}
\end{theoremA}

A consequence of the third statement is that the $F$-pure
centers of a sharply $F$-pure variety are 
identified as those points whose minimal 
{$\beth$}-log discrepancy is zero, 
cf. \Cref{detecting F-things} and \S\ref{Section: LSC}.

One of the most important properties of the extension of log
discrepancies to $\Val_X$ in characteristic zero is 
lower-semicontinuity; we prove that our extension also
has this property in \cref{log discrepancy lsc}. As
a corollary, we generalize a result of Ambro over $\C$:
the minimal ($\beth$-)log discrepancy 
is lower-semicontinuous, for any Cartier subalgebra on an
integral scheme $X$ of characteristic $p > 0$, 
if we consider $X$ with the 
\textbf{constructible topology}; cf. \cref{mld is lsc} and
\cite[Theorem 2.2]{AmbroOnMinimalLogDiscrepancies}. 

In many ways, the similarities between our log discrepancies
and the characteristic zero analogues amount to this
shared lower-semicontinuity. As a demonstration of this assertion, 
we prove, in the new setting of regular $F$-finite schemes,
the main theorems from \cite{JonssonMustata} regarding valuations 
computing the log canonical threshold $\lct(X, \ideala_\star)$ of 
a multiplicatively graded sequence of ideals. For simplicity,
we state our results for a fixed smooth variety $X$ over an
algebraically closed field $k$ in this introduction. 


Let $\ideala_\star$ be a graded sequence of ideals on $X$. 
We write $\shJ(X, \ideala_\star^t)$ for the 
sheaf of ideals whose sections 
over an affine open $U = \Spec(R) \subseteq X$ are those 
$f \in R$ satisfying
\[ v(f) + A_X(v) -  (t/m)\,v(\ideala_m) > 0 \]
for all $v \in \Val_U$ and all $m \gg 1$.
This ideal is well-known (in characteristic zero) as the 
{\em asymptotic multiplier ideal} of $(X, \ideala_\star^t)$. 
These are coherent whenever a log resolution exists for 
all pairs $(X, \ideala_m^{t/m})$, being defined in this case 
as the pushforward of a certain invertible sheaf on any log resolution 
of $(X, \ideala_m^{t/m})$ for $m \gg 0$ divisible enough 
\cite[Definition 1.4]{EinLazarsfeldSmithSymbolic}, cf. 
\cite[Lemma 11.1.1]{LazarsfeldPositivity2}. Lacking resolutions
in positive characteristics, it was unknown if these ideals were coherent,
and a reasonable expectation would be that a proof of coherence would require
a sufficient theory of resolutions of singularities. 
A significant result of this paper is a purely valuative proof 
that multiplier ideals are coherent on strongly $F$-regular schemes,
e.g. smooth varieties. 

\begin{theoremB}[\cref{coherence}]
  Let $X$ be a smooth variety of characteristic $p$. The 
  asymptotic multiplier ideal sheaf $\shJ(X, \ideala_\star^t)$ 
  is coherent for any graded sequence of ideals $\ideala_\star$ 
  on $X$ and $t \in [0, \infty)$. 
\end{theoremB}

The main technical statement one needs in this proof is 
\cref{V_t compact}, which identifies a compact subset of $\Val_X$
that is large enough to define these multiplier ideals and 
minimize the infima in log canonical thresholds. A number 
of subsets like this are known in the literature,
cf. \cite[Proposition 5.9]{JonssonMustata} and
\cite[Lemma 3.4, Theorem 3.1]{BdFFU}. If one has such a statement,
it is reasonable to expect that our proof of \ref{coherence} may 
be adapted to that setting; see \cref{example: char 0 coherence}. 

Once we have \cref{coherence}, we can adapt the
argument of Jonsson and Musta\c{t}\u{a}, with technical 
modifications, extending the following results to positive 
characteristics. 

\begin{theoremC}[\cref{existence of computing valuation}; 
  cf \cite{JonssonMustata}, Theorem A]
  Let $X$ be a smooth variety of characteristic $p$, 
  and let $\ideala_\star$ be a graded sequence of 
  ideals on $X$. Suppose $\lambda = \lct(\ideala_\star) < \infty$. 
  For any generic point $x$ of an irreducible component 
  of $\V(\shJ(X, \ideala_\star^\lambda))$ there exists 
  a valuation with center $x$ computing $\lct(\ideala_\star)$, 
  i.e. such that $\lambda = A_X(v)/v(\ideala_\star)$. 
\end{theoremC}

These computing valuations are obtained here and 
in \cite{JonssonMustata} using a compactness argument, 
and their properties (e.g. Abhyankar) do not seem to 
be revealed from the proof. Jonsson and Musta\c{t}\u{a} 
conjecture that these valuations must be {\em quasi-monomial}, 
a condition equivalent to Abhyankar for excellent 
schemes in characteristic zero. We state the analogous 
conjectures here using Abhyankar valuations,
all of which are {\em locally} quasi-monomial \cite{KnafKuhlmann}. 

\begin{conjecture*}
  [\cref{JM: Conjecture 7.4}; 
  cf. \cite{JonssonMustata}, Conjecture B]\label{JM: Conjecture B}
  Let $X$ be a smooth variety of characteristic $p$, 
  and let $\ideala_\star$ be a graded sequence of ideals
  on $X$ such that $\lct(\ideala_\star) < \infty$. 
  \begin{itemize}
    \item {\bf Weak version}: some Abhyankar
      valuation computes $\lct_B(\ideala_\star)$. 
    \item {\bf Strong version}: any valuation 
      computing $\lct_B(\ideala_\star)$ is Abhyankar. 
  \end{itemize}
\end{conjecture*}

Following \cite{JonssonMustata}, we reduce this conjecture 
to what is hopefully a more approachable form on affine spaces. 

\begin{conjecture*}
  [\cref{JM: Conjecture 7.5}; 
  cf. \cite{JonssonMustata}, Conjecture C]\label{JM: Conjecture C}
  Let $X = \A^n_k$, with $\bF_p \subset k = \overline{k}$,
  and let $\ideala_\star$ be a graded sequence of ideals 
  on $X$ with $\lct(\ideala_\star) < \infty$, 
  vanishing only at a closed point $x \in X$. 
  \begin{itemize}
    \item {\bf Weak version}: some Abhyankar
      valuation centered at $x$ computes $\lct(\ideala_\star)$. 
    \item {\bf Strong version}: any valuation of 
      transcendence degree $0$ over $\A^n_k$, centered at $x$, 
      and computing $\lct(\ideala_\star)$ must be Abhyankar. 
  \end{itemize}
\end{conjecture*}

\begin{theoremD}
  [\cref{affine conjecture implies regular conjecture};
  cf. \cite{JonssonMustata}, Theorem D]
  If \cref{JM: Conjecture 7.5} holds for all $n \le d$, then 
  \cref{JM: Conjecture 7.4} holds for all $X$ with $\dim(X) \le d$. 
\end{theoremD}

\subsection{Structure of the Paper}

\Cref{Section: Notation} introduces notation and some background.
In \Cref{Section: Log discrepancy}, we define our log discrepancy for 
Cartier subalgebras of integral schemes in characteristic $p > 0$, and prove 
some essential (but elementary) properties. In \Cref{Section: Main Theorem}, 
we undertake the proof of our main theorem, starting with Cascini, Musta\c{t}\u{a}, 
and Schwede's proof of their result. 
\Cref{Section: F-things} is a brief study of the 
relationship between our log discrepancies, $F$-purity, and strong 
$F$-regularity; here we generalize Hara and Watanabe's result. 
We prove in \Cref{Section: LSC} that log discrepancies are 
lower-semicontinuous on $X^\beth$, and use this to deduce constructible 
lower-semicontinuity of the minimal log discrepancy on any integral
scheme of positive characteristic. In \Cref{Section: LCT}, we extend to 
positive characteristics the aforementioned theorems of Jonsson and 
Muta\c{t}\u{a}; we prove along the way that asymptotic 
multiplier ideals are coherent sheaves on strongly $F$-regular schemes.

\textbf{Acknowledgements}:
  This paper is part of the author's Ph.D. thesis. 
  He thanks Tommaso de Fernex, Mattias Jonsson, Devlin Mallory,
  Mircea Musta\c{t}\u{a}, Karl Schwede, Matthew Stevenson, 
  Kevin Tucker, and Wenliang Zhang for numerous helpful 
  discussions. This project grew out of discussions with 
  Rankeya Datta, Felipe P\'erez, and Karl at the 2015 AMS MRC 
  in Commutative Algebra.  The author is indebted to 
  Paolo Cascini, Mircea, and Karl for sharing their result on 
  log discrepancies of divisorial valuations, upon which his 
  definition of log discrepancy is modeled.  He is especially 
  grateful to his Ph.D. advisor, Wenliang, for organizing and 
  supporting (DMS \#1606414) his visit to the University of 
  Utah during the Spring 2016 semester, during which he 
  worked out the key ideas for much of this paper. Thanks also to his thesis 
  committee: Wenliang, Lucho Avramov, Brian Harbourne, and Mark Walker, 
  for their careful readings of drafts of this document. Many
  thanks are due to the referee, whose suggestions greatly 
  improved the presentation and mathematics in this article. 
  Finally, the author thanks the Mathematics Department 
  at the University of Nebraska -- Lincoln for providing such a 
  rich environment in which to learn and grow mathematically, 
  and at the University of Utah for being such wonderful hosts.

%% file: zNotation.tex
\newpage

\section{Notation and background}
\label{Section: Notation}
\noindent Let us first establish conventions, and gather definitions, used 
throughout the paper. 

\subsection{Conventions}
\label{Conventions}
The following basic terminology is used throughout this article.
\begin{enumerate}
  \item A {\em ring} always has a multiplicative identity in this paper. 
    Except for Cartier subalgebras, all rings are commutative. 

  \item The group of units of a (commutative) ring $R$ is denoted $R^\times$. 

  \item By a {\em scheme} we (almost) always mean a separated noetherian scheme.
    The {\em sole} exception is that we may consider $\Spec(V)$ for a valuation 
    ring $V$; we will be explicit about this exception. 
    A point of a scheme refers to any (not necessarily closed) point. 

  \item A {\em variety} over a field $k$ is an integral scheme
    of finite type over $k$. 

  \item The {\em reduction} of $\Q$-Weil divisor $D = \sum_{i=1}^m r_i D_i$,
    where the $D_i$ are distinct prime divisors (on some normal scheme),
    is $D_{red} = \sum_{i=1}^m D_i$. 

  \item By an {\em $\shO_X$-module} we mean a quasi-coherent 
    sheaf of $\shO_X$-modules. An {\em ideal of $\shO_X$} is a quasi-coherent
    $\shO_X$-submodule of $\shO_X$. 

  \item If $I \subset \shO_X$ is an ideal on $X$, we denote by 
    $\V(I)$ the closed subscheme of $X$ with structure sheaf $\shO_X/I$. 

  \item {\em Neighborhoods} of a point in a topological space are 
    open neighborhoods. 

  \item Valuation rings $V$ are all rank-one, meaning the Krull dimension
    of $V$ is one. Equivalently, the value group $\Frac(V)^\times/V^\times$ 
    can taken to be a subgroup of $(\R, +)$) \cite{Bourbaki1998}. 
    Valuations take value $+\infty$ on $0$. 

  \item If $X$ is a scheme and $Z \subseteq X$ is an integral 
    subscheme with generic point $x$, we denote by 
    $\shO_{X, Z}$ or $\shO_{X,x}$ the local ring at $x$,
    and $\kappa(Z)$ or $\kappa(x)$ the residue field $\shO_{X,x}/\p_x$. 

  \item For local sections $f \in \shO_{X, x}$, 
    we write $f(x)$ for the residue of $f$ in $\kappa(x)$. 
    \label{f(x)}

  \item If $\shF$ and $\shG$ are two $\shO_X$-modules, we 
    denote by $\shHom_X(\shF, \shG)$ the $\shO_X$-module 
    $(U \mapsto \Hom_U(\shF|_U, \shG|_U))$. 

  \item If $X$ is a scheme of characteristic $p > 0$, and
    $\shI \subseteq \shO_{X}$ is an ideal, then
    we define the ideal $\shI\fpow{e}$ generated 
    by $p^e$-th powers of sections of $\shI$. 

\end{enumerate}

\subsection{Arithmetic on the extended real line}
We will often need the standard extension of arithmetic 
operations on $\R$ to $\R_{\pm \infty} = \R \cup \{+\infty, -\infty\}$, 
where $+\infty$ and $-\infty$ satisfy $-\infty < r < +\infty$ 
for every $r \in \R$. The following expressions are \textbf{undefined}: 
\[ (\pm \infty) + (\mp \infty), (\pm \infty) - (\pm \infty), 
0 \cdot (\pm \infty), (\pm \infty) \cdot 0. \]
We otherwise set
\begin{enumerate}
  \item $r \cdot (\pm \infty) = \pm \infty$ if $r > 0$, 
  \item $r \cdot (\pm \infty) = \mp \infty$ if $r < 0$, 
  \item $r + (\pm \infty) = \pm \infty$.
\end{enumerate}
Many technicalities in our definitions
in sections \ref{Section: Log discrepancy} 
and \ref{Section: LSC} arise from the need 
to avoid the undefined expressions above. 

\subsection{SNC Divisors}
A core concept is that of an {\em snc divisor} on a regular
scheme, short for {\em simple normal crossings (support)}, 
a global version of partial regular
systems of parameters. Such divisors play a key role
in the definition of quasi-monomial valuations and the
construction of {\em retraction morphisms} used to 
define log discrepancies of arbitrary valuations. 

\begin{definition}[SNC divisor]
  Let $Y$ be a regular scheme, and let $D_1, \dots, D_N$
  be prime divisors on $Y$. We say the $\Q$-divisor $D = \sum_1^N r_iD_i$,
  $r_i \in \Q$, is an {\em snc divisor} if:
  \begin{enumerate}
    \item Each $D_i$ is a regular scheme. 
    \item For each $y \in Y$, let $D_{i_1}, \dots, D_{i_s}$
      be those $D_i$ that contain $y$. Suppose $D_{i_j}$ corresponds
      to the prime $(f_{i_j}) \subset \shO_{Y,y}$. Then 
      $f_{i_1}, \dots, f_{i_s}$ form part of a regular system 
      of parameters for $\shO_{Y,y}$. 
  \end{enumerate}
\end{definition}

\subsection{Log resolutions and geometric log discrepancies}

We briefly review the definition of log resolutions, 
and log discrepancies in the classical sense,
which we will call {\em geometric} log discrepancies to distinguish
them {\em a priori} from the ones defined in \cref{Section: Log discrepancy}. 
In the more general setting of $F$-finite schemes, we prefer
the language of Cartier subalgebras to that of pairs.  

For this subsection, let $k$ be an algebraically closed field. 
We start with several standard definitions in birational geometry,
taking \cite[Notation 0.4]{KollarMori} as our primary source (see
also \cite[Ch. 9]{LazarsfeldPositivity2} for a simpler setting). 

\begin{definition}[Pair, log $\Q$-Gorenstein]
  A {\em pair} $(X, \Delta)$ is the data of a normal variety $X$
  over $k$ and an effective $\Q$-Weil divisor $\Delta$ on $X$.
  Fixing a canonical class $K_X$, we say $(X, \Delta)$ is 
  {\em log $\Q$-Gorenstein} if there exists
  an integer $m > 0$ such that $m(K_X + \Delta)$ is an 
  integral Cartier divisor. The least 
  such $m$ is the {\em Cartier index} of $K_X + \Delta$. 
\end{definition}

\begin{definition}[Strict transform]
  Let $(X, \Delta)$ be a pair, and $\pi: Y \to X$ a proper
  birational morphism from a normal variety $Y$ that is a isomorphism
  between open subsets $U \subseteq X$ and $V \subseteq Y$, where 
  $X \setminus U$ has codimension at least 2. Suppose 
  $\Delta = \sum_i r_i \Delta_i$ 
  with $\Delta_i$ prime divisors on $X$. The {\em strict transform} of 
  $\Delta$ on $Y$, denoted here $\pi_*\inv(\Delta)$, is the divisor 
  $\sum_i r_i \widetilde{\Delta}_i$ whose prime components
  $\widetilde{\Delta}_i$ are the (topological) closures
  of $\Delta_i \cap U$ in $Y$, identifying $U$ and $V$ via $\pi$. 
\end{definition}

\begin{definition}[Log resolution]
  Let $(X, \Delta)$ be a pair.  
  A {\em log resolution} of $(X, \Delta)$ is a proper 
  birational morphism $\pi: Y \to X$, from a smooth variety $Y$,
  whose exceptional locus $E \subset Y$ is a divisor, and 
  $E + \pi_*\inv(\Delta)_{red}$ is snc. 
\end{definition}

\begin{definition}[Geometric log discrepancy]
  Let $(X, \Delta)$ be a log $\Q$-Gorenstein pair on
  a normal variety over $k$ with canonical class $K_X$. 
  Suppose $Y$ is a normal variety with a proper birational
  morphism $\pi: Y \to X$, let $E \subset Y$ be a prime divisor,
  and let $K_Y$ be the canonical class on $Y$ with $\pi_*K_Y = K_X$. 
  Define a $\Q$-Weil divisor $\Delta_Y$ on $Y$ via
  \[ K_Y + \Delta_Y = \pi^*(K_X + \Delta). \]
  The {\em geometric log discrepancy of $(X, \Delta)$ on $E$}
  is 
  \[ A_{(X, \Delta)}(E) = 1 - \ord_E(\Delta_Y). \]
  Here, by $\ord_E(\Delta_Y)$ we mean the coefficient
  on $E$ in $\Delta_Y$. 
\end{definition}

It is well-known that $A_{(X, \Delta)}$ depends
only on the valuation $\ord_E$ on the fraction
field $k(X)$ and not the variety $Y$ on which we have
realized $E$ as a prime divisor, see e.g. \cite{KollarMori}.
In characteristic $p > 0$, this will also follow from
Cascini, Musta\c{t}\u{a}, and Schwede's result, \cref{CMS1}. 

\subsection{Cartier subalgebras, sharp F-purity, and strong F-regularity}
Our central organizational objects for birational geometry in 
positive characteristics
are the {\em Cartier subalgebras} defined by Schwede 
\cite{SchwedeTestIdealsInNonQGor} and Blickle 
\cite{BlickleTestIdealsViaAlgebras}. These can be thought of
as a positive characteristic variant of pairs or triples 
$(X, \Delta, \ideala_\star^t)$, though they can be much more general. 
These are the primary objects of study from 
\Cref{Section: Log discrepancy} onwards.

We follow the down-to-earth approach of $p^{-e}$-linear maps 
in lieu of the systematic use of pushforwards of sheaves under the 
Frobenius morphism. For this subsection, $X$ is an integral scheme 
of characteristic $p > 0$ with function field $L$. The {\em Frobenius 
morphism} on $X$ is denoted $F: X \to X$; this is the identity
on the underlying topological space of $X$, and is the $p$-th power
morphism on the structure sheaf. Some authors call this the absolute
Frobenius morphism. We also have the $e$-th iterated Frobenius
$F^e: X \to X$, i.e. the $p^e$-th power morphism on $\shO_X$. 

One of the starting points for studying singularities
of rings in characteristic $p > 0$ is Kunz's celebrated
1969 result.

\begin{theorem}[\cite{Kunz1969}]
   A Noetherian ring $R$ of characteristic $p > 0$
   is regular if and only if it is flat over the
   subring $R^p = \{f^p \,:\, f \in R\}$.
\end{theorem}

Notable examples include polynomial rings $k[x_1, \dots, x_n]$
and power series rings $k\ldb x_1, \dots, x_n \rdb$, where $k$
is a field with $[k:k^p] < \infty$, both
of which are in fact {\em free} over their rings of $p$-th
powers. For example, if $k$ is perfect, then a basis
for either of these rings over $R^p$ is given by
\[ \{x_1^{a_1} \cdots x_n^{a_n} \,:\, 0 \le a_i \le p - 1
\,\,\text{ for all $1 \le i \le n$}\}. \]
Based on Kunz's theorem, one studies and characterizes
singularities of characteristic $p$ rings and schemes
in terms of the existence of free $R^p$-summands
\[ R = R^p f \oplus M \]
where $M$ is an $R^p$-submodule of $R$. A summand
is equivalent to an $R^p$-linear mapping
$\phi: R \to R^p \subset R$. If $R$ is reduced,
then $R^p \isom R$ as rings, sending $r^p$ to $r$,
so we can view $\phi$ as an additive map
$\phi: R \to R$ with the additional property
$\phi(a^p g) = a\,\phi(g)$ for all $a, g \in R$. 
Such $\phi$ are called $p\inv$-linear maps,
see \eqref{p-e linear}.

\begin{definition}
  We say $X$ is {\em $F$-finite} if 
  $F^e$ is a finite morphism for 
  for some, equivalently any, $e \ge 1$. 
  When $X = \Spec(R)$ is affine, we will
  say $R$ is $F$-finite.  
\end{definition}

\begin{definition}\label{p-e linear}
  A {\em $p^{-e}$-linear map on $L$} is an additive
  function $\psi: L \to L$ with the property $\psi(f^{p^e} g) = f \psi(g)$
  for all $f, g \in L$. If $\psi$ is a 
  $p^{-e}$-linear map with $\psi(\shO_X) \subseteq \shO_X$,
  then we say $\psi$ is a {\em $p^{-e}$-linear map on $X$.}
\end{definition}

\begin{remark}\label{other perspectives}
  We have chosen one of several approaches to $p^{-e}$-linear
  maps; let us briefly discuss the two other approaches,
  both for continuity with other literature cited
  here, and because we need a few consequences of
  these alternative descriptions. 
  \begin{enumerate}
    \item Writing $F^e: L \to L$ for the Frobenius
      morphism, a $p^{-e}$-linear map $\psi$
      is none other than an $L$-linear map
      $F^e_*L \to L$. Similarly, a $p^{-e}$-linear
      map on $X$ is an $\shO_X$-linear morphism
      $F^e_*\shO_X \to \shO_X$. 

    \item If we view the Frobenius as the {\em inclusion}
      $L \subseteq L^{1/p^e}$, taking $L^{1/p^e}$ to
      be the ring of $p^{e}$-th roots of elements of $L$
      in some fixed algebraic closure of $L$, then
      a $p^{-e}$-linear map becomes an $L$-linear
      mapping $L^{1/p^e} \to L$. Taking $\shO_X^{1/p^e}$
      to be the integral closure of $\shO_X$ in the
      constant sheaf on $X$ associated to $L^{1/p^e}$,
      $p^{-e}$-linear maps on $X$ are $\shO_X$-linear
      morphisms $\shO_X^{1/p^e} \to \shO_X$. 
  \end{enumerate}
  The consequences 
  we will need occasionally are the following. 
  \begin{enumerate}
    \item Supposing $X$ is $F$-finite, the sheaf of $p^{-e}$-linear
      morphisms $\shHom_X(F^e_*\shO_X, \shO_X)$ is just
      $(F^e)^!\shO_X$ \cite[Exercise III.6.10]{Hartshorne}. 
      Consider an affine open subset $\Spec(R) \subseteq X$, 
      $\p \in \Spec(R)$, and let $\Spec(R') \to \Spec(R)$ be 
      the completion morphism at $\p$. Then 
      \begin{equation}\label{completion isomorphism}
        R' \tensor_R \Hom_R(F^e_*R, R) \isom \Hom_{R'}(F^e_*R', R') 
      \end{equation}
      since $R'$ is faithfully flat and $F^e_*R$ is finitely
      generated over $R$. \label{completing Cartier algebra}. 

    \item The dual $\Hom_L(L^{1/p^e}, L)$ has the structure
      of an $L^{1/p^e}$-vector space. If $L$ is an $F$-finite
      field, then $\dim_L \Hom_L(L^{1/p^e}, L) = \dim_L(L^{1/p^e})$.
      This implies $\Hom_L(L^{1/p^e}, L)$ is one-dimensional over
      $L^{1/p^e}$. In particular, given two non-zero maps
      $\phi, \psi \in \Hom_L(L^{1/p^e}, L)$ there exists a unique
      $h \in L$ such that $\phi = \psi h^{1/p^e}$. The right hand
      side of this equality will be written $\psi \cdot h$ in this article,
      cf. \eqref{module and multiplication}. \label{L vector space}
  \end{enumerate}
\end{remark}

To shorten notation throughout the article, we 
introduce the notation $(F^e)^!\shO_X$ for the 
collection of $p^{-e}$-linear maps on $X$. More
generally, we write $(F^e)^!\shL$ for the collection
of $p^{-e}$-linear maps $\shL \to \shO_X$, 
where $\shL$ is a rank-one reflexive sheaf on $X$.
The notation is chosen to accord with Grothendieck 
duality in the $F$-finite case (cf. \cite[Exercise III.6.10]{Hartshorne}). 
We develop what we can without an $F$-finite assumption, 
so the endofunctor $(F^e)^!$ on the derived category
of $X$ may not even be defined, and the notation $(F^e)^!$ 
is purely formal. If we need to emphasize the scheme $X$, 
as we do in \S\ref{Section: LCT}, we write $(F_X^e)^!\shL$. 

\begin{definition}
  The {\em Cartier algebra of $X$} is the graded sheaf
  \[ \shC^X = \oplus_{e \ge 0} (F^e)^!\shO_X. \]
  We will often use $\shC^X_e$ as a synonym for $(F^e)^!\shO_X$. 
\end{definition}

\begin{definition}
  \label{module and multiplication}
  The Cartier algebra of $X$ is a sheaf of
  non-commutative graded $\bF_p$-algebras on $X$ via composition. 
  Indeed, if $\psi_i$ is $p^{-e_i}$-linear, $i = 1, 2$, then 
  $\psi_1 \circ \psi_2$ is $p^{-(e_1 + e_2)}$ linear: 
  \[ (\psi_1 \circ \psi_2)(f^{p^{e_1 + e_2}}g) = \psi_1(f^{p^{e_1}} \psi_2(g)) 
  = f \,\psi_1(\psi_2(g)). \]
  To emphasize that we are thinking of $\psi_1 \circ \psi_2$ as
  the result of a graded multiplication (i.e.  in $\shC^X_{e_1 + e_2}$), 
  we write $\psi_1 \cdot \psi_2$, or $\psi^2$ if $\psi_1 = \psi = \psi_2$. 

  As a special case of this, where one $e_i = 0$, we see
  that each $\shC^X_e$ has {\em two distinct} structures
  of an $\shO_X$-module. Working affine-locally on $\Spec(R) \subseteq X$,
  suppose $\psi: R \to R$ is $p^{-e}$-linear, and let $f \in R$. 
  Then $f$ is $p^0$-linear (i.e. $R$-linear), so both
  $f \cdot \psi$ and $\psi \cdot f$ are $p^{-e}$-linear. 
  They are, however, not equal: 
  \[ (f \cdot \psi)(1) = f\,\psi(1) = \psi(f^{p^e}) \]
  which is not $\psi(f) = (\psi \cdot f)(1)$, generally. 
\end{definition}

\begin{definition}
  A {\em Cartier subalgebra on $X$} is a quasi-coherent sheaf of 
  graded subrings $\shD = \oplus_{e \ge 0} \shD_e \subseteq \shC^X$.
  We ask that $\shD_0 = \shO_X = \shC^X_0$, which implies
  $\shD_e \subseteq \shC^X_e$ is an $\shO_X$-submodule
  under both module structures, cf. \cref{module and multiplication}. 
\end{definition}

\begin{definition}\label{Cartier subalgebra of map}
  Given $0 \ne \psi \in \Gamma(U, \shC^X_e)$, there is an associated
  Cartier subalgebra $\lbrak \psi \rbrak$ on $U$. This is 
  non-zero only in degrees $ne$, $n \ge 0$, and is the $\shO_U$-module
  $\psi^n \cdot \shO_U$ in degree $ne$.
\end{definition}

\begin{definition}
  [{cf. \cite{HochsterHunekeTightClosureAndStrongFRegularity, HaraWatanabe,SchwedeSharpTestElements}}]
  Suppose $X$ is an $F$-finite integral scheme, and
  $\shD$ is a Cartier subalgebra on $X$. 
  \begin{enumerate}
    \item We say $\shD$ is {\em sharply $F$-pure} at
      $x \in X$ if there exists $\psi \in (\shD_e)_x$
      that is surjective, as a function $\shO_{X,x} \to \shO_{X, x}$,
      for some $e \ge 1$. If $\shC^X$ is sharply $F$-pure at $x$,
      then we say {\em $X$ is $F$-pure at $x$}. 

    \item We say $\shD$ is {\em strongly $F$-regular} at
      $x \in X$ if the following condition is satisfied. 
      For all $f \in \shO_{X,x}$ non-zero, there exists
      $e \ge 1$ and $\psi \in (\shD_e)_x$ such that
      $\psi(f) = 1$. If $\shC^X$ is strongly $F$-regular
      at $x$, we say {\em $X$ is $F$-regular at $x$}. 
  \end{enumerate}
  If $X$ is $F$-pure (or $F$-regular) at every
  $x \in X$, then we say $X$ is $F$-pure (resp., $F$-regular). 
\end{definition}

\begin{remark}
  Following a growing consenus among experts, we drop
  the adjectives ``sharply'' and ``strongly'' from 
  $F$-purity and $F$-regularity of $F$-finite schemes.  

\end{remark}

\begin{remark}
  There is a well-known, heuristic, correspondence between
  $F$-pure and log canonical varieties, and $F$-regular and 
  klt varieties. See, e.g., \cite{HaraWatanabe} and 
  \cref{characterization of ShFP and SFR}. 
\end{remark}

\subsection{Divisors and $p^{-e}$-linear maps}
\label{Subsection: divisors}
We now recall the close connection between 
$p^{-e}$-linear maps and certain divisors on a 
normal variety $X$. The ideas present here go back at least
to Mehta-Ramanathan 
\cite{MehtaRamanathanFrobeniusSplittingAndCohomologyVanishing}
and Ramanan-Ramanathan \cite{RamananRamanathanProjectiveNormality},
though the most direct origin of the technique is Hara and Watanabe's
\cite[Lemma 3.4]{HaraWatanabe}. We refer the reader also to 
\cite[Ch. 1]{BrionKumar} and the excellent surveys 
\cite{SchwedeTuckerTestIdealSurvey, 
BlickleSchwedeSurveyPMinusE, 
PatakfalviSchwedeTucker}. 

Let $X$ be a normal variety of dimension $n$ over an algebraically 
closed field $k$ of characteristic $p > 0$; this implies
$X$ is $F$-finite. Let $\omega_X$ be the canonical bundle on $X$, 
the rank-one reflexive sheaf agreeing with $\wedge^n \Omega_{X/k}$
on the smooth locus of $X$. Fix a line bundle $\shL$ on $X$. 
Grothendieck duality for the Frobenius morphism $F: X \to X$ provides an 
isomorphism of reflexive, coherent right $\shO_X$-modules
\[ (F^e)^!\shL \isom \shL\inv \otimes \omega_X^{\otimes (1-p^e)}. \]
As a consequence, a globally defined $p^{-e}$-linear
map $\shL \to \shO_X$ is equivalent, up to a unit of $\Gamma(X, \shO_X)$,
to an effective Cartier divisor $D$ with 
\begin{equation}\label{Cartier operators, divisors}
  \shO_X(D) \isom \shL\inv \otimes \omega_X^{\otimes (1-p^e)}. 
\end{equation}
Of course, $D$ is unchanged if we multiply $\psi$ (on the right)
by a unit of $\Gamma(X, \shO_X)$. 
If we set $\Delta = \frac{1}{p^e-1} D$, and choose a canonical class
$K_X$ on $X$, then the isomorphism \eqref{Cartier operators, divisors}
can be re-formulated as saying: $\Delta$ is an effective $\Q$-Weil divisor on
$X$ such that $m(K_X + \Delta)$ is Cartier for some $m$ coprime to $p$, i.e.
$(X, \Delta)$ is a log-$\Q$-Gorenstein pair, and the Cartier index of 
$K_X + \Delta$ is not divisible by $p$. The normalization by $(p^e-1)$
is useful because 
$\psi^m \in (F^{me})^! \shL^{\otimes (p^{(m-1)e} + \cdots + p^e + 1)}$ 
has the same associated divisor as $\psi$ for all $m \ge 1$. 

\begin{definition}\label{twist: divisor}
  Let $\Delta$ be an effective divisor on $X$,
  and assume $K_X + \Delta$ is $\Q$-Cartier, 
  with Cartier index not divisible by $p$. 
  We define $\shC^X \cdot \Delta$ to be
  \[ \oplus_e \, (F^e)^!\shO_X((p^e-1)\Delta) \]
  where $e \ge 0$ ranges over values for which
  $(p^e-1)(K_X + \Delta)$ is Cartier. Note that
  $\shC^X \cdot \Delta$ is a Cartier subalgebra on $X$, 
  cf. proof of \cite[Lemma 2.8]{SchwedeCentersOfFPurity}. 
  More precisely, if $\psi$ any $p^{-e}$-linear map on 
  an open subset $U \subseteq X$
  corresponding to a Cartier divisor $(1-p^e)(K_X + \Delta) \cap U$
  then $(F^e)^!\shO_X((p^e-1)\Delta)|_U = 
  \lbrak \psi \rbrak_e = \psi \cdot \shO_U$, 
  see \eqref{Cartier subalgebra of map}. 
\end{definition}

\subsection{Berkovich spaces}
Because we expect this article to be of interest to researchers
unfamiliar with constructions in non-archimedean geometry, 
we provide a summary of the Berkovich space theory as we need it. 
We use Thuillier's $\beth$-spaces (\cite{ThuillierToroidal}) 
and not the entire Berkovich analytification \cite{Berkovich90},
because they are compact and their points are more closely
tied to the birational geometry of a given (possibly non-proper) 
variety. 

We provide sketches of proofs when the ideas involved are 
necessary for later sections, deferring to \cite{Berkovich90} 
for more detailed developments of the following material. 
When terminology is already developed in birational geometry
that conflicts with Berkovich's terminology 
for these concepts, we use the birational language. 

\begin{definition}
  Let $X$ be a scheme over a field $k$. 
  \begin{enumerate}
    \item A {\em semivaluation on X} is a pair $\zeta = (w, x)$ consisting of 
      a point $x \in X$, with closure $Z = \overline{\{x\}}$, 
      and a valuation $w$ on $\kappa(x)$ that is {\em centered on $X$}, 
      meaning $w$ is non-negative on some local ring 
      $\shO_{Z, z} \subseteq \kappa(x) = \kappa(Z)$, $z \in Z$. 
      If $f \in \shO_{X, x}$, we define 
      $\ev_f(\zeta) = \zeta(f) := w(f(x)) \in [0, \infty]$. Note 
      that being centered on $X$ forces $w(u) = 0$ for every 
      $u \in k \setminus \{0\}$, i.e. $w$ restricts to the 
      {\em trivial} valuation on $k$.  

    \item We denote by $X^\beth$ the set of all semivaluations on $X$. 
      The point $x$ of $\zeta = (w, x) \in X^\beth$ is called the 
      {\em home} of $\zeta$ (on $X$), and the {\em home function}  
      $h_X: X^\beth \to X$ is $h_X(\zeta) = x$. The home map is also 
      commonly called the {\em kernel} map, e.g., in \cite{Berkovich90}. 

    \item If $X$ is integral with generic point $\eta_X$, we define 
      $\Val_X$ to be $h_X\inv(\eta_X)$, i.e. valuations on $\kappa(X)$ 
      having center on $X$. 

    \item The {\em semialuation ring} of $(w, x) \in X^\beth$
      is $w\inv[0, \infty] \subseteq \kappa(x)$, denoted throughout
      this article as $\shO_w$. Semivaluation rings of valuations
      are called valuation rings. 

    \item For any $\zeta = (w, x) \in X^\beth$, there is a unique
      morphism $i_\zeta: \Spec(\shO_w) \to X$ extending the
      natural map $\Spec(\kappa(x)) \to X$; here,
      $\shO_w = w\inv[0, \infty]$ is the semivaluation ring of $\zeta$. 
      The point $c_X(\zeta) := i_\zeta(\m_w)$ is called 
      the {\em center of $\zeta$ on $X$} (where $\m_w \subset \shO_w$ is
      the maximal ideal). In the more general setting of \cite{Berkovich90}, 
      $c_X(\zeta)$  is called the {\em reduction} of 
      $\zeta$, denoted there $\mathrm{red}(\zeta)$. 

    \item We will write $\shO_{X, c(\zeta)}$ 
      for $\shO_{X, c_X(\zeta)}$. We follow 
      a similar convention with $h_X(\zeta)$. 
  \end{enumerate}
\end{definition}

We put no finite type hypothesis on the scheme 
$X$ above, so we can take $k = \Q$ or $\bF_p$, and 
in some sense we are just using $k$
to force $X$ to have equal characteristic.  

\begin{remark}
  Integral subschemes of $X$ and $X_{\mathrm{red}}$ are the same, 
  and so $X^\beth = (X_{\mathrm{red}})^\beth$ as sets. Moreover, 
  any nilpotent section $f$ is sent to $+\infty$ by all 
  $\zeta \in X^\beth$ whose home is in an open subset where $f$ 
  is regular, so $\ev_f = \ev_0$. It will be clear from our 
  construction that $X^\beth$ and $(X_{\mathrm{red}})^\beth$ 
  agree as topological spaces, too. 
\end{remark}

We now topologize the set $X^\beth$ to define the 
{\em $\beth$-space} of $X$. 

\begin{definition}\label{definition: affine Berkovich space}
  Suppose $X = \Spec(R)$ is a reduced affine scheme. Consider 
  \begin{align*}
    \Spec(R)^\beth \into \prod_{f \in R} [0, \infty] \quad\text{defined by}\quad
        \zeta \mapsto \prod_{f \in R} \zeta(f). 
  \end{align*}
  Here $[0, \infty]$ has the topology making it homeomorphic to $[0, 1]$.
  We give $\Spec(R)^\beth$ the subspace topology via this injection. 
  One checks easily that the image of $\Spec(R)^\beth$ is a closed subspace 
  of this product space, so is compact by Tychonoff's theorem. 
  The Hausdorff property is inherited. 
\end{definition}

\begin{remark}
  It is equivalent to give the set $\Spec(R)^\beth$ the 
  weakest topology such that $\ev_f$
  is continuous for all $f \in R$. 
\end{remark}

The following lemma is well-known, see e.g. 
\cite[Corollary 2.4.2]{Berkovich90}, 
\cite[Remark 4.2]{JonssonMustata}.
We include a proof for convenience, and to illustrate the
basics of working with the center and home functions. 

\begin{lemma}
  For any affine $X = \Spec(R)$, $h_X: X^\beth \to X$ is continuous. 
  On the other hand, $c_X$ is anti-continuous, in the sense that if 
  $U \subseteq X$ is Zariski open, $c_X\inv(U) \subseteq X^\beth$ is closed. 
\end{lemma}
\begin{proof}
  We first check that $h_X\inv(\V(I))$ is closed for all 
  proper ideals $I \subset R$. If $\p = h_X(\zeta) \in \V(I)$ 
  then $I \subseteq \p$. Thus $\zeta(I) = \{+\infty\}$. 
  We conclude $h_X\inv(\V(I)) = \bigcap_{f \in I} \ev_f\inv(+\infty)$, 
  which is closed by continuity. 

  Considering $c_X\inv(\V(I))$, note that $f \in P$, $P \in \Spec(R)$ if and only if 
  $\zeta(f) > 0$ when $c_X(\zeta) = P$. Thus, 
  $c_X\inv(\V(I)) = \bigcup_{f \in I} \ev_f\inv(0, \infty]$, which
  is open in $X^\beth$. 
\end{proof}

\begin{remark}\label{compatibility of topologies}
  Any ring homomorphism $\pi: R \to S$ induces a continuous 
  map $\pi_*: \Spec(S)^\beth \to \Spec(R)^\beth$ by $\pi_*(\zeta)(f) = \zeta(\pi(f))$, 
  $f \in R$. We see that when $\pi$ gives an open immersion, the subspace 
  topology on (the compact subset) $\pi_*(\Spec(S)^\beth)$ agrees with the one defined 
  directly on $\Spec(R)^\beth$ as above. We use this observation to define
  the $\beth$-space of an arbitrary (non-affine) scheme.  
\end{remark}

\begin{definition}[$\beth$-space of a scheme]
  \label{definition: Berkovich space}
  Let $X$ be a scheme over a field $k$, and let 
  $U$ and $V$ be two affine open subschemes of $X$. 
  Since $X$ is separated, $U \cap V$ is again affine,
  and the topology on $(U \cap V)^\beth$ is identical 
  to the subspace topology induced from either $U^\beth$ 
  or $V^\beth$. Thus, there is a unique topology on 
  $X^\beth$ whose open subsets are those subsets 
  $\mathcal{U} \subset X^\beth$ such that 
  $\mathcal{U} \cap U^\beth$ is open in the topology 
  from \ref{definition: affine Berkovich space} 
  for every affine open subscheme $U \subseteq X$. 
\end{definition}

\begin{remark}
  The topology of $X^\beth$ is entirely determined 
  by the topology of $U^\beth$ as $U$ ranges over 
  affine subschemes $U \subset X$, and a finite 
  affine open cover $\{U_1, \dots, U_t\}$ of $X$
  leads to a finite cover of $X^\beth$ by the 
  \textbf{compact} subsets $U_i^\beth$. Thus, $c_X$ 
  remains anticontinuous and $h_X$ continuous for any $X$, 
  since they are so on $U^\beth$ for any affine 
  open $U \subseteq X$. 
\end{remark}

\begin{remark} 
  As a word of caution, it is eminently not true 
  that $U^\beth = h_X\inv(U)$ for open subsets $U$ 
  of $X$. Indeed: supposing that $X$ is integral 
  for simplicity, the home of every valuation 
  having center on $X$ is the generic point of 
  $X$, and so $\Val_X \subseteq h_X\inv(U)$.
  However, not every valuation having center 
  on $X$ necessarily has center on $U$. We do 
  have $U^\beth = c_X\inv(U)$ for $U \subseteq X$ 
  open, and $(X \setminus U)^\beth = h_X\inv(X \setminus U)$.
\end{remark}

\subsection{Retractions and Abhyankar valuations}
\label{Subsection: Locally quasi-monomial}
We recall some of the basic numerics of valuations,
and the construction of monomialization retractions,
in the setting of integral excellent schemes
over a field. We need this material for the proof
of our main theorem (where the more typical setting
of varieties over algebraically closed fields suffices)
and also in the final section (developed in the
general setting of strongly $F$-regular $F$-finite 
schemes). 

Let $X$ be an integral excellent scheme over a 
field $k$ with function field $L$, and $v \in \Val_X$. 
The {\em value group of $v$} is 
$\Gamma_v = v(L^\times) \subseteq \R$. The 
{\em rational rank of $v$} is 
$\ratrk(v) = \dim_\Q(\Gamma_v \otimes_\Z \Q)$. 
If $(B, \m)$ is a local subring of $L$ dominated 
by the valuation ring $(\shO_v, \m_v)$, then 
$(B/\m) =: \ell \subseteq \kappa(v) := \shO_v/\m_v$ 
and we set the {\em transcendence degree of $v$ over $B$} 
to be $\trdeg_B(v) = \trdeg(\kappa(v) \,|\, \ell)$. 
When $c_X(v) = x$, we define $\trdeg_X(v) = \trdeg_{\shO_{X, x}}(v)$. 
The fundamental estimate is {\em Abhyankar's inequality}: 
\[ \ratrk(v) + \trdeg_X(v) \le \dim(\shO_{X, x}). \]
Valuations achieving equality are called {\em Abhyankar valuations}. 
We refer the reader to \cite[Th\'eor\`eme 9.2]{Vaquie} for
this result in our generality here. 

\begin{definition}\label{monomial valuation}
  One easy way to obtain an Abhyankar valuation is via the 
  following construction, outlined more carefully in \S 3.1 
  of \cite{JonssonMustata}. Let $(R, \m, \kappa)$ be a 
  regular local ring. For any regular system of parameters 
  $r_1, \dots, r_d$ for $R$, there is an isomorphism
  $\hat{R} \isom \kappa \ldb r_1, \dots, r_d \rdb$, and so 
  we may view $f \in R \subseteq \hat{R}$ as having an 
  expansion of the form $f = \sum_{u \in \N_0^d} c_u r^u$, 
  where $r^{(u_1, \dots, u_d)} = r_1^{u_1} \cdots r_d^{u_d}$ and $c_u \in \kappa$.
  Jonsson and Musta\c{t}\u{a} prove in Proposition 3.1 of 
  \cite{JonssonMustata} that for any 
  $\alpha = (\alpha_1, \dots, \alpha_d) \in \R_{\ge 0}^d$, 
  there is a unique valuation $\val_\alpha$ on $\hat{R}$ 
  such that $\val_\alpha(f) = 
  \min \{ \sum_{j=1}^d \alpha_j u_j \,:\, c_u \ne 0\}$; such 
  valuations are sometimes called {\em Gauss} or {\em monomial} 
  valuations. Restricting $\val_\alpha$ to $R$ now gives an 
  Abhyankar valuation $v$;
  After blowing-up appropriately, there exists a regular local 
  ring $(R', \m')$, dominating and birational to $(R, \m)$, with
  Krull dimension 
  $\dim_{\Q}(\Q \alpha_1 + \cdots + \Q \alpha_d) = \ratrk(\val_\alpha)$, 
  on which $v$ is centered \cite[Proposition 3.6(ii)]{JonssonMustata}. 
  Moreover, the residue field $\kappa(v)$
  must be an algebraic extension of $R'/\m'$, thanks to the {\em dimension
  formula} \cite[Theorem 15.2]{MatsumuraCommutativeRingTheory} 
  (using that $X$ is excellent); 
  see the discussion at the beginning of \cite[\S 3.2]{JonssonMustata}. 
\end{definition}

\begin{definition}[Monomialization retraction, {cf. \cite{JonssonMustata}}]
  \label{definition: retraction}
  Fix an snc divisor $D$ on a regular excellent scheme 
  $Y$ with a proper birational morphism $\pi: Y \to X$. 
  Let $v \in \Val_X$ with $c_Y(v) = y$, 
  and suppose $D \cap \Spec(\shO_{Y,y}) = \div(z_1 \cdots z_t)$,
  with $z_i \in \m_y$; extend these to a generating set $z_1, \dots, z_d$ for 
  $\m_y$. The $\m_y$-adic completion of $\shO_{Y,y}$ is 
  isomorphic to $\hat R := \kappa(y)\ldb z_1, \dots, z_d \rdb$ by Cohen's structure 
  theorem, and one has an inclusion $\iota: \shO_{Y,y} \into \hat R$ sending
  $z_i$ to $z_i$. On $\hat R$, we have the monomial valuation
  $w$ with $w(z_i) = v(D_i)$ for $1 \le i \le t$ and $w(z_i) = 0$
  for $t < i \le d$. We define $r_{(Y, D)}(v) := \iota_*(w) \in \Val_X$. 
\end{definition}

\begin{remark}
  Note that, by construction and super-additivity of valuations,
  $v(f) \ge r_{(Y, D)}(v)(f)$ for all $f \in \shO_{Y,y}$.
  Consequently, $c_Y(v) \in \overline{\{c_Y(r_{(Y, D)}(v))\}}$. 
\end{remark}

For regular excellent schemes over $\Q$, every Abhyankar valuation 
is obtained as the result of some monomialization retraction, 
induced from a proper birational $Y \to X$ and some SNC 
divisor $H$ on $Y$ (see \cite{EinLazarsfeldSmithSymbolic} and 
Proposition 3.7 in \cite{JonssonMustata}). If $X$
is a variety and $k$ has characteristic $p > 0$, Knaf and 
Kuhlmann show in \cite{KnafKuhlmann} that an Abhyankar valuation 
whose residue field is separable over $k$ admits a 
{\em local monomialization} over $X$, meaning
that if $v$ is an Abhyankar valuation centered at $x \in X$,
then there exists an open neighborhood $U \subset X$ of $x$, 
a proper birational morphism $\pi: Y \to U$ from a regular 
variety $Y$, and an snc divisor $D$ on $Y$, such 
that $v = r_{(Y, D)}(v)$.


%% file: zLog-discrepancies-positive-characteristic.tex
\section{Log discrepancies of Cartier subalgebras}
\label{Section: Log discrepancy}
This is the foundational section of this article; the subsequent 
sections are applications of the ideas developed here. 

Let $X$ be an integral scheme of prime characteristic 
$p > 0$. We denote the function field of $X$ by $L$ 
thoughout this section. Recall from \cref{Subsection: divisors} 
that when $X$ is a normal variety, there is essentially a bijection between
log $\Q$-Gorenstein pairs $(X, \Delta)$ whose
log Cartier index is not divisible by $p$, and
$p^{-e}$-liner maps $\psi$ on $X$, furnished by
the isomorphism $(F^e)^!\shO_X((1-p^e)\Delta) \isom \lbrak \psi \rbrak$. 

Our primary goal is to define the log discrepancy of 
Cartier subalgebras on $X$, with the fundamental definition
being that of log discrepancies of $p^{-e}$-linear maps on
{\em the function field $L$ of $X$}. Much of this section goes through
without an $F$-finite assumption; because this may be of interest,
e.g. for varieties over non-$F$-finite fields, we do what we can 
without assuming $X$ is $F$-finite. Of course, for many of our more 
precise results, we need this assumption. 

%
%
%

\begin{remark}
  In previous versions of this paper, we developed these ideas
  in the setting of Cartier {\em pre} algebras. This obfuscated
  the main points, adding lots of technicalities unnecessary for
  the most interesting setting of Cartier subalgebras. We have
  thus revised this section to focus on Cartier subalgebras. 
\end{remark}

\begin{definition}
  Let $\shD$ be a Cartier subalgebra on $X$, and $Z \subseteq X$ a 
  closed integral subscheme with ideal $\shI_Z \subseteq \shO_X$. 
  We say that $Z$ is {\em uniformly $\shD$-compatible} if
  $\psi(\shI_{Z,x}) \subseteq \shI_{Z,x}$ for 
  each $x \in X$, $e \ge 1$, and $\psi \in (\shD_e)_x$.
  In this case, each $\psi$ induces a well-defined $p^{-e}$-linear 
  map on $Z$, which we denote by $\psi\|_Z$. The collection of all 
  $\psi\|_Z$ give a Cartier subalgebra $\shD\|_Z$ on $Z$, which we call
  the {\em exceptional restriction}. The usual restriction
  $\oplus_{e \ge 0} (\shO_Z \otimes_X \shD_e)$ is not useful for us.
\end{definition}

\subsection{The definition of log discrepancy}\label{definition}
\begin{definition}[{Log discrepancy of $p^{-e}$-linear maps 
  and Cartier subalgebras}]
  \label{definition: log discrepancy function}
  We define the log discrepancy of $p^{-e}$-linear maps and Cartier 
  subalgebras at a semivaluation in four steps. 
    \begin{enumerate}
      \item Suppose $\psi$ is a $p^{-e}$-linear map on $L$ for some $e \ge 1$, 
        let $0 \ne f \in L$, and $v \in \Val_X$. We define 
        \begin{equation}\label{definition: E}
          E(f, \psi, v) = \frac{v(f) - p^e v(\psi(f))}{p^e-1} 
          \in [-\infty, \infty). 
        \end{equation}
        Note that if $\psi(f)$ is a unit of the 
        valuation ring $\shO_v$, then $E(f, \psi, v) = (p^e - 1)\inv v(f)$. 
        The notation is meant to suggest the coefficient of 
        $K_Y - \pi^*(K_X + \Delta)$ on divisors $E \subset Y$ 
        over a variety $X$, when $\Delta$
        corresponds to $\psi$ via \eqref{Subsection: divisors}.

        If $g := \psi(f)$ is nonzero, then 
        $\psi(fg^{-p^e}) = g\inv \psi(f) = 1$ and 
        $v(fg^{-p^e}) = v(f) - p^ev(g) = v(f) - p^e v(\psi(f))$.
        Thus, given any $f \in L$ with $\psi(f) \ne 0$, 
        there is some other 
        $h \in L$ with $E(h, \psi, v) = E(f, \psi, v)$ and $\psi(h) = 1$. 
        On the other hand, \ref{definition: E} allows us to consider
        the full range of values of $E(f, \psi, v)$, as $f$ ranges over nonzero 
        $f \in L$, by restricting to any subring $R \subseteq L$ with 
        $L = \Frac(R)$. 

      \item The {\em log discrepancy} of $0 \ne \psi: L \to L$ at 
        $v \in \Val_X$ is
        \[ A(v; \psi) = \sup_{f \ne 0, n \ge 1} E(f, \psi^n, v). \]
        Note that, per the comment at the end of the previous item, we can 
        (and often do) restrict this supremum to those $f \ne 0$ with 
        $\psi^n(f) = 1$, or $E(f; \psi^n, v)$ for $f \in R$ for
        some fixed ring $R$ with $\Frac(R) = L$, e.g. $\shO_{X,x}$ for $x \in X$,
        or valuation rings $\shO_v$. 

      \item Let $\shD$ be a Cartier subalgebra on $X$. 
        The {\em log discrepancy} of $\shD$ at $v \in \Val_X$ 
        centered at $x$ is
        \[ A(v; \shD) = \sup_{e \ge 1} \left( \sup_{0 \ne \psi \in (\shD_e)_x } 
        A(v; \psi) \right). \]

      \item When $\zeta \in X^\beth \setminus \Val_X$, we define 
        $E(f, \psi, \zeta)$ for nonzero $p^{-e}$-linear maps $\psi$,
        when $\zeta(f)$ and $\zeta(\psi(f))$ are not both $+\infty$. 

      \item Let $x \in X$, with closure $Z = \overline{\{x\}}$ that
        is uniformly compatible with $\shD$. We define
        \[ A(\zeta; \shD) = A(\zeta; \shD\|_Z) \]
        for every $\zeta \in \Val_Z = h_X\inv(x) \subseteq X^\beth$. 
        For $\zeta \in X^\beth$ whose home is not uniformly $\shD$-compatible, 
        we set $A(\zeta; \shD) = +\infty$. 
    \end{enumerate}
\end{definition}

We immediately prove several very useful 
ways to simplify the calculation of log discrepancies. 
\Cref{effect of right multiplication}
presents an important proof method we build on 
several times in this section; this was inspired by 
(and our proof is a solution to) 
\cite[Exercise 4.11]{BlickleSchwedeSurveyPMinusE}. 
\Cref{finite type} is used constantly throughout the paper. 

\begin{proposition}\label{effect of right multiplication}
  Let $v \in \Val_X$, $\psi: L \to L$, and $f \in L$, where both $\psi$ 
  and $f$ are nonzero. Suppose $A(v; \psi) < \infty$. Then
  \[ A(v; \psi \cdot f) + (p^e-1)\inv v(f) = A(v; \psi). \]
\end{proposition}
\noindent Here, $\psi \cdot f$ denotes
the product in the Cartier algebra of $\Spec(L)$ 
\eqref{module and multiplication}. 
\begin{proof}
  Define $f_n = f^{(p^{ne} - 1)/(p^e-1)}$ for all $n \ge 1$. 
  The following observations are easy to check:
  \begin{enumerate}
    \item $(\psi \cdot f)^n = \psi^n \cdot f_n$. 
    \item $\psi^n(h) = 1$ if and only if 
      \[ (\psi \cdot f)^n \left( f_n\inv h\right) = 1. \]
    \item $E(h, \psi^n, v) = E(f_n\inv h, \psi^n, v) + (p^e-1)\inv v(f)$. 
  \end{enumerate}
  Thus, there is a bijection between $h \in L$ with $\psi^n(h) = 1$ 
  and $g \in L$ with $(\psi \cdot f)^n(g) = 1$ given by multiplication
  by $f_n$. Considering then the definitions of $A(v; \psi \cdot f)$
  and $A(v; \psi)$, and applying the third observation, gives the claimed
  expression for $A(v; \psi \cdot f)$. 
\end{proof}

\begin{proposition}\label{finite type}
  Let $\shD$ be a Cartier subalgebra on $X$ and $x \in X$. Suppose 
  that $\shD_x = \lbrak \psi \rbrak_x \subset \shC^X_x$ 
  for some $\psi \in (\shD_e)_x$, cf. \eqref{Cartier subalgebra of map}. 
  Then $A(v; \shD) = A(v; \psi)$ for every $v \in \Val_X$ with $c_X(v) = x$. 
\end{proposition}
\begin{proof}
  By definition, $A(v; \shD) \ge A(v; \psi)$.  
  On the other hand, the assumption 
  $\shD_x = \lbrak \psi \rbrak_x$ 
  implies $(\shD_m)_x = 0$ unless $m = ne$, and any $\phi \in (\shD_{ne})_x$ 
  can be written as $\psi^n \cdot f$ for some 
  $f \in \shO_{X,x}$. Now \eqref{effect of right multiplication} shows 
  $A(v; \phi) = A(v; \psi) + (1- p^e)\inv v(f)$.
  Thus, $A(v; \phi) \le A(v; \psi)$ since $c_X(v) = x$, which implies 
  $v(f) \ge 0$. Therefore, $A(v; \shD) = A(v; \psi)$. 
\end{proof}

We now prove some easy consequences of \eqref{effect of right multiplication}, 
inspired by \cite[Exercise 4.12]{BlickleSchwedeSurveyPMinusE}, 
cf. \cite[Lemma 4.9(i)]{SchwedeSmithLogFanoVsGloballyFRegular}. 
The first, we attribute to Cascini, Musta\c{t}\u{a}, and Schwede,
since the key claim in the middle of the proof was shared with the
author by Karl Schwede in private correspondence. 
\Cref{log discrepancy of QM} generalizes this result. 

\begin{corollary}[{Cascini-Musta\c{t}\u{a}-Schwede}]
  \label{CMS2}
  Let $v$ be a discrete valuation on $L$ 
  whose associated valuation ring $R$ is 
  $F$-finite, and let $\varpi \in R$
  be any generator for the maximal ideal. 
  Let $\shC^R$ be the Cartier algebra of $\Spec(R)$. 
  Then $A(v; \shC^R) = v(\varpi)$. 
\end{corollary}
\begin{proof}
  It is well-known that $R$ is a free $R^{p^e}$-module 
  (of rank $p^{e[L:L^p]}$) with a basis containing $\varpi^{(p^e-1)}$, 
  and that $\shC^R_e = \Phi^e \cdot R$, where
  $\Phi^e$ is the $p^{-e}$-linear projection $\Phi: R \to R$ 
  onto $\varpi^{(p^e-1)}$. \Cref{finite type} proves
  $A(v; \shC^R) = A(v; \Phi)$. 

  Because $\Phi(\varpi^{(p-1)}) = 1$, 
  $A(v; \Phi) \ge v(\varpi) = E(\varpi, \Phi, v)$. Now suppose
  $f \in R$ has $\Phi^e(f) = 1$ and
  $E(f, \Phi^e, v) \ge E(\varpi, \Phi, v)$. Write
  $f = u \varpi^s$, with $u \in R^\times$, 
  and $s \ge p^e-1$. I claim that in fact $s = p^e-1$. More generally:

  \noindent\textbf{Claim:} Define $s' = \lceil p^{-e} (s - p^e + 1) \rceil$. 
  Then $\Phi^e(\varpi^{s}R) = \varpi^{s'}R$. 

  \begin{pfClaim}
    The claim is clear when $p^e \,|\, (s - p^e + 1)$:
    $\varpi^{(p^e-1)}$ is sent to $1$ by $\Phi^e$, and 
    $\Phi^e$ is $p^{-e}$-linear, so
    \begin{align*}
      \Phi^e(\varpi^{s}) 
      &= \Phi^e(\varpi^{s - p^e + 1 } \varpi^{(p^e-1)}) \\
      &= \varpi^{s'}. 
    \end{align*}
    More generally, for $f \in R$, we have 
    \[ v(\Phi^e(f \varpi^{-np^e}))) = v(\Phi^e(f)) - v(\varpi^n) \]
    so $\varpi^{(n+1)p^e - 1}R$ is the smallest ideal
    of $R$ sent into $\varpi^n R$. 
  \end{pfClaim}
  Finishing the proof, we see that if $\Phi(f) = 1$, then $s' = 0$.
  This is not the case unless $s = p^e-1$.



\end{proof}

\begin{corollary}
  [{cf. \cite[Lemma 4.9(i)]{SchwedeSmithLogFanoVsGloballyFRegular}}]
  Suppose $v$ is a discrete valuation on $L$ whose valuation ring $R$ is
  $F$-finite. Let $0 \ne \psi_i: L^{1/p^{e_i}} \to L$, $i = 1, 2$. 
  Suppose $\min\{e_1, e_2\} \ge 1$. Then
  \begin{equation}\label{log discrep of product}
    (1 - \e) A(v; \psi_1) + \e A(v; \psi_2) = A(v; \psi_1 \cdot \psi_2) 
  \end{equation}
  where $\e = (p^{e_2}-1)/(p^{(e_1 + e_2)} - 1)$. 
\end{corollary}
\noindent We expect that when $R$ is not $F$-finite, then 
$A(v; \psi) = +\infty$ for every $p^{-e}$-linear map. This 
is true when $X$ is a variety over a perfect field,
see \eqref{non-F-finite DVRs}. 
\begin{proof}
  Suppose $\shC^R_1 = \Phi \cdot R$ as in the last proof. 
  Considering $\Phi$, $\psi_1$, and $\psi_2$ as elements of
  $\Hom_L(L^{1/p^a}, L) \isom L^{1/p^a}$, for 
  $a \in \{1, e_1, e_2\}$ (respectively),
  then there exist $h_i \in L$ with 
  $\Phi^{e_i} \cdot h_i = \psi_i$ for $i = 1, 2$, see 
  \Cref{other perspectives}\eqref{L vector space}. 
  Now \eqref{effect of right multiplication} and \eqref{CMS2} imply
  \[ A(v; \psi_i) = v(\varpi) + \frac{1}{1 - p^{e_i}} v(h_i). \]
  Additionally, 
  \begin{align*}
    \psi_1 \cdot \psi_2 &= (\Phi^{e_1} \cdot h_1) \cdot (\Phi^{e_2} \cdot h_2) \\
    &= \Phi^{e_1 + e_2} \cdot (h_1^{p^{e_2}} h_2)
  \end{align*}
  so another application of \eqref{effect of right multiplication}
  gives the value
  \begin{align*}
    v(\varpi) - A(v; \psi_1 \cdot \psi_2) &= 
    \frac{1}{p^{(e_1 + e_2)} - 1}(p^{e_2} v(h_1) + v(h_2). \\
    &= (1 - \e)(v(\varpi) - A(v; \psi_1)) + \e(v(\varpi) - A(v; \psi_2)) \\
    &= v(\varpi) - (1-\e) A(v; \psi_1) - \e A(v; \psi_2). 
  \end{align*}
  Re-arranging the terms of these equations, we are left with 
  \eqref{log discrep of product}. 
\end{proof}

\begin{example}\label{non-F-finite DVRs}
  Datta and Smith carefully studied $p^{-e}$-linear maps on valuation rings
  inside function fields \cite{DattaSmithFrobeniusAndValuations, 
  DattaSmithFrobeniusAndValuationsCorrection}. Among other things,
  they prove that if a valuation ring in a function field is $F$-finite, 
  then the associated valuation must be an Abhyankar discrete valuation. 
  Let us re-interpret their results as statements about log discrepancies
  of non-Abhyankar discrete valuations. 

  Suppose $v$ is a non-Abhyankar discrete valuation with value group equal
  to $\Z$ on the function field $L$ of some variety $X$ over a perfect field, 
  with valuation ring $R$, and let $\varpi \in R$ with $v(\varpi) = 1$. 
  Datta and Smith proved that $(F^e)^!R = 0$ for all $e \ge 1$ 
  \cite[Corollary 4.2.2 and Lemma 4.2.4]{DattaSmithFrobeniusAndValuations},
  cf. \cite[Theorem 0.1]{DattaSmithFrobeniusAndValuationsCorrection}). 
  I claim that this implies that $A(v; \psi)= +\infty$ for any nonzero 
  $p^{-e}$-linear map $\psi: L \to L$. Indeed, the set of real numbers 
  $\{v(\psi(r)) \,:\, r \in R, \psi(r) \ne 0 \}$ must be 
  unbounded below, for if $\psi(r) \ge -M = v(\varpi^{-M})$ for 
  all $r \in R$, then $\varpi^M\psi$ gives a nonzero element of $(F^e)^!R$. 
  Therefore, the set $\{E(r, \psi, v) \,:\, r \in R\}$ is unbounded above, 
  so $A(v; \psi) = +\infty$.  
\end{example}

%
%

\begin{proposition}\label{RLR}
  Suppose $(R, \m, k)$ is a regular $F$-finite ring and $v$ is a valuation
  centered on $\m$ and monomial with respect to some regular system of parameters
  $x_1, \dots, x_n$ for $R$. Then $A(v; \shC^R) = \sum_{i=1}^n v(x_i)$. 
\end{proposition}
\begin{proof}
  Since $R$ is regular and $F$-finite, the Cohen structure theorem yields 
  an isomorphism $\hat{R} \cong k \ldb x_1, \dots, x_n \rdb$, with $k$ an 
  $F$-finite field. Recalling 
  \Cref{other perspectives}\eqref{completing Cartier algebra}, 
  and letting $\hat{v}$ be the $\m$-adic extension of $v$
  to $\hat{R}$, we see $A(v; \shC^R) = A(\hat{v}; \shC^{\hat{R}})$. Thus, we
  have reduced to the case $R = k \ldb x_1, \dots, x_n \rdb$ for an $F$-finite
  field $k$, and $v = \val_{\mathbf{r}}$ for some 
  $\mathbf{r} = (r_1, \dots, r_n) \in \R_{\ge 0}^n$, cf. 
  \eqref{monomial valuation}. Suppose $\{\lambda_i\}_{i=1}^{[k : k^p]}$ is 
  a basis for $k$ over $k^p$, with $\lambda_1 = 1$. Then $R$ is $F$-finite, 
  and is free over $R^{p^e}$ with basis
  \[ \mathcal{B} = 
  \{\lambda_i (x_1^{a_1} \cdots x_d^{a_d}) \,:\, 
  0 \le a_j \le p-1, \, 1 \le i \le [k : k^p]\}.\]

  It is well-known that $\shC^R$ is generated canonically by the projection 
  $\Phi$ onto the $(x_1 \cdots x_d)^{p-1}$-basis element
  (see, e.g., \cite[Chapter 1]{BrionKumar} or
  \cite[Example 3.0.5]{BlickleSchwedeSurveyPMinusE}). 
  The elements in $\mathcal{B}$ also give a basis for $L$ over $L^p$, so 
  given any element $f \in L$, we can uniquely write
  \[ f = \sum_{i = 1}^{[k:k^p]} \sum_{j = 1}^n \sum_{a_j = 0}^{p-1} 
          f_{(i, a_1, \dots, a_n)}^p \lambda_i x_1^{a_1} \cdots x_n^{a_n}, \]
  with $f_{(i, a_1, \dots, a_d)}^p \in L^p$. It follows that
  $\Phi(f) = f_{(1, p-1, \dots, p-1)}$. Applying \Cref{finite type}, we see that 
  \[ A(v_{\mathbf{r}}; \shC^R) = A(v_{\mathbf{r}}; \Phi) = \sum_{j = 1}^d r_j. \]
  To see the rightmost equality, we can argue as in \eqref{CMS2}: localizing $R$ at 
  $(x_i)$ gives an $F$-finite DVR, and we conclude that if $\Phi(f) = 1$ then 
  $\ord_{x_i}(f) = p-1$ for all $1 \le i \le n$. 
\end{proof}

Let us say that $v \in \Val_X$ is {\em locally quasi-monimial}
if there exists a neighborhood $U$ of $x = c_X(v)$ and a 
proper birational morphism $\pi: Y \to U$ from a regular scheme
such that $v = r_{(Y, D)}(v)$ for some snc divisor $D$ on $Y$. 

\begin{proposition}\label{log discrepancy of QM}
  Suppose $\psi$ is a $p^{-e}$-linear map on an integral $F$-finite scheme $X$. 
  Suppose $v \in \Val_X$ is locally quasi-monomial
  with center $x$, and $\pi: Y \to U$ is a proper birational
  morphism for which $v = r_{(Y, D)}(v)$. Let $\psi = \Phi_y^e \cdot h$
  where $\shC^Y_{y} = \lbrak \Phi_y \rbrak_y$, 
  and $y = c_Y(v)$. Then
  \[ 
  A(v; \psi) = v(D) + \frac{v(h)}{1 - p^e}. 
  \]
\end{proposition}
\begin{proof}
  Applying \cref{finite type}, we see $A(v; \shC^Y) = A(v; \Phi_y)$. Now applying
  \Cref{effect of right multiplication} to $\psi = \Phi_y^{e} \cdot h$,
  and \Cref{RLR} to $\shO_{Y,y}$, we see that 
  $A(v; \psi) = v(D) + (1 - p^{e})\inv v(h)$. 
\end{proof}

We will return to the previous proposition, giving more precise information about
$\div_Y(h)$ in \cref{relative canonical lemma} and \ref{retraction lemma} with
added assumptions on $X$. 

\begin{proposition}\label{no residues}
  Let $\shD$ be a Cartier subalgebra on $X$, fix $x \in X$, 
  and let $Z \subseteq X$ be a uniformly $\shD$-compatible 
  subscheme passing through $x$. 
  Denote $\shO_{X,x}$ by $R$, and let $\p \in \Spec(R)$
  be the prime corresponding to $Z$. Then for all $\psi \in \shD_x$,
  \[ 
    A(\zeta; \psi\|_Z) = \sup_{n, f} \{E(f, \psi^n, \zeta) 
    \,:\, \psi^n(f) \in R \setminus \p\}. 
  \]
\end{proposition}
\begin{proof}
  The condition $\psi^n(f) \in R \setminus \p$ is equivalent
  to $(\psi\|_Z)^n(f(\p)) \ne 0 \in \kappa(\p)$, thus 
  also equivalent to $\zeta(\psi^n(f)) < \infty$. 
  The claimed expression for $A(\zeta; \psi\|_Z)$ is then just unwinding 
  the definition of $A(\zeta; \psi\|_Z)$, recalling that $\zeta(f)$ is 
  defined to be $\zeta(f(\p))$, see Convention \ref{Conventions}\eqref{f(x)}. 
\end{proof}

\begin{proposition}\label{limsup}
  Let $\shD$ be a Cartier subalgebra on $X$ and suppose $\zeta \in X^\beth$ has
  uniformly $\shD$-compatible home. Let $x = c_X(\zeta)$, $R = \shO_{X,x}$, and 
  $\p = h_X(\zeta) \in \Spec(R)$. 
  \begin{enumerate}
    \item For $\psi \in (\shD_e)_x$,
      \[ 
      A(\zeta; \psi) = 
      \limsup_{n \to \infty} \left( \sup_f 
      \left\{ 
      E(f, \psi^n, \zeta) \,:\, f \in R, \psi^n(f) \in R \setminus \p
      \right\} \right).
      \]

    \item Similarly, we have an equality 
      \[
        A(\zeta; \shD) = 
        \limsup_{e \to \infty} \sup_{\psi \in (\shD_e) x} A(\zeta; \psi).
      \]
      \label{limsup for D}
  \end{enumerate}
\end{proposition}
\begin{proof}
  If we replace the limit supremum with supremum in the first statement, 
  then we are left with \Cref{no residues}. Thus, it suffices to show
  that $\{E(f, \psi^n, \zeta) \,:\, f\in R, \psi^n(f) \in R \setminus \p\}$
  is contained in 
  $\{E(f, \psi^{nm}, \zeta) \,:\, f\in R, \psi^{nm}(f) \in R \setminus \p\}$
  for all $m \ge 1$. The key idea here is similar to the observations 
  in the proof of \Cref{effect of right multiplication}. 
  Suppose $f \in R$ has the property $\psi^n(f) \in R \setminus \p$. 
  I claim that for each $m \ge 1$, there is some 
  $f_m \in R$ with $\psi^{nm}(f_m) \in R \setminus \p$ and 
  $E(f_m, \psi^{nm}, v) = E(f, \psi^n, v)$. 
  
  Localizing at $\p$, we are looking for $f_m \in R$ so that 
  $\psi^{nm}(f_m)$ is a unit of $R_\p$. By assumption,
  $\psi^n(f) = u$ is a unit of $R_\p$ that is contained 
  in $R$. Thus, as a $p^{-ne}$-linear map $\psi^n: L \to L$, 
  $\psi^n(u^{-p^{ne}}f) = 1$. 
  Define $g_m = (u^{-p^{ne}}f)^{(p^{nme}-1)/(p^{ne}-1)}$ and 
  $u_m = u^{p^{ne}(p^{nme}-1)/(p^{ne}-1)}$; note 
  $\psi^{nm}(g_m) = \psi^n(u^{-p^{ne}}f) = 1$ 
  for all $m \ge 1$. Moreover, the idea for proving the first observation
  in the proof of \eqref{effect of right multiplication} can be used
  to show
  \[ \psi^{nm}(u_m g_m) = u \]
  for each $m$. Set $f_m = u_m g_m \in R$. By construction, 
  $E(f_m, \psi^{nm}, v) = E(g_m, \psi^{nm}, v)$, and also 
  \[ \frac{v(g_m)}{p^{nme} - 1} = \frac{v(f) - p^{ne} v(u)}{p^{ne}-1}. \]
  The right-hand expression is precisely $E(f, \psi^n, v)$. 
  This proves that we can find the claimed $f_m$, and that
  the supremum over $n$ in the definition of $A(\zeta; \psi)$ 
  is equal to the limit supremum over $n$. 

  For the second claim, about $A(\zeta; \shD)$, note that 
  by definition (or the first part of this proposition)
  $A(v; \psi) = A(v; \psi^n)$ for all $n$. Thus, given
  $\psi \in (\shD_e)_x$ we have $\psi^n \in (\shD_{ne})_x$
  with the same log discrepancy, and $A(\zeta; \shD)$ can
  be calculated as a limit supremum over $e \ge 1$ instead
  of a supremum. 
\end{proof}

\subsection{Multiplicatively and $F$-graded sequences of ideals}
\label{graded sequences} 
We now study how Cartier subalgebras can be twisted by sequences 
of ideals with multiplicative structures, proving formulas similar 
to \cref{effect of right multiplication}. Let 
$\{\ideala_e \,:\, e \in \N_1\}$ be a sequence of non-zero ideals on $X$. 
We say that this collection is {\em multiplicatively graded} 
(or simply {\em graded}) if $\ideala_s\ideala_t \subseteq \ideala_{s + t}$ 
for all $s, t \in \N$, and that this sequence is {\em $F$-graded} 
if $\ideala_1 = \shO_X$ and $\ideala_s^{[p^t]}\ideala_t \subseteq \ideala_{s + t}$ 
for all $s, t \ge 0$. While it is standard to use the notation $\ideala_\bullet$ 
for both of these, we have need to use these two concepts together 
and so we will \textbf{always} interpret $\ideala_\bullet$ as an 
$F$-graded sequence and $\ideala_\star$ as a
multiplicatively graded sequence. 

Interesting examples of (multipicatively) graded sequences of ideals 
arise as base loci of line bundles on $X$ (\cite[Def. 1.1.18]{LazarsfeldPositivity1}),  
and symbolic powers of a fixed ideal. Another common source of 
graded sequences of ideals, especially relevant here, are those 
associated to valuations $v$ on $X$, defined by
\[ \ideala_s(v) = \{f \in \shO_X \,:\, v(f) \ge s \} \text{ for $s \in \Z_{\ge 0}$.} \]
We write $\ideala_\star(v)$ for this graded sequence.

The $F$-graded condition is precisely what is needed to make new Cartier subalgebras 
from old as described below. 
Every Cartier subalgebra on a Gorenstein scheme arises in this way, 
cf. \cite{BlickleTestIdealsViaAlgebras, BlickleSchwedeTuckerFSigPairs1}. 
 
\begin{definition}
  Let $\shD$ be a Cartier subalgebra, and $\ideala_\bullet$ an $F$-graded
  sequence, on $X$. We define a Cartier subalgebra 
  $\shD \cdot \ideala_\bullet \subset \shD$ by
  \begin{align*}
    \shD\cdot\ideala_\bullet &= \oplus_{e \ge 0} \left( \shD_e \cdot \ideala_e \right) \\
    &= \oplus_{e \ge 0} \left\{\sum (\phi_i \cdot a_i) \,:\, a_i \in \ideala_e 
    \text{ and }\phi_i \in \shD_e\right\}. 
  \end{align*}

  A special case of interest is constructed from the data of a Cartier subalgebra 
  $\shD$ on $X$, an ideal $\ideala \subseteq \shO_X$, and a real number $t \ge 0$. 
  Setting $\ideala_e = \ideala^{\lceil t(p^e-1) \rceil}$,
  we get an $F$-graded sequence $\ideala_\bullet$. We define 
  \[ \shD \cdot \ideala^t := \shD \cdot \ideala_\bullet. \]
  These $F$-graded sequences seem to first appear in the theory of tight closure with 
  respect to an ideal \cite{HaraYoshida}, see also \cite{SchwedeSharpTestElements}. 
  Their study was key to the development of sharp $F$-purity, and Schwede's approach 
  to $F$-singularities of pairs and triples 
  \cite{SchwedeSharpTestElements,
  SchwedeFAdjunction,
  SchwedeCentersOfFPurity,
  SchwedeTestIdealsInNonQGor}. 
\end{definition}

For $\zeta \in X^\beth$ with center $x$, and an ideal sheaf 
$\ideala \subseteq \shO_X$, we define
\[ \zeta(\ideala) = \min\{\zeta(f) \,:\, f \in \ideala_x \}. \]
The following lemma allows us to evaluate semivaluations on sequences 
of ideals in two ways; the existence of the limits is well-known. 

\begin{lemma}
  [{cf. \cite{EinLazMusNakPopAsymptoticInvariantsOfBaseLoci,JonssonMustata}}]
  \label{limit lemma}
  Let $\zeta \in X^\beth$ and $\{\ideala_e\}_{e \in \N_1}$ be a sequence of 
  ideals on $X$. If $\ideala_\bullet$ is $F$-graded, we have the limit
  \[ \zeta_F(\ideala_\bullet) := \lim_{e \to \infty} \frac{\zeta(\ideala_e)}{p^e-1} 
  = \inf_{e \ge 1} \frac{\zeta(\ideala_e)}{p^e-1}. \]
  Similarly, if $\ideala_\star$ is a graded sequence, then 
  \[ \zeta(\ideala_\star) 
  := \lim_{e \to \infty} \frac{\zeta(\ideala_e)}{e} 
  = \inf_{e \ge 1} \frac{\zeta(\ideala_e)}{e}. \]
\end{lemma}
\begin{proof}
  See \cite[Lemma 2.3]{JonssonMustata}.
\end{proof}

\noindent We close this section by proving some statements
analogous to \cref{effect of right multiplication}, but with
Cartier subalgebras and $F$-graded sequences of ideals
in place of single $p^{-e}$-linear maps an field elements. 

\begin{lemma}\label{twisting rule}
    Let $\shD$ and $\shR$ be Cartier subalgebras on $X$, and let $\zeta \in X^\beth$ a 
    whose home is uniformly $\shD$-compatible. 
    \begin{enumerate}
      \item \textbf{Monotonicity:} If $\shD \subseteq \shR$ 
        then $A(\zeta; \shD) \le A(\zeta; \shR)$. 
        \label{monotonicity}
      \item \textbf{Conservation:} Let $\ideala_\bullet$ be an $F$-graded 
        sequence of ideals on $X$.  If $\zeta_F(\ideala_\bullet)$, $A(\zeta; \shD)$, 
        and $A(\zeta; \shD \cdot \ideala_\bullet)$ are all finite, then 
        $A(\zeta; \shD) = A(\zeta; \shD\cdot\ideala_\bullet) + \zeta_F(\ideala_\bullet).$ 
        \label{conservation}
    \end{enumerate}
\end{lemma}
\begin{proof}
    Monotonicity follows directly from the definition; 
    we therefore begin with conservation. Towards this end, 
    let $x = h_X(\zeta)$ and $Z = \overline{\{x\}}$. Then
    $\shD \cdot \ideala_\bullet \subseteq \shD$ implies
    that $Z$ is uniformly compatible with $\shD \cdot \ideala_\bullet$. 
    By passing to $\ideala_\bullet \shO_Z$ and
    $\shD\|_Z$, and re-setting notation (replacing $Z$ with $X$), 
    we assume $\zeta = v$ is a valuation, with valuation ring 
    $\shO_v \subset L$. Furthermore, the definitions of log discrepancy
    and $\zeta_F(\ideala_\bullet)$ are local near the center $z = c_X(v) \in X$, 
    so we restrict our attention to $R := \shO_{X,x}$ and write $\shD$ for the
    stalk $\shD_x$. Since $R$ is noetherian, $\ideala_e$ is finitely generated, 
    and so $\ideala_e \shO_v$ is principally generated, say by
    $g_e \in \ideala_e$. Possible generators are characterized by  
    $v(g_e) = v(\ideala_e)$.  

    Let $\psi \in D_e$ be nonzero. Then \cref{effect of right multiplication}
    tells us 
    \begin{align*}
      A(v; \psi^n) - A(v; \psi^n \cdot g_{ne}) &= \frac{v(g_{ne})}{p^{ne}-1} \\
      &= \frac{v(\ideala_e)}{p^{ne}-1}
    \end{align*}
    so $\lim_{n \to \infty} A(v; \psi^n) - A(v; \psi^n \cdot g_{ne}) 
    = v_F(\ideala_\bullet)$. 
    On the other hand, $A(v; \psi^n) = A(v; \psi)$ for all $n$ \eqref{limsup}, so 
    \[ A(v; \psi) = 
    v_F(\ideala_\bullet) + \lim_{n \to \infty} A(v; \psi^n \cdot g_{ne}). \]
    Since $\psi^n \cdot g_{ne} \in (\shD \cdot \ideala_\bullet)_{ne}$, applying
    \cref{limsup}\eqref{limsup for D} gives
    \[ A(v; \psi) \le v_F(\ideala_\bullet) + A(v; \shD \cdot \ideala_\bullet) \]
    for all $\psi \in \shD_e$. Now taking a supremum over $\psi \in (\shD_e)_x$,
    and $e \ge 1$, shows $A(v; \shD) \le v_F(\ideala_\bullet) + A(v; \shD \cdot \ideala_\bullet)$.

%
    We now establish the reversed inequality. By \cref{limit lemma}, it suffices
    to show that for all $\e > 0$ there exists $e > 0$ such that
    \[ A(v; \shD \cdot \ideala_\bullet) + \frac{1}{p^e-1} v_F(\ideala_e) 
    < A(v; \shD) + \e. \]
    From the definition of $A(v; \shD \cdot \ideala_\bullet)$, we know there 
    exists $e > 0$, $\psi \in (\shD \cdot \ideala_\bullet)_{e, x}$, and 
    $f \in L$ such that $\psi(f) = 1$ and
    \begin{equation}\label{conservation estimate} 
      A(v; \shD \cdot \ideala_\bullet) - \e < \frac{1}{p^e-1}v(f) = E(f, \psi, v). 
    \end{equation}
    By definition of $\shD \cdot \ideala_\bullet$, there exist
    $\psi_1, \dots \psi_n \in \shD_{e, x}$ and 
    $a_1, \dots, a_n \in \ideala_{e, x}$ such that
    \[ \psi(f) = \sum_1^n \psi_i(a_i f) = 1. \]
    Let $-c = \min_{1 \le i \le n}\{v(\psi_i(a_i f))\} 
    \le v(\psi(f)) = 0$. 
    By reindexing, we may assume that $-c = v(\psi_1(a_1b))$ 
    so that $\psi_1(ug^{p^e}a_1f) = 1$ for some unit 
    $u \in \shO_v^\times$ and $g = \psi_1(a_1f) \in \shO_v$ (note $v(g) = c$). 
    Since $a_1 \in \ideala_e$, it follows that $v(\ideala_e) \le v(a_1)$, so 
    \begin{align*}
        A(v; \shD \cdot \ideala_\bullet) - \e + \frac{1}{p^e-1}v(\ideala_e) 
        &\le A(v; \shD \cdot \ideala_\bullet) - \e + \frac{1}{p^e-1}v(a_1) \\
        &< E(f, \psi, v) + \frac{1}{p^e-1}v(a_1) \\
        &\le E(f, \psi, v) + \frac{1}{p^e-1}v(a_1) + \frac{p^e}{p^e-1}cp^e \\
        &= E(ug^{p^e}a_1f; \psi, v) \\
        &\le A(v; \shD). 
    \end{align*}
    Here we are using \eqref{conservation estimate}, and $c \ge 0$, 
    between the first and second lines, and second and third, respectively. 
    Finally, we used that $\psi_1 \in \shD_e$ and $\psi_1(ug^{p^e}a_1f) = 1$, 
    so $E(ug^{p^e}a_1f, \psi_1, v) \le A(v; \shD)$. 
\end{proof}

\begin{remark}
    The assumption that $\zeta_F(\ideala_\bullet)$, $A(\zeta; \shD)$, and 
    $A(\zeta; \shD \cdot \ideala_\bullet)$ are finite is 
    essential. Indeed, for a fairly trivial counterexample without these 
    assumptions (which demonstrates the main trouble), 
    we can take $\zeta = \triv_X$, $\shD = \shC_X$, and $\ideala_e = 0$ for 
    $e > 1$. Then $\zeta_F(\ideala_\bullet) = +\infty$, $A(\zeta; \shD) = 0$, 
    and $A(\zeta; \shD \cdot \ideala_\bullet) = -\infty$. 

    Alternatively, suppose $X$ is $F$-finite and regular, $Z \subset X$ is a 
    proper integral subscheme of dimension at least one, let $\p_Z$ be
    the associated prime ideal, and set $\ideala_e = (\p_Z\fpow{e} : \p_Z)$
    for $e > 1$. Then $\p_Z$ is compatible with 
    $\shD = \shC_X \cdot \ideala_\bullet$, and 
    $\shD\|_Z = \shC_Z$ \cite{Fedder83}. 
    Let $z \in Z_{reg}$ and $\zeta = \ord_z$ be the valuation corresponding 
    to the exceptional fiber of $\Bl_z(Z) \to Z$. 
    Then $A(\zeta; \shD) = A(\zeta; \shC_Z) = \dim(Z)$ while 
    $\zeta_F(\ideala_\bullet) = +\infty = A(\zeta; \shC_X)$.  
\end{remark}

\begin{definition}[Log discrepancies of graded sequences]
  \label{log discrepancy of graded sequences}
  Let $\ideala_\star$ be a graded sequence of ideals, and $\shD$ a 
  Cartier subalgebra, on $X$. For $t \in [0, \infty)$ we define
  the Cartier subalgebra
  \[ \shD \cdot \ideala_\star^t = \sum_{m \ge 1} \shD \cdot \ideala_m^{t/m}. \]
  For $\zeta \in X^\beth$, we have
  \[ A(\zeta; \shD \cdot \ideala_\star^t) 
  := \sup_{m \ge 1} A(\zeta; \shD \cdot \ideala_m^{t/m}). \]
\end{definition}

\noindent We also have conservation for graded sequences.

\begin{corollary}\label{twisting corollary}
    Let $\shD$ be a Cartier subalgebra on $X$, $t \in [0, \infty)$, and 
    $\zeta \in X^\beth$ whose home $Z$ is $\shD$-compatible. 
    Suppose $\ideala_\star \subseteq \shO_X$ is a multiplicatively graded 
    sequence of ideals on $X$ such that 
    $\zeta(\ideala_\star) < \infty$. If $A(\zeta; \shD)$, 
    $A(\zeta; \shD \cdot \ideala_\star^t)$, and $\zeta(\ideala_\star)$ are finite,
    then
    \[ A(\zeta; \shD) = 
    A(\zeta; \shD \cdot \ideala_\star^t) + t \,\, \zeta(\ideala_\star). \]
\end{corollary}
\begin{proof}
  We apply \cref{twisting rule} to the sequence of Cartier subalgebras
  $\shD \cdot \ideala_m^{t/m}$. By definition, 
  $\zeta(\ideala_m^{\lceil (t/m)(p^e-1) \rceil}) = \lceil (t/m)(p^e-1) \rceil \zeta(\ideala_m)$,
  so $\zeta_F(\{\ideala_m^{\lceil (t/m)(p^e-1) \rceil}\}_{e}) = (t/m) \zeta(\ideala_m)$. Taking suprema over
  $m \ge 1$ gives
  \begin{align*}
    A(\zeta; \shD \cdot \ideala_\star^t) &= \sup_{m \ge 1} A(\zeta; \shD \cdot \ideala_m^{t/m}) \\
    &= A(\zeta; \shD) - t v(\ideala_\star). 
  \end{align*}
\end{proof}

\noindent 
We return to studying asymptotic invariants of graded sequences of ideals 
in \S \ref{Section: LCT}.

%% file: zMainTheorem.tex
\section{Proof of the Main Theorem}
\label{Section: Main Theorem}

For this section, we fix a normal variety $X$ over an 
algebraically closed field $k$ of characteristic $p > 0$. 
We also fix a canonical Weil divisor $K_X$ on $X$, which 
fixes a canonical divisor $K_Y$ on every normal variety 
$Y$ with a proper birational morphism $\pi: Y \to X$ by 
requiring $\pi_*K_Y = K_X$.  
To state and prove our main theorem, we review 
the construction of log discrepancies of arbitrary 
valuations of log $\Q$-Gorenstein pairs $(X, \Delta)$. 
But first, let us give the proof of \eqref{CMS1}, using \eqref{CMS2}. 
Recall the notation $\shC^X \cdot \Delta$ from \cref{twist: divisor}. 

\begin{proposition}[{Cascini-Musta\c{t}\u{a}-Schwede}]\label{CMS3}
  Let $(X, \Delta)$ be a log $\Q$-Gorenstein pair,
  and assume the Cartier index of $K_X + \Delta$
  is not divisible by $p$. Then 
  $A_{(X, \Delta)}(E) = A(\ord_E; \shC^X \cdot \Delta)$
  for every divisor $E$ over $X$. 
\end{proposition}
\begin{proof}
  The valuation ring $R$ of $\ord_E$ must be $F$-finite: 
  $X$ is $F$-finite, so every variety birational 
  to $X$ is, also. Suppose $E \subset Y$ is a prime divisor
  on some normal variety $Y$ with a proper birational morphism 
  $\pi: Y \to X$. Then $\shO_{Y,E} \cong R$ is $F$-finite. 
  In particular, \eqref{CMS2} gives
  $A(\ord_E; \shC^Y) = A(\ord_E; \shC^R) = 1$. 
  If $\shC^R = \lbrak \Phi_E \rbrak$ (see proof of \ref{CMS2}),
  then the fact that $\shC^Y_1 \cong \omega_Y^{\otimes (1-p)}$
  implies $\Phi_E$ must correspond to a generator 
  for $\omega_{Y,E}^{\otimes (1-p)}$ over $R$. 

  Now suppose $(1-p^e)(K_X + \Delta)$ is Cartier, and 
  let $\psi \in (F^e)^!\shO_X((p^e-1)\Delta)$ 
  be the corresponding $p^{-e}$-linear map, \eqref{Subsection: divisors}.
  Then $\lbrak \psi \rbrak \cong (F^e)^!\shO_X((p^e-1)\Delta)$.  

  We define $\Delta_Y$ as in \cref{Subsection: divisors}:
  \[ K_Y + \Delta_Y = \pi^*(K_X + \Delta). \]
  The $p^{-e}$-linear map corresponding to $\Delta_Y$ in 
  $(F^e)^!\shO_Y((p^e-1)\Delta_Y)$ is also $\psi$, since 
  these maps must agree when restricted to any dense open 
  set where $\pi$ is an isomorphism. 
  Normality of $Y$ implies that $(p^e-1)\Delta_Y$ 
  is Cartier in a smooth neighborhood $V$ of the generic 
  point of $E$ in $Y$, say $(p^e-1)\Delta_Y \cap V = \div_V(h)$. 
  Letting $K_V = K_Y \cap V$, we then have 
  \[ V \cap (1-p^e)(K_Y + \Delta_Y) = (1-p^e)K_V - \div_V(h) \]
  so $\psi$ gives a generator of the line bundle
  \[ \shO_V((1-p^e)K_V - \div_V(h))= h \cdot \omega_V^{\otimes (1-p^e)}. \]
  Thus, $\psi = \Phi_E^e \cdot h$ via the $\shO_V$-linear isomorphism 
  $\omega_V^{\otimes (1-p^e)} \cong (F^e)^!\shO_V$. Applying 
  \eqref{effect of right multiplication} and \eqref{finite type} now gives
  \begin{align*}
    A(\ord_E; \shC^X \cdot \Delta) &= A(\ord_E; \psi) \\
    &= A(\ord_E; \Phi_E) - \frac{\ord_E(h)}{p^e-1} \\
    &= 1 - \ord_E(\Delta_Y)\\
    &= A_{(X, \Delta)}(E). 
  \end{align*}
\end{proof}

The remainder of this section builds on \eqref{CMS3}. 
We must first review and extend the method of 
Jonsson and Musta\c{t}\u{a} to our setting, a process 
we work out significantly more carefully than in previous
versions of this paper. 

\subsection{Log smooth pairs and domination.}
A {\em log smooth pair over $X$} consists of a smooth variety $Y$ 
with a proper birational morphism $\pi: Y \to X$ and 
a reduced snc divisor $D = \sum_i D_i$ on $Y$; 
we will say $\pi: (Y, D) \to X$ is a log-smooth pair over $X$. 
We would like to follow
\cite{JonssonMustata} in the construction of a partial order on 
log smooth pairs over $X$; since in our setting $X$ is not assumed to
be $\Q$-Gorenstein, we must handle boundary divisors, so our definition
becomes more involved. We fix, on each normal variety admitting 
a proper birational morphism $\pi: Y \to X$, a canonical class $K_Y$
by requiring $\pi_*K_Y = K_X$. This choice of $K_Y$ implies that
it is supported on $E + \pi_*\inv(K_X)$, where $E$ is the 
exceptional locus of $\pi$. 

\begin{definition}
  Fix a log $\Q$-Gorenstein pair $(X, \Delta)$.  
  We will say a log-smooth pair $\pi: (Y, D) \to X$ {\em dominates} 
  $(X, \Delta)$, and write $(Y, D) \succeq (X, \Delta)$, if 
  $\pi: Y \to X$ is a log resolution with the two properties below.
  \begin{enumerate}
    \item The support of $\pi^*(K_X + \Delta)$ is contained 
      in the support of $K_Y + D$. 
    \item The support of $E + (\pi_*\inv(\Delta))_{red}$ is
      contained in the support of $D$. 
  \end{enumerate}
  When $(Y, D) \succeq (X, \Delta)$, we call $\pi$ {\em the domination morphism}. 
\end{definition}

Following Jonsson and Musta\c{t}\u{a}, we extend $\succeq$ to 
a partial order on log smooth pairs over $X$ as follows. 
If $(Y', D')$ and $(Y, D)$ are two log smooth pairs dominating 
$(X, \Delta)$, with domination morphisms $\pi': Y' \to X$
and $\pi: Y \to X$, we will write $(Y', D') \succeq (Y, D)$ 
whenever $\pi'$ factors as $\pi \circ \mu$ for a proper birational morphism
$\mu: Y' \to Y$, and $\mu^*D$ is supported on $D'$. 

\subsection{Comparison of retractions}
For every log-smooth pair $(Y, D)$ over $X$, recall
from \cref{definition: retraction}
the retraction morphism $r_{(Y, D)}: \Val_X \to \Val_X$,
the image of which we denote by $QM(Y, D)$. Suppose
$D = \sum_i D_i$, with each $D_i$ a prime divisor on $Y$. 

Following \cite[Proposition/Definition 5.1]{JonssonMustata}, for 
$(Y, D) \succeq (X, \Delta)$ and for $v \in QM(Y, D)$ we define 
\begin{equation}\label{QM log discrepancy}
  A_{(X, \Delta)}(v) = \sum_i v(D_i) A_{(X, \Delta)}(D_i). 
\end{equation}
Our first result towards agreement of our log discrepancy with
established approaches is that the expression \eqref{QM log discrepancy}
is the value of $A(v; \shC_X \cdot \Delta)$ for $v \in QM(Y, D)$. 
The result, and proof, is quite similar to \eqref{log discrepancy of QM}. 
We choose to give a full proof since we can
be much more explicit in this setting, and give
details we refer to in the proof of our main theorem. 

\begin{lemma}\label{agreement for QM}
  Let $\Delta \ge 0$ be a $\Q$-Weil divisor on $X$ such that 
  $(1-p^e)(K_X + \Delta)$ is Cartier for some $e > 0$. 
  Suppose $(Y, D) \succeq (X, \Delta)$ and $v \in QM(Y, D)$. 
  Then $A_{(X, \Delta)}(v) = A(v; \shC_X \cdot \Delta)$. 
\end{lemma}
\begin{proof}
  Fix $v \in QM(Y, D)$ and define $x = c_X(v)$, $y = c_Y(v)$. Let $\pi: Y \to X$
  be the domination morphism. By assumption, 
  $\pi^*(K_X + \Delta) - K_Y = \Delta_Y$
  is supported on $D$. Suppose
  $\Delta_Y = \sum_{i = 1}^t \frac{b_i}{p^e-1} D_i$, where $b_i \in \Z$. 

  Since $(1-p^e)(K_X + \Delta)$ is Cartier, there is a corresponding 
  $\psi_\Delta \in (F^e)^!\shO_X((p^e-1)\Delta)$; similarly, 
  $(1-p^e)(K_Y + \Delta_Y) = \pi^*((1-p^e)(K_X + \Delta))$
  corresponds to some 
  $\psi_{\Delta_Y} \in (F^e)^!\shO_Y((p^e-1)\Delta_Y)$. 
  Since these $p^{-e}$-linear maps must agree on any dense open
  subset where $\pi$ is an isomorphism, they must be the same map. Therefore,
  \[ A(v; \shC_X \cdot \Delta) = A(v; \psi_{\Delta}) = A(v; \psi_{\Delta_Y}). \]
  The first equality is \cref{finite type}, the second following from 
  $\psi_\Delta = \psi_{\Delta_Y}$.

  Let $f_1, \dots, f_t \in \shO_{Y, y}$ be such that $D_i = \div(f_i)$ in 
  some neighborhood of $y$. By re-numbering the $D_i$ if necessary, we can 
  assume that $f_1, \dots, f_s$ give a regular system of parameters for 
  $\shO_{Y, y}$ 
  (so for $s < i \le t$, the divisor $D_i$ does not contain $y$). 
  The monomials $(f_1^{n_1} \cdots f_s^{n_s})$, with 
  $0 \le n_i \le p^e-1$,
  give a free basis for $\shO_{Y, y}$ over $\shO_{Y, y}^{p^e}$. 
  The projection $\Phi_y^e$ onto the basis element 
  $(f_1 \cdots f_s)^{(p^e-1)}$ gives a generator for 
  $\shC^Y_y$ over $\shO_{Y,y}$. Since 
  $y \not\in D_i$ for $s < i \le t$, 
  we see that $u := f_{s+1}^{b_{s+1}} \cdots f_t^{b_t}$ is a unit in $\shO_{Y,y}$.

  Recalling that $\Delta_Y = \sum_1^t b_i D_i$, there is a unit 
  $v \in \shO_{Y,y}^\times$ such that we have
  \[ \psi_{\Delta_Y} = \Phi_y^e \cdot (uv f_1^{b_1} \cdots f_s^{b_s}). \]
  For example, $v = 1$ when the generator $\Phi_y$ is chosen to correspond 
  to the specific embedding 
  $\omega_Y^{\otimes (1-p^e)} \cong \shO_Y((1-p^e)K_Y) \subseteq L$ 
  furnished by our choice of the Weil divisor $K_Y$. 

  Now using \cref{log discrepancy of QM} and 
  \cref{effect of right multiplication}, we get
  \begin{align*}
    A(v; \psi_{\Delta_Y}) &= 
    A(v; \Phi_y^e) - \frac{1}{p^e-1} \left( \sum_{i = 1}^s b_i v(f_i) \right) \\
    &= \left( \sum_{i = 1}^s v(f_i) \right) - \frac{1}{p^e-1} 
    \left( \sum_{i = 1}^s b_i v(f_i) \right) \\
                          &= \sum_{i = 1}^s v(f_i) \left( 1 - \frac{b_i}{p^e-1} \right) \\
                          &= \sum_{i = 1}^s v(D_i) (1 - \ord_{D_i}(\Delta_Y) \\
                          &= \sum_{i = 1}^s v(D_i) A_{(X, \Delta)}(D_i) \\
                          &= A_{(X, \Delta)}(v). 
  \end{align*}
\end{proof}

\begin{lemma}\label{non-decreasing along retraction}
  Let $(X, \Delta)$ be a log $\Q$-Gorenstein pair. 
  Suppose $(Y', D')$ and $(Y, D)$ are log smooth pairs 
  over $X$, and $(Y', D') \succeq (Y, D) \succeq (X, \Delta)$. 
  Then for all $v \in \Val_X$, 
  \[ A_{(X, \Delta)}(r_{(Y, D)}(v)) \le A_{(X, \Delta)}(r_{(Y', D')}(v)). \]
\end{lemma}
\begin{proof}
  Let $w' = r_{(Y', D')}(v)$, which is in $QM(Y', D')$.
  Since $r_{(Y, D)} \circ r_{(Y', D')} = r_{(Y, D)}$, 
  it suffices to prove that 
  $A_{(X, \Delta)}(r_{(Y, D)}(w')) \le A_{(X, \Delta)}(w')$. 

  Let $\pi': Y' \to X$, $\pi: Y \to X$, and $\mu: Y' \to Y$ be
  the domination morphisms between these pairs. 
  Define $\Delta_Y$ and $\Delta_{Y'}$ via:
  \[ K_Y + \Delta_Y = \pi^*(K_X + \Delta) \]
  and
  \[ K_{Y'} + \Delta_{Y'} = (\pi')^*(K_X + \Delta) = \mu^*K_Y + \mu^*\Delta_Y. \]
  Write $D = \sum_i D_i$ and $D' = \sum_j D'_j$ for the snc
  divisors on $Y$ and $Y'$, respectively, with each $D_i$ and $D'_j$
  a prime divisor. From $(Y', D') \succeq (Y, D)$ we see that there
  are integers $b_{i,j}$ such that
  \[ \mu^*D_i = \sum_j b_{i,j}D'_j \]
  for all $i$. The log discrepancies $A_{(X, \Delta)}(r_{(Y, D)}(w'))$
  and $A_{(X, \Delta)}(w')$ of question are:
  \[ A_{(X, \Delta)}(r_{(Y, D)}(w')) = \sum_i w'(D_i)A_{(X, \Delta)}(D_i) \]
  and 
  \[ A_{(X, \Delta)}(w') = \sum_j w'(D'_j)A_{(X, \Delta)}(D'_j). \]
  Since $w'(D_i) = w'(\mu^*D_i) = \sum_j b_{i,j}w'(D'_j)$, we can re-write
  $A_{(X, \Delta)}(r_{(Y, D)}(w'))$ as
  \[ \sum_i \left( \sum_j b_{i,j}w'(D'_j) \right) A_{(X, \Delta)}(\ord_{D_i}) = \sum_j \left[ w'(D'_j) \left( \sum_i b_{i,j} A_{(X, \Delta)}(\ord_{D_i}) \right) \right]. \]
  Each $w'(D'_j)$ is non-negative, since $w' \in QM(Y', D')$, so to prove 
  $A_{(X, \Delta)}(r_{(Y, D)}(w')) \le A_{(X, \Delta)}(w')$ it suffices 
  to prove
  \[ \sum_i b_{i,j} A_{(X, \Delta)}(D_i) \le A_{(X, \Delta)}(D'_j) \]
  for each $j$; fix one, and rename the corresponding divisor $G$. 
  To simplify notation, we define $b_i$ to be the coefficient on $G$ in $\mu^*D_i$. 
  Thus,
  \begin{align*}
    A_{(X, \Delta)}(G) &= 1 + \ord_{G}(K_{Y'} - \mu^*(K_Y + \Delta_Y)) \\
                          &= A_{(Y, 0)}(G) - \ord_{G}(\mu^*\Delta_Y). 
  \end{align*}
  From $(Y, D) \succeq (X, \Delta)$, we know $\Delta_Y$ is supported on $D$;
  let $\Delta_Y = \sum_i r_i D_i$ for some $r_i \in \Q$. Therefore, 
  \[ \mu^*\Delta_Y = \sum_i r_i \mu^*D_i. \]
  Focusing on $G$, we see that
  \[ \ord_G(\mu^*\Delta_Y) = \sum_i r_i b_i. \]
  To complete the proof, we need the following estimate. 
  Jonsson and Musta\c{t}\u{a} prove essentially the same 
  estimate in \cite[Lemma 1.5]{JonssonMustata} for regular excellent
  $\Q$-schemes. We are on a variety, so we can use the K\"ahler differential
  sheaves instead of sheaves of special differentials, cf. \cite{dFEM}. 

  \noindent\textbf{Claim:} $A_{(Y, 0)}(G) \ge \sum_i b_i$.

  \begin{pfClaim}
    Let $\eta'$ be the generic point of $G$, and $\eta = \mu(\eta')$. 
    By re-numbering the $D_i$, we may assume $\eta \in D_1 \cap \cdots \cap D_s$,
    and $\eta \not\in D_i$ for $i > s$. 
    For $1 \le i \le s$, choose $f_i \in \shO_{Y, \eta}$ with 
    $D_i = \div(f_i)$ near $\eta$. Let $\varpi$ be a generator
    for the maximal ideal of $\shO_{Y', \eta'}$. 

    The differentials $df_1, \dots, df_s$ form part of a basis for
    $\Omega_{Y, \eta}$, since $D$ is snc. Choose 
    $g_{s+1}, \dots, g_n \in \shO_{Y,\eta}$ so that the 1-forms 
    $df_1, \dots, df_s, dg_{s+1}, \dots, dg_n$ 
    give a basis for $\Omega_{Y, \eta}$ over $\shO_{Y, \eta}$. 
    Setting 
    \[ \delta = (\wedge_{i = 1}^s df_i) \wedge (\wedge_{j = s+1}^n dg_j) \]
    gives a generator for $\omega_{Y, \eta}$. Suppose
    $c_j = \ord_G(g_j)$, $s < j \le n$. Then it is straightforward to see 
    \[ (\mu^*\delta) \shO_{Y', \eta'} =
    \varpi^{(\sum_i b_i) + (\sum_j c_j) - 1} \omega_{Y', \eta'}. \]
    Therefore, by definition $A_{(Y, 0)}(G) = (\sum_i b_i) + (\sum_j c_j)$,
    which is at least $(\sum_i b_i)$ since each $c_j \ge 0$. 
  \end{pfClaim}
  Putting all of this together, we conclude:
  \begin{align*}
    A_{(X, \Delta)}(G) &= A_{(Y, 0)}(G) - \ord_{G}(\mu^*\Delta_Y) \\
                              &= A_{(Y, 0)}(G) - \sum_i r_i b_i \\
                              &\ge \sum_i b_i (1 - r_i) \\
                              &= \sum_i b_i(1 - \ord_{D_i}(\Delta_Y)) \\
                              &= \sum_i b_iA_{(X, \Delta)}(D_i). 
  \end{align*}

\end{proof}

\begin{corollary}\label{non-decreasing for QM}
  With the notation as in the previous proposition, 
  \[ A(r_{(Y, D)}(v); \shC^X \cdot \Delta) \le 
  A(r_{(Y', D')}(v); \shC^X \cdot \Delta). \]
\end{corollary}
   
\subsection{Log discrepancies of arbitrary valuations.}
Assume for this subsection that
log resolutions of Weil divisors on varieties birational to $X$
exist, so that the collection of log-smooth
pairs $(Y, D) \succeq (X, \Delta)$ is non-empty, and 
Abhyankar valuations admit global monomializations. 
Still following Jonsson and Musta\c{t}\u{a}, 
we define the log discrepancy of arbitrary $v \in \Val_X$ to be:
\begin{equation}\label{definition: JM log discrep}
  A_{(X, \Delta)}(v) = \sup_{(Y, D) \succeq (X, \Delta)} A_{(X, \Delta)}(r_{(Y, D)}(v)). 
\end{equation}
Note that this is well defined, thanks to 
\eqref{non-decreasing along retraction}. 
There are numerous reasons to believe this 
is the correct extension of $A_{(X, \Delta)}$ 
from $X^\div$ to all of $\Val_X$. For example, this 
extension is the maximal lower-semicontinuous extension 
of the log discrepancy on $X^\div$. In characteristic zero, 
Mauri, Mazzon, and Stevenson \cite{MauriMazzonStevenson} 
recently proved that this definition coincides with Temkin's 
pluricanonical metric from \cite{TemkinMetrization}; their 
approach relates log discrepancies to the {\em weight metrics}
of Musta\c{t}\u{a} and Nicaise \cite{MustataNicaiseWeightFunctions}, 
which give an analogous function for discretely 
valued ground fields. 

Our main theorem is that in the case log resolutions exist, 
defining $A_{(X, \Delta)}(v)$ for valuations $v$ using 
\eqref{definition: JM log discrep} gives the same function 
on $\Val_X$ as \cref{definition: log discrepancy function}. 

\begin{theorem}\label{agreement with JM}
  Let $(X, \Delta)$ be a log $\Q$-Gorenstein pair, and suppose 
  that the Cartier index of $K_X + \Delta$ is not divisible by $p$. 
  Suppose log resolutions exist for Weil 
  divisors on varieties birational to $X$. Then
  $A_{(X, \Delta)}(v) = A(v; \shC^X \cdot \Delta)$ for all $v \in \Val_X$. 
\end{theorem}
\begin{proof}
  Fix $v \in \Val_X$. Since $r_{(Y, D)}(v) \in QM(Y, D)$, 
  \cref{agreement for QM} provides an equality 
  $A_{(X, \Delta)}(r_{(Y, D)}(v)) = A(r_{(Y, D)}(v); \shC^X \cdot \Delta)$
  for every $(Y, D) \succeq (X, \Delta)$. 

  All of our considerations are local near $x = c_X(v)$, so we restrict our
  attention to $R = \shO_{X,x}$, denoting by $\m = \m_x$ the maximal ideal 
  of $R$. The prime ideal $\p$ of $R$ associated to 
  $c_X(r_{(Y, D)}(v)) \in X$ is contained in $\m$ for every 
  $(Y, D) \succeq (X, \Delta)$. 

  \noindent\textbf{Claim ($\star$) :} 
  For every $(Y, D) \succeq (X, \Delta)$, there exists 
  $(Y', D') \succeq (Y, D)$ such that $c_X(r_{(Y', D')}(v)) = \m$. 

  \begin{pfClaim}
    Fix some $(Y, D) \succeq (X, \Delta)$ with domination morphism 
    $\pi: Y \to X$, and assume $\p = c_{(Y ,D)}(v) \subsetneq \m$.
    Then we can choose $f_1 \in \m \setminus \p$; this element is 
    defined in some neighborhood $U$ of $\m$ in $X$, and we define 
    $N$ to be the scheme-theoretic closure of  $\div_U(f)$ in $X$, 
    which is an effective Weil divisor. Take $(Y^{(1)}, D^{(1)})$ to 
    be a log resolution of $(X, \Delta + N)$ dominating $(Y, D)$;
    note that such a log resolution exists by taking a log
    resolution $\mu: Y^{(1)} \to Y$ of $(Y, D + \pi_*\inv(\Delta + N)_{red})$. 
    If $\mu^{(1)}: Y^{(1)} \to X$ is the domination morphism, then by 
    definition the strict transform $(\mu^{(1)})_*\inv N$ is supported on 
    $D^{(1)}$. In particular, $r_{(Y^{(1)}, D^{(1)})}(v)(f) > 0$. 
    Therefore, 
    \[ \p \subsetneq c_X(r_{(Y^{(1)}, D^{(1)})}(v)) \subseteq \m. \]
    Call $\p_1 = c_X(r_{(Y^{(1)}, D^{(1)})}(v))$. If there exists 
    $f^{(2)} \in \m \setminus \p_1$, we repeat this argument with 
    $(Y^{(1)}, D^{(1)})$ in place of $(Y, D)$, giving 
    $(Y^{(2)}, D^{(2)}) \succeq (Y^{(1)}, D^{(1)})$ and a new center 
    $\p_2 = c_X(r_{(Y^{(2)}, D^{(2)})}(v))$. 
    Thus, we get a chain
    \[ \p \subsetneq \p_1 \subsetneq \p_2 \subsetneq \cdots \m \]
    with $\cup_i \p_i = \m$. Since $X$ is noetherian,
    this chain must stabilize, meaning $\p_i = \m$ for $i \gg 0$. 
    Taking $(Y', D') = (Y^{(i)}, D^{(i)})$ gives the desired pair. 
  \end{pfClaim}

  If $(Y, D)$ is a pair dominating $(X, \Delta)$ with 
  $c_X(r_{(Y, D)}(v)) = c_X(v) = \m$, we write 
  $(Y, D) \gtrdot (X, \Delta)$. Using 
  \eqref{agreement for QM}, \eqref{non-decreasing along retraction}, 
  and Claim ($\star$) above, we see
  \begin{align*}
    A_{(X, \Delta)}(v) 
    :&= \sup_{(Y, D) \succeq (X, \Delta)} A_{(X, \Delta)}(r_{(Y, D)}(v)) \\
    &= \sup_{(Y, D) \gtrdot (X, \Delta)} A_{(X, \Delta)}(r_{(Y, D)}(v)) \\
    &= \sup_{(Y, D) \gtrdot (X, \Delta)} A(r_{(Y, D)}(v); \shC^X \cdot \Delta) \\
  \end{align*}

  The Cartier subalgebra $\shC^X \cdot \Delta$ is generated near $\m$ 
  by a single $p^{-e}$-linear map $\psi$; see \eqref{Subsection: divisors}. 
  \Cref{finite type} implies $A(w; \shC^X \cdot \Delta) = A(w; \psi)$
  for every $w \in \Val_X$ with $c_X(w) = \m$, so we would be
  done if we could show 
  \begin{equation}\label{key claim}
    \sup_{(Y, D) \gtrdot (X, \Delta)} A(r_{(Y, D)}(v); \psi) = A(v; \psi). 
  \end{equation}
  We cannot follow a definitional proof of this because 
  the expressions $E(f, \psi^n, v)$ used to define $A(v; \psi)$ are 
  {\em not necessarily non-decreasing}
  along retractions, i.e. do not satisfy an analogue of 
  \eqref{non-decreasing along retraction}. We therefore proceed
  by careful estimates of $A(v; \psi)$, and by using ideas
  similar to those in the proofs of \eqref{log discrepancy of QM} and 
  \eqref{agreement for QM}. To simplify the notation going
  forward, call $\mathcal{S}$ the supremum on the left in 
  \eqref{key claim}. 
  
  For any fixed $f \in R$, \cite[Lemma 4.7]{JonssonMustata} shows
  $v(f) = r_{(Y, D)}(f)$ whenever $(Y, D)$ gives a 
  log resolution of $(X, \overline{\div(f)})$.  
  Using ideas similar to those in the proof of Claim 
  ($\star$), any log resolution for the closure $N_{f,n}$ of
  the divisor $\div(f) + \div(\psi^n(f))$ is 
  dominated by some $(Y, D) \gtrdot (X, \Delta)$. When 
  $(Y, D) \gtrdot (X, \Delta)$ is a log resolution of $N_{f,n}$, 
  we do have $E(f, \psi^n, v) = E(f, \psi^n, r_{(Y, D)}(v))$. 

  Let $\e > 0$, and choose $n \ge 1$ and $f \in R$ so that
  \[ A(v; \psi) - \e < E(f, \psi^n, v). \]
  Suppose $(Y, D) \gtrdot (X, \Delta)$ is also a log resolution 
  of the Weil divisor $N_{f,n}$ defined in the previous paragraph. 
  Then $E(f, \psi^n, v) = E(f, \psi^n, r_{(Y, D)}(v)) \le A(r_{(Y, D)}(v); \psi)$, so
  \[ A(v; \psi) - \e < A(r_{(Y, D)}(v); \psi) \le \mathcal{S}. \]
  Since this was true for all $\e$, we conclude 
  $A(v; \psi) \le \mathcal{S}$. 
  
  For the reversed inequality, we use \eqref{log discrepancy of QM}. 
  Let $(Y, D) \gtrdot (X, \Delta)$ and $w = r_{(Y, D)}(v)$. We prove 
  $A(w; \psi) \le A(v; \psi)$. 
  Call $\eta = c_Y(w)$, and $y = c_Y(v)$. 
  If we number the components $D_i$ of $D$ so that $v(D_i) > 0$ if 
  and only if $1 \le i \le s$, then local equations 
  $f_1, \dots, f_s \in \shO_{Y, \eta}$ for $D_1, \dots, D_s$ (resp.)
  generate the maximal ideal 
  of $\shO_{Y, \eta}$. This generating set can be extended by 
  $g_1, \dots, g_{n-s} \in \shO_{Y,y}$ to a set of generators 
  for the maximal ideal of $\shO_{Y,y}$. Fix the generator 
  $\Phi_y$ for $\shC^Y_y$ that projects onto 
  $(f_1 \cdots f_s g_1 \cdots g_{n-s})^{(p-1)}$. 
  Then $\psi = \Phi_y^e \cdot h$ for some $h \in L = \kappa(X)$. 
  We have assumed $(Y, D)$ gives a log resolution for $(X, \Delta)$, 
  so in particular $\pi^*(K_X + \Delta)$ is snc and supported
  on $K_Y + D$. This implies that, after possibly multiplying 
  $\Phi_y^e$ on the right by a unit $u \in \shO_{Y,y}^\times$, 
  we have $h = \prod_{i = 1}^s f_i^{c_i}$ for some $c_i \in \Z$.
  See the proof of \eqref{agreement for QM} for more details. 
  \Cref{log discrepancy of QM} shows
  \[ A(w; \psi) = w(D) - \frac{w(h)}{p^e-1} = \frac{v(D) - v(h)}{p^e-1} =  
  \left( \sum_i v(D_i) (p^e -1 - c_i) \right). \]

  Suppose we can find $f \in L$ such that $\psi(f)$ is a
  unit in $\shO_{Y,y}$ and $E(f, \psi, v) \ge A(w; \psi)$; from this
  it follows $A(v; \psi) \ge A(w; \psi)$. The choice for $f$ is clear:
  $f =(\prod_{j = 1}^{n-s} g_j^{p^e-1})(\prod_{i = 1}^s f_i^{p^e - 1 - c_i})$. 
  By construction, 
  \[
    \psi(f) = \Phi_y^e(u(f_1 \cdots f_s g_1 \cdots g_{n-s})^{(p^e-1)}) 
    \in \shO_{Y,y}^\times.
  \]
  Recall that to get $h$ to be monomial, we may have had to multiply $\Phi_y^e$
  by some $u \in \shO_{Y,y}^\times$, so $\psi(f)$ may not be $1$.
  We also see:
  \begin{align*}
    (p^e-1)E(f, \psi, v) &= \left( \sum_{j = 1}^{n-s} (p^e-1)v(g_i) \right) + \left(\sum_{j = 1} (p^e-1-c_j)v(f_i) \right) \\
                         &\ge (p^e-1)A(w; \psi)
  \end{align*}
  since $v(g_j) > 0$ for all $j$. 
  We conclude that \eqref{key claim} is true, which completes the proof. 
\end{proof}

%% file: zF-singularities.tex
\section{Connections with $F$-singularities}
\label{Section: F-things}

In this section, we briefly explore the relationship 
between our log discrepancies, sharp $F$-purity,
and strong $F$-regularity. These results are of independent
interest, and are also important in \Cref{Section: LCT}
to prove, e.g., that asymptotic multiplier ideals of
graded sequences on strongly $F$-regular schemes are coherent.

We prove that sharply $F$-pure and strongly $F$-regular 
Cartier subalgebras are characterized as non-negativity (resp. positivity)
of log discrepancies on $\beth$-spaces. This builds on
the heuristic correspondence between sharply $F$-pure and log 
canonical singularities (resp. strongly $F$-regular and klt). 
In particular, our result greatly generalizes Hara and
Watanabe's theorem \cite[Theorem 3.3]{HaraWatanabe}. 

Throughout, we fix a Cartier subalgebra $\shD \subseteq \shC_X$ 
on an integral $F$-finite scheme $X$. To avoid trivialities,
we assume $\shD_e \ne 0$ for some $e > 0$.

\begin{definition}
  The {\em splitting prime} of $\shD$ at $x \in X$ is the ideal 
  \[ \mc{P}(\shD_x) = 
  \{f \in \shO_{X, x} \,:\, \psi(f) \in \m 
  \text{ for all $\psi \in (\shD_e)_x$, for all $e \ge 1$ }\}. \]
  Standard facts about $\mc{P} = \mc{P}(\shD_x)$ include 
  (see \cite{AberbachEnescuStructureOfFPure, BlickleSchwedeTuckerFsig1}):
  \begin{enumerate}
    \item $\mc{P} \ne \shO_X$ if and only if $\shD$ is sharply $F$-pure at $x$.
    \item As suggested by the name, $\mc{P}$ is prime whenever it is proper. 
    \item When $\mc{P}$ is proper, no prime $\p \in \Spec(R)$ with 
      $\mc{P} \subsetneq \p$ can be $\shD_x$-compatible. In particular, 
      $\mc{P} = 0$ if and only if $\shD$ is strongly $F$-regular at $x$. 
    \item The restriction $\shD_x\|_{\mc{P}}$ is strongly $F$-regular if
      $\shD$ is sharply $F$-pure at $x$. 
  \end{enumerate}
\end{definition}

\begin{lemma}\label{detecting F-things}
  Let $x \in X$. Denote by $\m_x$ the maximal ideal of $\shO_{X,x}$
  and by $\triv_x$ the trivial valuation $\kappa(x)^\times \to \{0\}$. 
  There are three possible values for $A(\triv_x, \shD)$:
  \[
    A(\triv_x, \shD) = \left\{\begin{tabular}{ c  l |}
        \hline
      $-\infty$ & iff $\shD_x$ is not sharply $F$-pure.\\
        \hline
      $0$       & iff $\shD_x$ is sharply $F$-pure and 
      $\mc{P}(\shD_x) = \m_x$.\\
        \hline
      $+\infty$ & iff $\shD_x$ is sharply $F$-pure and 
      $\mc{P}(\shD_x) \ne \m_x$.\\
      \hline
    \end{tabular}\right.
  \]
\end{lemma}
\begin{proof}
  Restricting our attention to $R := \shO_{X,x}$, let us
  write $\shD$ for $\shD_x$, $X = \Spec(R)$, and $x$
  for the closed set $\{x\} \subset X$. 
  
  Suppose $\shD$ is not sharply $F$-pure. 
  Then no $\psi \in \shD_e$ is surjective for $e > 0$, 
  so in particular $\psi(\m_x) \subseteq \m_x$.
  Thus, $x$ is uniformly $\shD$-compatible, and $\psi(R) \subseteq \m_x$
  for all $\psi \in \shD_{> 0}$ implies the exceptional
  restriction $\shD\|_x$ is zero. But then
  \[ A(\triv_x; \shD) = A(\triv_x; 0) = \sup \varnothing = -\infty. \]
  
  If $\shD$ is sharply $F$-pure and $\mc{P}(\shD) = \m_x$, then $x$
  is again $\shD$-compatible. In this case, one has non-zero elements of 
  $\shD\|_x$ in positive degrees, corresponding to surjective
  maps $\psi \in \shD_e$. If $\psi\|_x \ne 0$, then $A(\triv_x; \psi\|_x) = 0$.
  Thus, $A(\triv_x; \shD) = \sup \{0\} = 0$. 

  In the last case, we see $x$ is not uniformly $\shD$-compatible,
  and we defined \linebreak $A(\triv_x; \shD) = +\infty$.  
\end{proof}

The center map admits a section $\triv: X \to X^\beth$,
sending $x \in X$ to $\triv_x$. Let us write
$X^\triv$ for the image of this map and $\triv_X$
for the image of the generic point of $X$. Let 
$X^{\beth, *} = X^\beth \setminus \{\triv_X\}$, and 
$X^{\triv, *} = X^\triv \cap X^{\beth, *}$. Our main theorem
in this section is:

\begin{theorem}\label{characterization of ShFP and SFR}
  Let $X$ be an $F$-finite integral scheme, and let $\shD$
  be a Cartier subalgebra on $X$. Then $\shD$ is:
  \begin{enumerate}
    \item sharply $F$-pure if and only if $A(\zeta; \shD) \ge 0$ 
      for all $\zeta \in X^\beth$. 
    \item strongly $F$-regular if and only if $A(\zeta; \shD) > 0$ 
      for all $\zeta \in X^{\beth,*}$. 
  \end{enumerate}
  Moreover, it suffices to check these statements 
  on $X^\triv$ and $X^{\triv, *}$, respectively. 
\end{theorem}
\begin{proof}
  Sharp $F$-purity, and strong $F$-regularity, are conditions $\shD$ 
  must satisfy at each point of $X$, and $A(\zeta; \shD)$ depends only 
  on $\phi \in \shD_{c(\zeta)}$.  Therefore, we assume $X = \Spec(R)$, 
  $R$ is local with maximal ideal $\m$, and $c_X(\zeta) = \m$. 
  Simplifying notation, let $\shD = \shD_\m$. 
  
  Suppose first that $\shD$ is sharply $F$-pure and 
  let $\psi \in \shD_e$ be surjective with $\psi(f) = 1$. Then
  for all $\zeta \in c_X\inv(\m)$, we have
  \[ A(\zeta; \shD) \ge E(f, \psi, \zeta) \ge 0. \]
  On the other hand, if $\shD$ is not sharply $F$-pure,
  then \eqref{detecting F-things} shows $A(\triv_\m; \shD) = -\infty$,
  so $A(\zeta; \shD) < 0$ for some $\zeta \in X^\beth$. 

  Suppose then $\shD$ is strongly $F$-regular, and let $\zeta \in X^{\beth, *}$
  with $c_X(\zeta) = \m$. If $h_X(\zeta) \ne (0)$, then $h_X(\zeta)$ is not 
  uniformly $\shD$-compatible, and we have defined 
  $A(\zeta; \shD) = +\infty > 0$. If $h_X(\zeta) = (0)$, meaning
  $\zeta \in \Val_X$, take $f \in \m$ and $\psi \in \shD_e$ with 
  $\psi(f) = 1$. Then $A(\zeta; \shD) \ge E(f, \psi, v) > 0$. 
  Contrapositively, suppose $\shD$ is not strongly $F$-regular. 
  If $\shD$ is not even sharply $F$-pure,
  then the previous case shows $A(\triv_\m; \shD) = -\infty \le 0$. 
  We therefore assume $\shD$ is sharply $F$-pure. Then
  $\mc{P}(\shD) =: \p$ is a nonzero prime ideal of $R$, and
  $\shD_\p$ is a sharply $F$-pure Cartier subalgebra on $R_\p$
  with $\mc{P}(\shD_\p) = \p R_\p$. We have seen $A(\triv_\p; \shD) = 0$. 

  To complete the proof, we note that in both cases, the points of $X^\triv$
  gave semivaluations with negative (resp. non-positive) log discrepancy. 
\end{proof}

\begin{corollary}[{cf. \cite{HaraWatanabe}}]
  \label{HW corollary}
  Let $X$ be an integral $F$-finite scheme. 
  \begin{enumerate}
    \item If $X$ is $F$-pure, then $A(E; \shC^X) \ge 0$ for all
      divisors $E$ over $X$. 

    \item If $X$ is $F$-regular, then $A(E; \shC^X) > 0$ for all
      divisors $E$ over $X$. 
  \end{enumerate}
\end{corollary}
  
\begin{question}
  Suppose $\shD$ is sharply $F$-pure but not strongly $F$-regular at $x \in X$. 
  Does there exist a non-trivial $\zeta \in X^\beth$ with $c_X(\zeta) = x$ 
  and $A(\zeta; \shD) = 0$? 
\end{question}

\begin{question}
  What is the relationship between the $F$-signature $s(D_x)$ 
  \cite{BlickleSchwedeTuckerFSigPairs1} and $A_X(\zeta; \shD)$, 
  for various $\zeta \in X^\beth$ centered at $x$? 
  Cf. \cite[Theorem 3.1]{CantonLeftDerivativeOfFSignature}; note 
  that the limit function in that theorem 
  factors as $\left( \frac{1}{2} \mld(R, f^t) \right)^2$. 
\end{question}

%% file: zLSC.tex
\section{Lower-semicontinuity}
\label{Section: LSC}
For this section, let $X$ be an integral scheme of characteristic $p > 0$. 
We show that $A(-; \shD)$ is lower-semicontinuous on $X^\beth$ for any 
Cartier subalgebra $\shD$ on $X$. As a first application, we deduce that 
the minimal log discrepancy function derived from log discrepancies 
on $X^\beth$ is lsc on $X$ considered with \textbf{the constructible topology}. 
In \Cref{Section: LCT}, we give our major applications of 
lower-semicontinuity of $A(-; \shD)$: 
coherence of asymptotic multiplier ideals, and existence of 
valuations calculating log canonical thresholds, on regular
$F$-finite schemes. 

\begin{definition}\label{lsc: definition}
  Recall that a function $f: Y \to \R_{\pm \infty} = [-\infty, \infty]$ on a 
  topological space $Y$ is {\em lower-semicontinuous} (lsc) at 
  $y_0 \in Y$ if $f(y_0) = -\infty$ or one of the following equivalent conditions holds:
  \begin{enumerate}
    \item For every convergent net $y_\alpha \to y_0$, 
      $f(y_0) \le \liminf_\alpha f(y_\alpha)$. \label{lsc: nets}
    \item For every $\e > 0$ there exists an open neighborhood 
      $U \subseteq Y$ of $y_0$ such that $f(y_0) - \e < f(y)$ for all 
      $y \in U$;\label{lsc: opens}
    \item If we consider $\R_{\pm \infty}$ with the topology whose open 
      subsets are of the form $(a, \infty]$, $a \in \R$, and 
      $[-\infty, \infty]$, then $f$ is continuous. 
      \label{lsc: continuous}
  \end{enumerate}
\end{definition}

\begin{remark}
  Let us make some comments on topological notions used here.  
  \begin{enumerate}
    \item Since $X^\beth$ is not generally first countable, we use nets 
      and not sequences for questions of convergence and compactness. 

    \item If a topological space $Y$ is not Hausdorff, there may be 
      more than one limit point of a convergent net. We write $\lim y_\beta$ 
      for the {\bf set of limit points} of a convergent net $y_\beta$. 
      In the case $\lim y_\beta$ consists of one point $y_\ast$, per usual
      we write $y_\ast = \lim y_\beta$. 
  \end{enumerate}
\end{remark}

The following lemma is a technical generalization of a classical way 
of producing a new lsc function from a given collection of lsc functions. 
The author thanks Kevin Tucker for suggesting this approach, 
which leads to a much simpler proof of lower-semicontinuity than 
the author's original. 

\begin{lemma}[cf. \cite{Folland}, Proposition 7.11(c)]\label{lsc sheaf lemma}
  Let $Y$ be a topological space and let $\shG$ be a sheaf of 
  lsc functions on $Y$, meaning for every open subset $V \subseteq Y$,
  $\shG(V)$ is a (possibly empty) collection of $\R_{\pm\infty}$-valued lsc functions 
  defined on $V$. Then $a(y) := \sup\{ g(y) \,:\, g \in \shG_y\}$ is lsc. 
\end{lemma}
\begin{proof}
  The proof is very similar to one in \cite{Folland}. 
  We use definition \ref{lsc: definition}\eqref{lsc: continuous}. 

  Fix $r \in \R$ and $(r, \infty] \subseteq \R_{\pm\infty}$, and 
  let $G = \displaystyle\bigsqcup_{y \,:\, a(y) > r} \shG_y$. 
  Each $g \in \shG_y \subset G$ represents an equivalence class of 
  lsc functions $g_U: U \to \R_{\pm \infty}$ on open subsets $U$ containing $y$. 
  If $r < g(y)$, then   $r < g_U(y)$ for each $U$ and $g_U$ that $g$ represents, so
  $g_U\inv(r, \infty]$ is a non-empty open subset of $U$ (so also of $X$). 
  Write $g \sim (U, g_U)$ to mean $g$ is the image of $g_U$ in $\shG_y$. I claim that 
  \[ a\inv(r, \infty] = 
  \cup_{g \in G} \cup_{g \sim (U, g_U)}\,\, g_U\inv(r, \infty]. \]
  From this, it will follow that $a\inv(r, \infty]$ is open. 

  If $a(y) > r$, then $\shG_y \subseteq G$ and there exists 
  some $g \in \shG_y$ with $g(y) > r$. Therefore, 
  $y \in g_U\inv(r, \infty]$ for any $g_U$ with $g \sim (U, g_U)$. 
  On the other hand, if 
  $y \in \cup_{g \in G} \cup_{g \sim (U, g_U)} g_U\inv(r, \infty]$, then 
  $y \in g_U\inv(r, \infty]$ for some $(U, g_U) \sim g \in \shG_y$ and
  some $U \ni y$. Thus, $r < g_U(y) \le a(y)$, so $y \in a\inv(r, \infty]$. 
\end{proof}

\begin{theorem}\label{log discrepancy lsc}
  For every Cartier subalgebra $\shD$ on an integral scheme $X$ of positive 
  characteristic, the log discrepancy $A(-; \shD)$ is lsc on $X^\beth$. 
\end{theorem}
\begin{proof}
  To make the notation for preimages of sets of $\R_{\pm \infty}$
  under $A(-; \shD)$ more sensible, let us simply write $A(\zeta)$ 
  for $A(\zeta; \shD)$. We first reduce to the affine case. Suppose 
  $\{U_i = \Spec(R_i)\}_{i = 1}^N$ is an affine open cover of $X$. 
  We then have a corresponding finite, compact
  cover $\{U_i^\beth\}_{i = 1}^N$ of $X^\beth$. Suppose when 
  we restrict $A$ to each $U_i$, we have an lsc function. I claim this 
  implies $A$ is lsc. Indeed, fix $r \in \R$, and let 
  $\mathcal{V}_i \subset U_i^\beth$ be $A|_{U_i^\beth}\inv(r, \infty]$,
  which is an open subset of $U_i^\beth$. Then $\mathcal{V} = \cup_i \mathcal{V}_i$ 
  is an open subset, since $\mathcal{V} \cap U_i^\beth = \mathcal{V}_i$ is open
  for all $i$. By definition, $\mathcal{V} = A\inv(r, \infty]$.  
  Thus, it suffices to check that $A$ is lower-semicontinuous when 
  $X = \Spec(R)$ is affine. Let $D = \oplus_e D_e = \Gamma(X, \shD)$. Each
  $\psi \in D_e$ has globally defined log-discrepancy 
  $A_\psi : X^\beth \to \R_{\pm \infty}$,
  and $A$ is built from these $A_\psi$ as in
  \eqref{lsc sheaf lemma}, taking $\shG$ to be the constant sheaf associated to 
  $(\mathcal{U} \mapsto \sqcup_e \{A_\psi \,:\, \psi \in D_e\})$, where 
  $\mathcal{U} \subset X^\beth$ is an open subset. Therefore, it suffices
  to fix $0 \ne \psi \in D_e$, for some $e > 0$, and prove $A_\psi$ is lsc. 

  Recall the notation $\ev_g: X^\beth \to [0, \infty]$ for 
  $(\zeta \mapsto \zeta(g))$, $g \in R$. By definition of the topology on $X^\beth$
  \eqref{definition: affine Berkovich space}, all $\ev_g$ are continuous. Thus, 
  $\ev_v\inv(\R) = \ev_g\inv[0, \infty)$ is an open subset of $X^\beth$, 
  and we can identify $\ev_g\inv(\R)$ as $h_X\inv(\Spec(R_g))$. 
  For any $f \in R$, the function $E(f, \psi, -)$ is continuous 
  on $\ev_{\psi(f)}\inv(\R)$, taking values in $(-\infty, \infty]$:
  it is the sum of the continuous
  functions $\frac{1}{p^e-1} \ev_f$ and $\frac{-p^e}{p^e-1}\ev_{\psi(f)}$. 

  Applying \eqref{no residues}, we see that for $\zeta \in X^\beth$
  with $\p = h_X(\zeta)$, 
  \[ 
  A_\psi(\zeta) = 
  \sup_{n, f} \{E(f, \psi^n, \zeta) \,:\, \ev_{\psi(f)}\inv(\R) \ni \zeta \} \]
  We then apply \eqref{lsc sheaf lemma}  to the sheaf $\shE$ of continuous functions
  \begin{equation}\label{sheaf E}
    \Gamma(\mathcal{U}, \shE) =
    \cup_{n, f} \{E(f, \psi^n, -) \,:\, \mathcal{U} \subseteq \ev_{\psi^n(f)}\inv(\R) \} 
  \end{equation}
  on an open subset $\mathcal{U} \subseteq X^\beth$, noting
  that each $E(f, \psi^n, -)$ is continuous on any open subset
  $\mathcal{U}$ for which $\mathcal{U} \subseteq \ev_{\psi^n(f)}\inv(\R)$.
\end{proof}

\begin{remark}
  The sheaf $\shE$ from \eqref{sheaf E} may have stalks $\shE_\zeta$ that
  are empty: let $Z \subset X$ be the subset where $\shD$ is not
  sharply $F$-pure; it is easy to see that $Z$ is closed. 
  For any $\zeta \in h_X\inv(Z)$, the stalk $\shE_\zeta = \varnothing$
  since $\zeta(\psi(f)) = +\infty$ for all $f \in R$. 
  In this case, $A_\psi(\zeta) = -\infty$. Thus, $A_\psi$ is
  automatically lsc on the closed set $h_X\inv(Z)$ (recall
  $h_X$ is continuous). 
\end{remark}

\begin{remark}
  If $A(v; \psi) \in \R$ for $v \in \Val_X$, and $\e > 0$, we can be more precise,
  demonstrating a basic open subset $\mathcal{U} \subseteq \Val_X$ with 
  $A(v; \psi) - \e < A(w; \psi)$ for all $w \in \mathcal{U}$. 
  Indeed, keeping notation from the proof, suppose 
  $A(v; \psi) - \e/2 < E(f, \psi^n, \zeta)$ for some
  $f \in L$ with $\psi^n(f) = 1$. One can check that 
  $\mathcal{U} = \{w \in \Val_X \,:\, |v(f) - w(f)| < (p^{en}-1)\e/2\}$ 
  has the desired property. 
\end{remark}

\subsection{$\D$-spaces and minimal $\beth$ log discrepancies}

The remainder of this section is devoted to proving the minimal $\beth$ log
discrepancy is lower-semicontinuous in the constructible topology of $X$.
The key ingredient is \cref{lsc lemma}, which is also applicable
to the log discrepancy in characteristic zero on excellent regular
schemes, or singular complex varieties. The resulting constructible
lower-semicontinuity in characteristic zero 
greatly generalizes Ambro's theorem for complex varieties
\cite[Theorem 2.2]{AmbroOnMinimalLogDiscrepancies}. This section 
is independent of the rest of the paper, 
and the $\D(X)$-spaces introduced do not come up again. 
The reader may safely skip to \Cref{Section: LCT}, if they like.  

\begin{definition}
  The {\em constructible topology} on $X$ is the minimal 
  topology containing both Zariski open and Zariski closed subsets. 
  This is equivalent to the topology with basis consisting of finite 
  unions of locally closed subsets. 
\end{definition}

\begin{definition}
  We will say that $w \in \Val_X$ is a {\em $\Z$-valuation} 
  if $\im(w) \subseteq \Z$; denote this set by $\Val_X^\Z$. 
  The {\em $\Z$-semivaluations on $X$} are the elements in the closure 
  of $\Val_X^\Z$ inside $X^\beth$, which we denote $\mf{D}(X)$. 
\end{definition}

\begin{remark}
  We define $\mathfrak{D}(X)$ to be the closure of $\Val_X^\Z$, 
  rather than the set of all $\zeta \in X^\beth$ with image contained 
  in $\Z$, because we wish to approximate $\Z$-semivaluations by 
  $\Z$-valuations. Our interest in $\Z$-valuations, as opposed
  to, e.g., all divisorial or discrete valuations, stems from the 
  uninteresting nature of minimizing log discrepancies on $X^\div$: 
  log discrepancies are homogeneous for the
  $\R_{>0}$ scaling action, meaning $A(c \,\zeta; \shD) = c\,A(\zeta; \shD)$
  for $c \in (0, \infty)$, so the minimum of $A(\zeta; \shD)$ for
  $\zeta$ in some $\R_{>0}$-invariant subset $S \subseteq X^\beth$ 
  (e.g. $X^\div$, or fibers of $h_X$ or $c_X$) is either $0$ or $-\infty$. 
\end{remark}

\begin{remark}
  A natural basis for the topology near a given $\zeta \in \D(X)$ 
  over an open affine $\Spec(R) \subseteq X$ is given by finite intersections 
  of sets of the form
  \[ \ev_{f, \D}\inv(-\e + \zeta(f), \zeta(f) + \e) = 
  \{\alpha \in \D(\Spec(R)) \,:\, |\zeta(f) - \alpha(f)| < \e\} \]
  or $\ev_{f, \D}\inv(\e, \infty]$, where $\e$ is a positive real number 
  and $f \in R$; the first type of set is used in the case $\zeta(f) < \infty$ 
  and the second when $\zeta(f) = \infty$. An important property of $\D(X)$ 
  not shared by $X^\beth$ is that 
  $\ev_f\inv(a, b) = \cup_{n \in (a, b) \cap \Z_{\ge 0}} \ev_f\inv(n)$. 
\end{remark}

\begin{lemma}
  Let $U \subseteq X$ be an open subscheme. Then 
  $\D(U) = \D(X) \cap U^\beth$ 
\end{lemma}
\begin{proof}
  First, note that for any open subscheme $U \subseteq X$ we have 
  $U^\beth = c_X\inv(U)$, and so in particular if
  $U = \Spec(R)$ is an affine open subscheme of $X$ then 
  $\Val_U^\Z = \Val_X^\Z \cap U^\beth$, as
  both sets are described as $\Z$-valuations on the function field 
  $L$ of $X$ that are non-negative on $R$. 
  Since $U^\beth$ is closed in $X^\beth$, the closure of $\Val_U^\Z$ 
  in $X^\beth$ must be contained in $U^\beth$, 
  and so agrees with $\D(U)$. This implies that 
  $\D(U) \subseteq \D(X) \cap U^\beth$. For the reverse inclusion,
  let $\zeta \in \D(X) \cap U^\beth$ and let 
  $\{v_\alpha\} \subset \Val_X^\Z$ be a convergent net with 
  limit $\zeta$. Since $c_X(\zeta) \in U$, we know $\zeta(f) \ge 0$ 
  for all $f \in R$, and so $v_\alpha \to \zeta$
  implies that there exists $\alpha_0$ such that for 
  $\alpha \ge \alpha_0$ we have $v_\alpha(f) > -1$. Since 
  $v_\alpha$ takes values only in $\Z$, we conclude that 
  $v_\alpha(f) \ge 0$ for all $\alpha \ge \alpha_0$ and 
  $f \in R$, which is to say that 
  $\{v_\alpha\}_{\alpha \ge \alpha_0} \subset \Val_U^\Z$, 
  so $\zeta \in \D(U)$. 
\end{proof}

It seems much more difficult to determine if $\D(Y) = \D(X) \cap Y^\beth$ 
when $Y \subsetneq X$ is a proper closed subscheme.  
In general, neither inclusion is clear to the author. 

\begin{lemma}
  The center function $c_\mf{D} = c_X|_\mf{D}: \mf{D}(X) \to X$ is continuous. 
\end{lemma}
\begin{proof}
  We may assume $X$ is affine, $X = \Spec(R)$. 
  Let $f \in R$ and $\zeta \in \mf{D}(X)$. Then $c_\mf{D}(\zeta) \in \V(f)$ 
  if and only if $\zeta(f) > 0$. Since $\zeta \in \mf{D}(X)$, 
  $\im(\zeta) \subseteq \Z$, and so $\zeta(f) > 0$ is equivalent to 
  $\zeta(f) \ge 1$. Therefore, 
  $c_\mf{D}\inv(\V(f)) = \ev_f\inv[1, \infty] \cap \mf{D}(X)$ is closed. 
\end{proof}

\begin{lemma}\label{closed fibers}
  For all $x \in X$, the fiber $c_\mf{D}\inv(x)$ is closed and compact.
\end{lemma}
\begin{proof}
  Let $U = \Spec(R)$ be an affine neighborhood of $x$ and suppose 
  $x \in \Spec(R)$ corresponds to the prime $\p \subset R$.  
  Since $c_\mf{D}\inv(x) \subseteq \mf{D}(U)$, we may assume that 
  $X = \Spec(R)$. Now, $c_X(\zeta) = x$  for $\zeta \in \mf{D}(X)$ 
  if and only if $\zeta(g) = 0$ for all $g \in R \setminus \p$ and 
  $\zeta(f) \ge 1$ for all $f \in \p$. These are both closed conditions.
\end{proof}

\begin{proposition}
  The image of a basic open subset of $\D(X)$ under $c_\D = c_X|_{\D(X)}$ is 
  a finite union of (Zariski) locally closed subsets, 
  and $c_\D$ induces the constructible topology as its the quotient topology.
\end{proposition}
\begin{proof}
  We assume that $X = \Spec(R)$ and that
  our basic open subset is of the form 
  $\mf{U} := \cap_{i=1}^s \left[ \cup_{j=1}^{t_i} \ev_{f_i}\inv(n_{i,j}) \right]$ 
  for some $f_i \in R$ and $n_{i,j} \in \Z_{\ge 0}$; the case involving 
  $\ev_{f_i}\inv[n, \infty]$ is very similar. 
  Suppose $\alpha \in \ev_{f_i}\inv(n)$. The condition that $\p = c_\D(\alpha)$
  is equivalent to $\p \in \mathbb{D}(f_i) := \V(f_i)^c$ when $n = 0$, and 
  $\p \in \mathbb{V}(f_i)$ when $n > 0$. For each $1 \le i \le s$, we re-number 
  so that $n_{i,j} = 0$ for $1 \le j \le t_i'$, and $n_{i, j} > 0$ for 
  $t_i' < j \le t_i$. Then the image under $c_\D$ is:
  \[ c_\D(\mf{U}) = \bigcap_{i=1}^s \left[ \bigg( \bigcup_{j=1}^{t_i'} \mathbb{D}(f_i) \bigg) \bigcup \bigg( \bigcup_{j > t_i'} 
  \mathbb{V}(f_i) \bigg), \right] \]
  which is a finite union of locally closed sets, proving our first claim. 
  
  One checks directly that the primage of $c_\D(\mf{U})$ is a basic 
  open subset of $\mathfrak{D}(X)$ (involving various $\ev_f\inv(-1/2, 1/2)$ 
  and $\ev_f\inv(0, \infty]$). Therefore, the quotient topology on $X$ 
  induced by $c_\D$ is the constructible topology. 
\end{proof}

We use the following lemma to pass lsc functions from $\D(X)$ to $X$ by 
minimizing on fibers of $c_\D$. These semicontinuity results are typically 
deduced for excellent schemes over $\Q$ using log resolutions 
({\em e.g.} \cite[Theorem 2.2]{AmbroOnMinimalLogDiscrepancies}); such results 
become special corollaries of our lemma. 

\begin{lemma}\label{lsc lemma}
  Let $f: Y \to Z$ be a continuous surjective function between topological 
  spaces. Assume that $Y$ is compact  and $Z$ is Hausdorff. Fixing 
  $a: Y \to \R_{\pm \infty }$, define a function on $Z$ by 
  \[ m(z) = \inf_{f(y) = z} a(y). \]
  Then $m$ is lsc on $Z$ whenever $a$ is lsc on $Y$. 
\end{lemma}
\begin{proof}
  If $m(z) = -\infty$ then lower-semicontinuity is automatic at $z$, so we 
  fix $z_\ast \in Z$ with $m(z_\ast) > -\infty$. 
  Suppose there exists a convergent net $z_\nu \to z_\ast$, indexed by 
  a directed set $\mc{N}$, with $\liminf_\nu m(z_\nu) < m(z_\ast)$. 
  Then for some fixed $0 < \e \ll 1$, there exists $\nu_0 \in \mc{N}$ with
  the property that for every $\nu \ge \nu_0$, there is some $\mu \ge \nu$
  making
  \[ m(z_\mu) + \e < m(z_\ast). \]
  We may therefore select, for each $\nu \ge \nu_0$, 
  $w_\nu \in \{z_\mu\}_{\mu \ge \nu}$ with
  $m(w_\nu) < m(z_\ast)$; note
  $w_\nu \to z_\ast$. Set $\mc{N}' = (\nu \ge \nu_0) \subset \mc{N}$.
  By construction, $m(w_\nu) + \e < m(z_\ast)$ for all $\nu \in \mc{N}'$.

  By definition of $m(w_\nu) = \inf_{f(y) = w_\nu} a(y)$,
  for each $\nu \in \mc{N}'$ there must exist $y_\nu \in Y$ with
  $f(y_\nu) = w_\nu$ and $a(y_\nu) < m(w_\nu) + \e < m(z_\ast)$. 
  Compactness of $Y$ allows us to pass to a convergent subnet 
  $\{y_\beta\}_{\beta \in \mc{B}}$ of $\{y_\nu\}_{\nu \in \mc{N}'}$; 
  note that $\{ f(y_\beta) \}_{\mc{B}}$ is a sub-net of
  $\{f(y_\nu)\}_{\mc{N}'} = \{w_\nu\}_{\mc{N}'}$, so
  $z_\ast \in \lim_\beta f(y_\beta)$. 
  If $y_\ast \in \lim_\beta y_\beta$, then 
  $a(y_\ast) \le \liminf_{\mc{B}} a(y_\beta) < m(z_\ast)$ because $a$ is 
  lsc on $Y$ (\ref{lsc: definition}\eqref{lsc: nets}). This gives
  a contradiction:
  \[ m(z_\ast) \le a(y_\ast) \le \liminf_{\mc{B}} a(y_\beta) < m(z_\ast). \]
  Therefore, $m$ must be lsc on $Z$. 

\end{proof}

\begin{definition}
  The {\em minimal $\beth$ log discrepancy} of $\shD$ at $x \in X$ 
  is defined to be
  \[ \mld_\beth(x; \shD) =  \min A(\zeta; \shD), \]
  where we minimize over $\zeta \in \D(X)$ with 
  $c_\D(\zeta) = x$. \Cref{closed fibers} and \Cref{log discrepancy lsc} 
  implies this minimum is achieved. 
\end{definition}

  Note that if $\mld_\beth(x; \shD) < \infty$ then any $\zeta$ achieving this minimum must have a $\shD$-compatible home $Z$. 

\begin{theorem}[cf. \cite{AmbroOnMinimalLogDiscrepancies}]\label{mld is lsc}
  For any integral scheme $X$ of characteristic $p > 0$, 
  $\mld_\beth(-; \shD): X \to \R_{\pm \infty}$ is lsc in the 
  constructible topology on $X$. Explicitly, for any $x \in X$ 
  and $\e > 0$, there is a locally closed subset $G \subset X$ containing $x$ such 
  that $\mld_\beth(x'; \shD) > \mld_\beth(x; \shD) - \e$ for all $x' \in G$. 
\end{theorem}
\begin{proof}
  The theorem follows directly from \autoref{lsc lemma}, since $\D(X)$ is compact, 
  $X^{Constr}$ is Hausdorff, $c_\D: \D(X) \to X^{Constr}$ is continuous, and 
  $A(-, \shD): X^\beth \to \R_{\pm \infty}$ is lsc \eqref{log discrepancy lsc}. 
\end{proof}

\begin{remark}
  Log discrepancies on $\Val_X$ and $X^\beth$ are also lsc over fields of characteristic
  zero, so the same proof recovers Ambro's result over $\C$, and applies more generally
  \cite{JonssonMustata, BdFFU, BlumThesis}.
\end{remark}

\begin{remark}
  It is a major open problem in the (characteristic zero) minimal model program 
  to determine if the (usual) minimal log discrepancy is lsc in the {\em Zariski} 
  topology on the set of closed points of a variety. Some results in this
  direction are known for complex varieties, cf. \cite{EinMustataYasuda, EinMustata}. 
\end{remark}

Another natural function to consider is the appropriate version of the log canonical
threshold of graded sequences of ideals with respect to strongly $F$-regular Cartier 
subalgebras. Many of the proofs found in the final sections of \cite{JonssonMustata}
can be adapted to our setting, and we devote \S \ref{Section: LCT} to this function.

%% file: zLCT-of-graded-sequences.tex
\section{Log canonical thresholds of graded sequences of ideals}
\label{Section: LCT}
\newcommand{\QM}{\mathrm{QM}}

Let us fix an integral, $F$-finite, strongly $F$-regular scheme $X$ with 
fraction field $L$. We note that $F$-finiteness implies excellence. 
We also fix a strongly $F$-regular Cartier subalgebra $\shD$ and a nonzero 
ideal $\q$ on $X$. Starting with subsection \ref{Subsection: Conjectures}, 
we assume that $X$ is regular. We skip these hypotheses when stating most 
lemmas, but make them explicit in theorems and some definitions for emphasis, 
clarity, and ease of reference. Denote by $\triv_X$ the trivial valuation 
$L^\times \to \{0\}$ and by $\Val_X^* := \Val_X \setminus \{\triv_X\}$. 

In this final section, we study log canonical thresholds of graded sequences 
of ideals in positive characteristics, proving the positive characteristic 
versions of several theorems of Jonsson and Musta\c{t}\u{a} along the way. 
We follow the same general strategy as \cite{JonssonMustata}. There are 
a number of points where we must replace parts of their approach, especially
those making use of log resolutions, with arguments involving $p^{-e}$-linear
maps, and more topological arguments involving valuation spaces. An interesting 
outcome of this approach is we give the first proof that asymptotic
multiplier ideals associated to multiplicatively graded sequences are coherent sheaves 
of ideals on strongly $F$-regular schemes \eqref{coherence}. Our approach applies
also to klt varities in characteristic zero, giving an alternative to the
usual description involving log resolutions; see \eqref{example: char 0 coherence}. 

\textbf{Convention:} To avoid constantly passing between points and 
the associated sheaves of ideals, we will make statements like 
``Let $\m \in X$ be a point$\dots$'' understanding schemes to have 
underlying topological spaces consisting of the set of prime ideal 
subsheaves of $\shO_X$ (that is, $X \isom \mathcal{S}pec(\shO_X)$). Unless 
explicitly stated otherwise, all schemes in this section are understood 
to have characteristic $p > 0$. Recall from \S\ref{graded sequences} 
that we write $\ideala_\star$ for (multiplicatively) graded sequences 
of ideals, and $\idealb_\bullet$ for an $F$-graded sequence. All 
elements of a (multiplicatively/$F$-) graded sequence of ideals are
assumed to be nonzero. 


\subsection{Log canonical thresholds of graded sequences}
Any $v \in \Val_X^*$ defines a graded sequence of ideals, denoted
here by $\ideala_\star(v)$. On an open affine chart $\Spec(R) \subset X$
containing $c_X(v)$, $\Gamma(\Spec(R), \ideala_s(v))$ is the ideal
\[ \ideala_s(v) = \{f \in R \,:\, v(f) \ge s \}; \quad s \in \N_0. \]
If $c_X(v) \not\in \Spec(R)$, we set $\Gamma(\Spec(R), \ideala_s(v)) = R$
for all $s \ge 0$. 

Jonsson and Musta\c{t}\u{a} prove the following statement showing that 
$w(\ideala_\star(v))$ compares the values of $w$ and $v$, asymptotically. 

\begin{lemma}[cf. Lemma 2.4 \cite{JonssonMustata}]\label{valuation ideal lemma}
  Let $v \in \Val_X^*$ be nontrivial. Then
  \[ w(\ideala_\star(v)) = \inf \frac{w(\idealb)}{v(\idealb)} \]
  for every $w \in \Val_X$, where $\idealb$ ranges over ideals on $X$ such that
  $v(\idealb) > 0$. 
\end{lemma}

Recall from \eqref{log discrepancy of graded sequences} 
\[ \shD \cdot \ideala_\star^t = \sum_{m \ge 1} \shD \cdot \ideala_m^{t/m} 
\quad \text{ and }\quad A(\zeta; \shD \cdot \ideala_\star^t) = 
\sup_m A(\zeta; \shD \cdot \ideala_m^{t/m}). \]


We now introduce the central topic of study in this subsection, 
log canonical thresholds of $X, \shD,$ and $\ideala_\star$ with 
respect to the nonzero ideal $\q$ on $X$; recall our standing assumptions
on $X$, $\shD$, and $\ideala_\star$. 

\begin{definition}[Log canonical threshold of graded sequence, 
  {cf. \cite{JonssonMustata, BdFFU}}]
  Let $v \in \Val_X^*$ with $v(\ideala_\star) > 0$ and $A(v; \shD) < \infty$. 
  The {\em log canonical threshold} of $\ideala_\star$ with respect 
  to $\q$, $\shD$, and $v \in \Val_X^*$ is 
  \[ \lct^\q(v; \shD, \ideala_\star) = 
  \frac{A(v; \shD) + v(\q)}{v(\ideala_\star)}. \]
  When $v(\ideala_\star) = 0$, or $A(v; \shD) = +\infty$, we define 
  $\lct^\q(v; \shD, \ideala_\star) = +\infty$. 
  The {\em log canonical threshold} of $\ideala_\star$ with respect to $\q$ 
  and $\shD$ is
  \[ \lct^\q(\shD, \ideala_\star) 
  = \inf_{v \in \Val_X^*} \lct^\q(v; \shD, \ideala_\star). \]
\end{definition}


\begin{remark}
  If $v \in \Val_X^*$ has $A(v; \shD) < \infty$ and 
  $0 < v(\ideala_\star)$, then \eqref{twisting corollary} implies
  \begin{equation}\label{lct as sup 1}
    \lct^\q(v; \shD, \ideala_\star) = 
    \sup \{ t \ge 0 \,:\, v(\q) + A(v ; \shD \cdot \ideala_\star^t) > 0 \}, 
  \end{equation}
  which is equivalent to 
  \begin{equation}\label{lct as sup 2}
    \lct^\q(v; \shD, \ideala_\star) = 
    \sup \{t \ge 0 \,:\, v(\q) + A(v ; \shD \cdot \ideala_m^{t/m}) > 0 
              \text{ for some $m$} \}. 
  \end{equation}
  These expressions are quite closely related to the expressions
  for log discrepancies in terms of sub-canonical divisors in \cite{dFH}. 
  We note that \eqref{lct as sup 2} implies that $\lct^\q(-; \shD, \ideala_\star)$
  is lower-semicontinuous on $\Val_X^*$. 
\end{remark}

\begin{remark}
  In previous versions of this article, we introduced log canonical thresholds
  with respect to semivaluations $\zeta \in X^\beth$ with $\zeta(\q) < \infty$,
  taking \eqref{lct as sup 1} for the defintion. We only treated
  this on $\Val_X^*$ in any depth, and the extension to $X^\beth$ we proposed was 
  very technical, so have chosen to focus on $\Val_X^*$ in revisions. 
\end{remark}

\subsection{Asymptotic multiplier ideals of Cartier subalgebras}
\label{Subsection: coherence}

We define sheaves of ideals $\shJ(\shD \cdot \ideala_\star^t)$ 
containing information about values, and in particular minima, 
of $A(-; \shD \cdot \ideala_\star^t)$. We model our definition 
on the valuation-theoretic description of asymptotic multiplier 
ideals, see e.g. \cite{LazarsfeldPositivity2, JonssonMustata, BdFFU}. 

\begin{definition}[Asymptotic multiplier ideal of $(\shD \cdot \ideala_\star^t)$]
  \label{definition: multiplier ideal}
  Consider $X$, $\shD$, and $\ideala_\star$ as before. 
  For $t \in \R_{\ge 0}$ and an affine open 
  $U = \Spec(R) \subseteq X$, the {\em asymptotic multiplier ideal} of 
  $(\shD \cdot \ideala_\star^t)$ on $U$ is
  \begin{align*}
    \Gamma(U, \shJ(\shD \cdot \ideala_\star^t))
    &= \bigcap_{v \in \Val_U^*} \{f \in R \,:\, v(f) + 
  A(v; \shD \cdot \ideala_\star^t) > 0 \} \\
    &= \bigcap_{v \in \Val_U^*} \{f \in R \,:\, v(f) + 
  A(v; \shD \cdot \ideala_m^{t/m}) > 0 \,\text{ for some $m \ge 1$}\}.
  \end{align*}
\end{definition}

We now prove that this sheaf of abelian groups gives a coherent sheaf of ideals.
The following lemma is used to make compactness arguments several times 
throughout this section. Compactness statements of this form are well-known
to experts, as is their usefulness when applying
valuation spaces to the study of multiplier ideals; cf. 
\cite[Proposition 5.9]{JonssonMustata} \cite[Theorem 3.1(c)]{BdFFU}. 
We choose to re-state the hypotheses on $X$ and $\shD$, 
since this lemma is cited frequently (and in the introduction). 

\begin{lemma}[{cf. \cite{JonssonMustata, BdFFU}}]\label{V_t compact}
  Let $X$ be an integral, $F$-finite, strongly $F$-regular scheme, 
  $\shD$ a strongly $F$-regular Cartier subalgebra on $X$, and 
  $\ideala \ne 0$ an ideal on $X$. For any $t \in [0, \infty)$, the set
  \[ \mathcal{V}_t := \{\zeta \in X^\beth \,:\, \zeta(\ideala) = 1, \quad A(\zeta; \shD) \le t \} \]
  is a compact subset of $\Val_X^*$. 
\end{lemma}
\begin{proof}
  This set is closed, since $(\ideala \mapsto \zeta(\ideala))$ is 
  continuous on $X^\beth$ and $A(-; \shD)$ is lower-semicontinuous; 
  thus, $\mathcal{V}_t$ is compact since $X^\beth$ is. 
  Note also that $\mathcal{V}_t \subseteq \Val_X^*$, since strong 
  $F$-regularity of $\shD$ implies
  $A(\zeta; \shD) = +\infty$ for every $\zeta \in X^\beth$ with 
  $h_X(\zeta) \ne \eta_X$, 
  and $\zeta(\ideala) = 1$ implies that $\zeta \ne \triv_X$. 
  Thus, $\mathcal{V}_t$ is a compact subset of $X^\beth$ contained 
  in $\Val_X^*$. 
\end{proof}

\begin{theorem}\label{coherence}
  Suppose $X$ is an integral, $F$-finite, strongly $F$-regular scheme 
  and $\shD$ is a strongly $F$-regular Cartier subalgebra on $X$. 
  The multiplier ideal $\shJ(\shD \cdot \ideala_\star^t)$ is a coherent 
  sheaf for every graded sequence of ideals 
  $\ideala_\star$ on $X$ and every $t \in [0, \infty)$. 
\end{theorem}
\begin{proof}
  Coherence is a property on each affine chart of $X$, so we reduce to 
  the case $X = \Spec(R)$, and only must check that
  the multiplier ideal is preserved by localizing at a single element  of$R$. 
  If $t < \lct(X; \shD, \ideala_\star)$ then 
  the multiplier ideals in question are $\shO_X$, so we assume 
  $\lct(X; \shD, \ideala_\star) < \infty$, and 
  $t \ge \lct(X; \shD, \ideala_\star)$. 

  I claim that it is enough to show the theorem when $\ideala_\star$ is 
  a constant sequence, meaning $\ideala_s = \ideala$ for all $s \ge 1$ 
  and some fixed $\ideala \subseteq \shO_X$. To ease notation in proving 
  this claim, for localizations $R_g$ of $R$ we define 
  $\shJ_\star(R_g) = \Gamma(\Spec(R_g), \shJ(\shD \cdot \ideala_\star^t))$ 
  (resp. $\shJ_m(R_g) = \Gamma(\Spec(R_g), \shJ(\shD \cdot \ideala_m^{t/m})$).
  If $\shJ(R_g)_m = \shJ(R)_mR_g$ for all $m \ge 1$, then 
  \[ \shJ(R)_\star R_g = \left( \sum_m \shJ(R)_m \right) R_g = 
  \sum_m (\shJ(R)_m R_g) = \sum_m \shJ(R_g)_m = \shJ(R_g)_\star. \]
  
  Thus, we simplify the setting and notation, writing 
  $\shJ(R_g)$ for the ideal associated to 
  $\shE := \shD \cdot \ideala^t$ as above. 
  We must prove that $\shJ(R_g) = \shJ(R)R_g$ for every $g \in R$. 
  Of course, if $g \in \sqrt{\shJ(R)}$, then
  $R_g = \shJ(R)R_g \subseteq \shJ(R_g) \subseteq R_g$ since $\shJ(-)$ is a 
  sheaf on $X$. Thus, we fix $g \in R \setminus \sqrt{\shJ(R)}$ and set $U := \Spec(R_g)$. 


  We wish to show that $y \not\in \shJ(R)R_g$ implies 
  $y \not\in \shJ(R_g)$, and it is enough to check this for $y \in R$.
  We do this by showing that there exists $w \in \Val_X^*$ such that 
  $w(g) = 0$ and $w(y) + A(w; \shE) \le 0$, which implies $w \in \Val_U^*$
  and $y \not\in \shJ(R_g)$. Since $y \not\in \shJ(R)R_g$, 
  we know that $g^n \not\in (\shJ(R) : y)$ for any $n \ge 0$, 
  so by definition for each $n \ge 0$ there exists some
  $w_n \in \Val_X^*$ such that
  \begin{equation}\tag{$\dagger_n$}\label{inequality n} 
    w_n(g^ny) + A(w_n; \shE) \le 0. 
  \end{equation}
  Being a sum of a continuous and lower-semicontinuous function, 
  \[ w \mapsto (w(g^ny) + A(w; \shE)): \Val_X^* \to \R \cup \{+\infty\} \]
  is also lower-semicontinuous. Thus, if we denote by $\mathcal{W}_n$
  the set of all $w \in \Val_X^*$ satisfying \eqref{inequality n}, 
  each $\mathcal{W}_n$ is a closed subset of $\Val_X^*$. Because 
  $w \in \mathcal{W}_n$ are centered on $X$, 
  $w(g^{n-1}y) \le w(g^ny)$, so 
  $\mathcal{W}_n \subseteq \mathcal{W}_{n-1}$
  for all $n \ge 1$. 
  
  Note also that if $w$ satisfies \eqref{inequality n} then so does 
  $\beta \,w$ for all $\beta \in \R_{> 0}$, so 
  $\R_{> 0} \cdot \mathcal{W}_n = \mathcal{W}_n$ for each $n$. 
  We assumed that $\shD$ is strongly $F$-regular, so $A(w; \shD) > 0$ 
  for every $w \in \Val_X^*$. This implies that if 
  $w \in \mathcal{W}_0$, then $w(\ideala) > 0$: indeed, by 
  \eqref{inequality n} with $n = 0$, we have 
  \begin{equation}\label{bound above by t}
    0 < A(w; \shD) \le t\,w(\ideala) - w(y) \le t \,w(\ideala),
  \end{equation}
  Therefore, there exists an $\R_{> 0}$ multiple of $w$ with $w(\ideala) = 1$. 
  Now considering $w \in \mathcal{W}_0$ with $w(\ideala) = 1$, 
  \ref{bound above by t} tells us $A(w; \shD) \le t$. It follows that 
  $\mathcal{W}_n \cap \mathcal{V}_t$ is non-empty for all $n \ge 0$, where
  \[ \mathcal{V}_t := \{\zeta \in X^\beth \,:\, \zeta(\ideala) = 1, 
        \quad A(\zeta; \shD) \le t \} \subseteq \Val_X^*, \text{ as in \eqref{V_t compact}.}\]
  Being the intersection of two non-empty compact sets,
  $\widetilde{\mathcal{W}}_n := \mathcal{W}_n \cap \mathcal{V}_t$
  is also a non-empty compact subset of $\Val_X^*$. A descending chain 
  of non-empty compact subsets has non-empty intersection, hence there
  exists $w_\infty \in \cap_n \widetilde{\mathcal{W}}_n$. 
  By construction, $w_\infty$ satisfies the inequalities \eqref{inequality n}
  for all $n \ge 1$, which is impossible if $w_\infty(g) > 0$. 
  Thus, $w_\infty(g) = 0$, or equivalently $c_X(w_\infty) \in U$. 
\end{proof}

\begin{remark}\label{example: char 0 coherence}
  As mentioned in the introduction, our argument above can be used to 
  prove that asymptotic multiplier ideals are coherent 
  whenever one has a statement such as \cref{V_t compact}. 
  
  For example, suppose $X$ is a normal variety over $\C$, or is 
  regular and excellent over $\Q$. Take a graded
  sequence of ideals $\ideala_\star$ on $X$ and suppose 
  $N \subseteq X$ is a closed, proper subscheme containing the singular locus 
  of $X$ and the support of $\ideala_1$ (hence of all $\ideala_s$). We 
  may reduce, as above, to the case that all 
  $\ideala_s = \ideala \subseteq \shO_X$, and following that argument 
  can produce the closed subsets $\mathcal{W}_n$. 
  Then \cite[Proposition 5.9]{JonssonMustata} and 
  \cite[Theorem 3.1(iii)]{BdFFU} prove 
  \[ \mathcal{V}_t = \{v \in \Val_X^* \,:\, A_X(v) \le t \,\,\text{ and }\,\, v(I_N) = 1\} \]
  is a compact subset for all $t$, so $0 < v(\ideala) \le v(I_N)$ 
  implies that $\mathcal{W}_n \cap \mathcal{V}_t$ is non-empty for all $n$. 
  Thus, our argument above proves that $\shJ(X, \ideala_\star^t)$ is coherent. 
  We remark, however, that this compactness requires simplicial
  decompositions of $\Val_X$, in \cite{JonssonMustata}, and 
  $\{ v \,:\, v(I_N) = 1\}$ in \cite{BdFFU}. These are provided 
  by log resolutions. 
\end{remark}

\begin{question}
  Suppose $X$ is a connected normal klt variety of positive characteristic, 
  in the sense that $A(\ord_E; \shC_X) > 0$ for every divisor $E \subset Y$ 
  on a normal variety admitting a proper birational morphism $Y \to X$. 
  Are the sets $\mathcal{V}_t \subseteq X^\beth$ defined as in 
  \eqref{V_t compact}, with $\shD = \shC_X$, contained in $\Val_X^*$? 
  Note that while $\mathcal{V}_t$ is always a compact subset of $X^\beth$, 
  by lower-semicontinuity of $A(-; \shC_X)$, there are now uniformly
  $\shC_X$-compatible subschemes when $X$ is not strongly $F$-regular. 
\end{question}

\subsection{The Conjectures of Jonsson and Musta\c{t}\u{a}}
\label{Subsection: Conjectures}

For the remainder of this section, 
\textbf{we assume $X$ is a regular $F$-finite scheme}. 
Suppose $\lct^\q(\shD, \ideala_\star) < \infty$. 
Our first goal is to prove that there exist valuations $v \in \Val_X^*$
with $\lct^\q(v; \shD, \ideala_\star) = \lct^\q(\shD, \ideala_\star)$;
we say any such $v$ {\em computes} this lct. 
We then prove the implications between the
conjectures numbered 7.4 and 7.5 in \cite{JonssonMustata} hold
also in our setting. These theorems
ultimately reduce to affine-local considerations, and so we
assume $X = \Spec(R)$. To simplify notation, we write 
$\shJ_t$ for $\Gamma(X, \shJ(\shD \cdot \ideala_\star^t))$. 


If $v \in \Val_X^*$ computes 
$\lct^\q(\shD, \ideala_\star) =: \lambda < \infty$, I claim
$c_X(v) \in \V(\shJ_\lambda : \q)$, cf. the beginning of 
\cite[\S 7]{JonssonMustata}. Indeed, if $v$ computes 
$\lct^\q(\shD, \ideala_\star)$, then 
$v(\q) + A(v; \shD \cdot \ideala_\star^\lambda) = 0$.
Now consider $f \in (\shJ_\lambda : \q)$. Since 
$f\q \subseteq \shJ_\lambda$, the definition of $\shJ_\lambda$
forces 
$(v(f) + 0) = (v(f) + v(\q) + A(v; \shD \cdot \ideala_\star^\lambda)) > 0$. 
Therefore, $v(f) > 0$ and $c_X(v) \in \V(\shJ_\lambda : \q)$. 

The following lemmas allow us to pass Abhyankar valuations on 
an $F$-finite regular scheme to an appropriate affine 
space via completion. The proofs for these lemmas given in 
\cite{JonssonMustata} are independent of characteristic.

\begin{lemma}[cf. {{\cite[Lemma 3.10]{JonssonMustata}}}]
  \label{quasimonomial preserved: completion}
  Let $\m \in X$ and consider the completion morphism 
  ${X'} = \Spec(R') \to X$, where $R' = \widehat{\shO_{X, \m}}$. 
  If $v' \in \Val_{X'}$ is centered at the closed point and $v = v'|_X$, 
  then $\kappa(v) = \kappa(v')$ and the value
  groups of $v$ and $v'$ are equal. In particular, 
  $\trdeg_{X'}(v') = \trdeg_X(v)$ and $\ratrk(v') = \ratrk(v)$. 
\end{lemma}

\begin{lemma}[cf. {{\cite[Lemma 3.11]{JonssonMustata}}}]
  \label{quasimonomial preserved: algebraic extension}
  Let $k \subseteq K$ be an algebraic field extension and 
  $\phi: \A^n_K \to \A^n_k$ the corresponding morphism of affine spaces.
  Suppose that $v'$ is a valuation on $K(x_1, \dots, x_n)$ with center 
  $0 \in \A^n_K$, and let $v$ be its restriction to 
  $k(x_1, \dots, x_n)$. Then $v$ is centered at 
  $0 \in \A^n_k$, $\trdeg_{\A^n_K}(v') = \trdeg_{\A^n_k}(v)$, 
  and $\ratrk(v') = \ratrk(v)$. 
\end{lemma}

Jonsson and Musta\c{t}\u{a} prove the next lemma for the more general 
setting of a regular morphism, but the only case we ever apply it to 
here is that of completion of local rings of $X$, so content ourselves 
to prove it in this case. In this subsection, we have assumed $X$ is 
regular, hence Gorenstein, so $\shD \subseteq \shC^X$ is equal
to $\shC^X \cdot \idealb_\bullet$ for some $F$-graded sequence of ideals. 
Indeed, $\shC^X_e \isom \Hom_R(F^e_*R, R)$ with its right $R$-module 
structure (i.e. the $F^e_*R$-module structure) 
is a canonical module for $R$, thus isomorphic to $R$. 

\begin{lemma}\label{completion lemma}
  Let $\m \in X$ and let $X' = \Spec(R') \to X$ 
  be the completion morphism at $\m$. 
  Set $\ideala'_\star = \{\ideala_s R'\}_{s \ge 1}$ 
  and $\idealb'_\bullet = \{\idealb_e R'\}_{e \ge 0}$, 
  and extend $\shD = \shC_X \cdot \idealb_\bullet$ to 
  $\shD' = \shC_{X'} \cdot \idealb'_\bullet$. 
  Suppose $\lambda := \lct^\q(\shD, \ideala_\star) < \infty$. 
  Let $\shH_\lambda = \shJ(X'; \shD' \cdot (\ideala_\star')^\lambda)$. 
  \begin{enumerate}
    \item For any $v' \in \Val_{X'}$ and $t \ge 0$, we have 
      $A(v'|_R; \shD \cdot \ideala_\star^t) = 
      A(v'; \shD' \cdot \ideala_\star'^t)$. \label{completion lemma (a)}
    \item $\lct^{\q R'}(\shD', \ideala_\star') \ge \lambda$ with 
      equality when $\m$ is a minimal prime of $(\shJ_\lambda : \q)$.
      \label{completion lemma (b)}
    \item If $\m$ is a minimal prime of $(\shJ_\lambda : \q)$, then 
      $\sqrt{(\shJ_\lambda :_R \q)R'} = \sqrt{(\shH_\lambda:_{R'} \q R')}$.
      \label{completion lemma (c)}
  \end{enumerate}
\end{lemma}
\begin{proof}

  We start with \eqref{completion lemma (a)}. 
  Fix $v' \in \Val_{X'}$ and let $v = v'|_R$. By definition, 
  $v_F(\idealb_\bullet) = v'_F(\idealb'_\bullet)$ and 
  $v(\ideala_\star) = v'(\ideala_\star')$. Since $X$ is $F$-finite and $R'$
  is a faithfully flat $R$-algebra, $R' \otimes_R (\shC_X)_e \isom (\shC_{X'})_e$
  for all $e \ge 1$, see \eqref{completion isomorphism}. A direct calculation, using
  \eqref{twisting corollary}, 
  shows that  $A(v; \shD \cdot \ideala_\star^t) = A(v'; \shD' \cdot \ideala_\star'^t)$.

  Moving on to \eqref{completion lemma (b)} and \eqref{completion lemma (c)}, 
  we just proved 
  $\lct^{\q R'}(v'; \shD' , \ideala_\star') = \lct^{\q}(v'|_R; \shD, \ideala_\star)$
  for all $v' \in \Val_{X'}^*$. Taking infema over all $v' \in \Val_{X'}^*$ and 
  $v \in \Val_X$, we arrive at the inequality claimed in 
  \eqref{completion lemma (b)}. 
  
  Now assume $\m$ is a minimal prime of 
  $(\shJ_\lambda : \q)$, so $\sqrt{(\shJ_\lambda : \q)R_\m} = \m R_\m$. 
  Since $R \to R'$ is flat, 
  \[ 
  (\shJ_\lambda : \q)R' = (\shJ_\lambda R' : \q R'), 
  \text{ hence } \sqrt{(\shJ_\lambda R' : \q R')} = \m R'. 
  \]
  Note also that 
  $\shJ_\lambda R' \subseteq \shH_\lambda$: if $f \in \shJ_\lambda$, then
  $v(f) + A(v; \shD \cdot \ideala_\star^\lambda) > 0$ for all $v \in \Val_X^*$. 
  In particular, this is true for all $v$ of the form $v'|_R$, $v' \in \Val_{X'}^*$,
  so $f \in \shH_\lambda$. Therefore, to show equality of radicals, it 
  suffices to prove $1 \not\in (\shH_\lambda : \q R')$. This is easily seen
  to be equivalent to $\lambda \ge \lct^{\q R'}(\shD', (\ideala_\star'))$, 
  so we reduce to proving the equality of log canonical thresholds in 
  \eqref{completion lemma (b)}.  

  The key observation is that any $v \in \Val_X$ with $c_X(v) = \m$ has
  an extension, by $\m$-adic continuity, to $\hat{v} \in \Val_{X'}$. A corollary
  is that $(v' \mapsto v'|_R)$ gives a surjection 
  \[ (\Val_{X'} \cap c_{X'}\inv(\m R')) \to (\Val_X \cap c_X\inv(\m)). \]
  Since $1 \not\in (\shJ_\lambda : \q)R_\m$, and the radical of this colon
  is $\m R_\m$, we know $1 \in (\shJ_\lambda : \q)R_\p$ for all $\p \subsetneq \m$. 
  It must thus be the case that $v(\q) + A(v; \shD \cdot \ideala_\star^\lambda) \le 0$
  for some $v \in \Val_X \cap c_X\inv(\m)$. But then 
  $\hat{v}(\q) + A(\hat{v}; \shD \cdot \ideala_\star^\lambda) \le 0$, 
  where $\hat{v}$ is the continuous extension of $v$ to $R'$. This implies
  $1 \not\in (\shH_\lambda : \q R')$. 

\end{proof}

\begin{remark}
  Jonsson and Musta\c{t}\u{a} show that in fact $\shH_t = \shJ_t R'$ for all 
  $t \ge 0$, but their proof uses that one can base change log resolutions 
  (used to compute $\shJ_t$) along regular morphisms. 
  Lacking this technique for computing $\shJ_t$, we do not know if this holds 
  in our setting. 
\end{remark}

\begin{theorem}[cf. {\cite[Theorem 7.8]{JonssonMustata}}]\label{TFAE}
  Let $v \in \Val_X^*$ with $A(v; \shD) < \infty$. 
  The following assertions are equivalent:
  \begin{enumerate}
    \item There exists $\ideala_\star$ on $X$ such that $v$ computes 
      $\lct^\q(\shD, \ideala_\star) < \infty$. \label{7.8 i}
    \item For every $w \in \Val_X^*$ with $w(\ideala) \ge v(\ideala)$ 
      for every ideal $\ideala \subset \shO_X$,
      $A(w; \shD) + w(\q) \ge A(v; \shD) + v(\q)$. \label{7.8 iii}
    \item The valuation $v$ computes $\lct^\q(\shD, \ideala_\star(v)) < \infty$. 
      \label{7.8 iv}
  \end{enumerate}
\end{theorem}
Our proof is very similar to the one found in \cite{JonssonMustata}. Notice, 
however, that we do not have part (ii) of their theorem, since the multiplier 
ideals $\shJ_t$ are not known to be subadditive. A proof of subadditivity in 
characteristic $p > 0$ would have to be quite different from the characteristic 
zero case, since Kawamata-Viehweg Vanishing is used in the proof, which is
known to be \textbf{false} in positive characteristics, see
\cite{Raynaud,XieEffectiveNonVanishing}. 

\begin{proof}
    Certainly \eqref{7.8 iv} implies \eqref{7.8 i}. 
    We prove \eqref{7.8 i} $\Rightarrow$ \eqref{7.8 iii} $\Rightarrow$ 
    \eqref{7.8 iv}.  
    For the first implication, suppose $v$ computes 
    $\lct^\q(\shD, \ideala_\star) < \infty$, so that 
    in particular $v(\ideala_\star) > 0$. If $w(\ideala) \ge v(\ideala)$ 
    for all ideals $\ideala$ on $X$, then 
    \eqref{valuation ideal lemma} shows $w(\ideala_\star) \ge v(\ideala_\star)$. 
    Since $v$ minimizes $\lct^\q(-; \shD, \ideala_\star)$, we know
    \[ \frac{A(w; \shD) + w(\q)}{w(\ideala_\star)} \ge \frac{A(v; \shD) + 
    v(\q)}{v(\ideala_\star)} \]
    implying that 
    \[ \frac{A(w; \shD) + w(\q)}{A(v; \shD) + v(\q)} \ge 
    \frac{w(\ideala_\star)}{v(\ideala_\star)} \ge 1 \]
    proving \eqref{7.8 iii}. 

    Supposing \eqref{7.8 iii} holds, we prove \eqref{7.8 iv}. 
    We see from the definition that $v(\ideala_\star(v)) = 1$, and so assertion 
    \eqref{7.8 iv} is equivalent to proving
    \[ \lct^\q(w; \shD, \ideala_\star(v)) \ge A(v; \shD) + v(\q). \]
    If $w(\ideala_\star(v)) = 0$ then 
    $\lct^\q(w; \shD, \ideala_\star(v)) = +\infty$ and this is trivial. We 
    therefore assume $w(\ideala_\star(v)) > 0$, in which case the desired 
    inequality is
    \begin{equation}\label{desired inequality}\tag{$\ast$}
      \frac{A(w; \shD) + w(\q)}{w(\ideala_\star(v))} \ge A(v; \shD) + v(\q). 
    \end{equation}
    The left hand side is invariant under $\R_{>0}$-scaling on $w$, and so we 
    may assume $w(\ideala_\star(v)) = 1$. Since $\ideala_\star(v)$ is the
    sequence of valuation ideals for $v$, if $w(\ideala_\star(v)) = 1$ then 
    $w(\ideala) \ge v(\ideala)$ for all ideals $\ideala$. 
    But then  \eqref{desired inequality} holds by assumption \eqref{7.8 iii}. 
\end{proof}

Still following the approach of Jonsson and Musta\c{t}\u{a}, we now prove that 
we may modify $\ideala_1$ and $\q$ so that they are both locally primary to a 
chosen minimal prime of $(\shJ_\lambda : \q)$. 
\begin{lemma}\label{v(c_star)}
  For $s \in \N_1$ and an ideal $\m$ 
  on $X$, define 
  \begin{equation}\label{enlarging ideal sequence}
    \idealc_j = \sum_{i = 0}^{j} \ideala_i \m^{s(j-i)} 
  \end{equation}
  and $\idealc_\star = \{c_j\}_{j \ge 1}$. 
  Then $\idealc_\star$ is a graded sequence, and 
  $v(\idealc_\star) = \min\{v(\ideala_\star), v(\m^s)\}$. 
\end{lemma}
\begin{proof}
  That $\idealc_\star$ is graded follows definitionally 
  from the graded property of $\ideala_\star$. 

  Let us then prove $v(\idealc_\star) = \min\{v(\ideala_\star), v(\m^s)\}$. 
  Note that $\ideala_j + \m^{sj} \subseteq \idealc_j$ for all $j \ge 1$, 
  so $v(\idealc_\star) \le \min \{ v(\ideala_\star), v(\m^s) \}$ for all 
  $s$ and $v \in \Val_X$. 
  \begin{itemize}
    \item If $v(\m^s) \le v(\ideala_\star)$, then $-v(\ideala_i) \le -i\,v(\m^s)$ 
      for all $i \ge 1$. Therefore, 
      \[ v(\m^s) = \frac{j\,v(\m^s) - v(\ideala_i) + v(\ideala_i)}{j} 
                  \le \frac{(j-i)\,v(\m^s) + v(\ideala_i)}{j}.\]
      We then see that
      \[ \frac{v(\idealc_j)}{j} = 
      \min_{0 \le i \le j} \left\{ 
      \frac{(j-i)v(\m^s) + v(\ideala_i)}{j} 
      \right\} = v(\m^s) \]
      giving $v(\idealc_\star) = v(\m^s)$. 
    \item If $v(\m^s) > v(\ideala_\star)$, then $j\,v(\m^s) > v(\ideala_j)$ for all 
      $j \ge j_0 \gg 0$. Now, for $n > 2j_0$ we have
      \[ \frac{v(\ideala_n)}{n} \le \frac{v(\ideala_j) + v(\ideala_{n-j})}{n} < 
      \frac{v(\ideala_j) + (n - j) v(\m^s)}{n} < v(\m^s). \]
      Therefore, $v(\idealc_\star) = v(\ideala_\star)$, since 
      \[ v(\idealc_n) = \min_j \{v(\ideala_j), (n-j)v(\m^s)\} \]
      for all $n$. 
  \end{itemize}
\end{proof}

\begin{lemma}[cf. \cite{JonssonMustata}]\label{enlarging lemma 1}
  Assume $\lct^\q(\shD, \ideala_\star) = \lambda < \infty$ and let $\m$ 
  be the generic point of an irreducible component of $\V(\shJ_\lambda : \q)$. 
  Defining $\idealc_\star$ as in \eqref{enlarging ideal sequence} with $s \gg 0$ gives
  $\lct^\q(\shD, \ideala_\star) = \lct^\q(\shD, \idealc_\star)$. If $v \in \Val_X^*$ 
  computes $\lct^\q(\shD, \idealc_\star)$, then $v$ computes 
  $\lct^\q(\shD, \ideala_\star)$. 
\end{lemma}
\begin{proof}
  Note that $\sqrt{(\shJ_\lambda : \q)} \subseteq \m$. 
  First consider the special case $\m = \sqrt{(\shJ_\lambda : \q)}$ and let $n \in \N_1$ 
  such that $\m^n\q \subseteq \shJ_\lambda$. Define   
  $\lambda' = \lct^{\m^n\q}(\shD, \ideala_\star) > \lambda$, and 
  define $\idealc_\star$ using any $s \gg 0$ with
  $n/s <(\lambda' - \lambda)$. 
  Fix $0 < \e \ll 1$ with $n/s < (\lambda'(1-\e) - \lambda)$. 

  We denote by $W$ the set of $v \in \Val_X^*$ with 
  $\lct^\q(v; \shD, \ideala_\star) < \infty$, and $W_\e = 
  \left\{v \in W \,:\, \lct^\q(v; \shD, \ideala_\star) 
  \le \frac{\lambda}{1-\e} \right\}$. Since 
  $\m^n \subseteq (\shJ_\lambda : \q)$ and $v(\shJ_\lambda : \q) > 0$, it follows 
  $v(\m) > 0$. Therefore, 
  \begin{align}
    \lct^\q(\shD, \idealc_\star)  &= \inf_{v \in W} \frac{A(v; \shD) + v(\q)}{\min\{v(\ideala_\star), v(\m^s)\}} \nonumber \\
                                  &\le \inf_{v \in W_\e} \frac{A(v; \shD) + v(\q)}{\min\{v(\ideala_\star), v(\m^s)\}} \nonumber \\
                                  &= \inf_{v \in W_\e} \left( \frac{A(v; \shD) + v(\q)}{v(\ideala_\star)} 
                                    \max \left\{ 1, \frac{v(\ideala_\star)}{v(\m^s)} \right\} \right) \label{W_e estimate}
  \end{align}
  When $v \in W_\e$ we have 
  \begin{align*}
    \lambda' - \lct^\q(v; \shD, \ideala_\star) 
                    &\ge \frac{ (1-\e)\lambda' - \lambda }{1-\e} \\
                    &> (1-\e)\lambda' - \lambda > n/s, 
  \end{align*}
  so in particular
  \begin{equation}\label{enlarging eqn 1}
    \frac{n}{s} \left( \lambda' - \lct^\q(v; \shD, \ideala_\star) \right)\inv < 1.
  \end{equation}
  Additionally, 
  \[ \lambda' - \lct^\q(v; \shD, \ideala_\star)  \le 
    \frac{n \, v(\m)}{v(\ideala_\star)} = \frac{n\, v(\m^s)}{s \, v(\ideala_\star)} 
  \]
  for all $v \in W$, so if $v(\m) > 0$ we can re-arrange this estimate and 
  apply \eqref{enlarging eqn 1}:
  \begin{equation}\label{enlarging eqn 2}
    \frac{v(\ideala_\star)}{v(\m^s)} \le \frac{n}{s} \left( \lambda' - \lct^\q(v; \shD, \ideala_\star) \right)\inv < 1.
  \end{equation} 
  Thus, by applying \eqref{enlarging eqn 2} to \eqref{W_e estimate} we have:
  \begin{align*}
    \lct^\q(\shD, \idealc_\star) &\le 
    \inf_{v \in W_\e} \left( \frac{A(v; \shD) + v(\q)}{v(\ideala_\star)} 
    \max \left\{ 1, \frac{v(\ideala_\star)}{v(\m^s)} \right\} \right) \\
    &= \inf_{v \in W_\e} \frac{A(v; \shD) + v(\q)}{v(\ideala_\star)} \\
    &= \lct^\q(\shD, \ideala_\star). 
  \end{align*}
  We conclude $\lct^\q(\shD, \idealc_\star) \le \lct^\q(\shD, \ideala_\star)$. 
  The other inequality follows from monotonicity, 
  \Cref{twisting rule}\eqref{monotonicity}. 

  Now treating the general case of $\m$ a minimal prime 
  of $(\shJ_\lambda : \q)$, we complete $X$ at $\m$ without changing
  the log canonical thresholds \ref{completion lemma}\eqref{completion lemma (b)}. 
  After completing, 
  $\sqrt{(\shJ_\lambda : \q)} = \m$,
  so we reduce to the previous case. 

  Finally, suppose $v$ computes $\lct^\q(\shD, \idealc_\star)$. Then
  \[
    \lct^\q(\shD, \ideala_\star)  = \lct^\q(\shD, \idealc_\star)
                                  = \frac{A(v; \shD) + v(\q)}{v(\idealc_\star)},
  \]
  and from $v(\idealc_\star) \le v(\ideala_\star)$ we get
  \begin{align*}
    \lct^\q(\shD, \ideala_\star) &\ge \frac{A(v; \shD) + v(\q)}{v(\ideala_\star)} \\
    &\ge \lct^\q(\shD, \ideala_\star). 
  \end{align*}
  Therefore, the inequalities are equalities, proving that $v$ also 
  computes $\lct^\q(\shD, \ideala_\star)$. 
\end{proof}

We are ready to prove the following existence result using essentially 
the same method as Jonsson and Musta\c{t}\u{a}. 

\begin{theorem}[{cf. \cite[Theorem 7.3]{JonssonMustata}}]
  \label{existence of computing valuation}
  Let $\lambda = \lct^\q(\shD, \ideala_\star) < \infty$. 
  For any generic point $\m$ of an irreducible component of 
  $\V(\shJ_\lambda : \q)$ there exists a valuation with center 
  $\m$ computing $\lct^\q(\shD, \ideala_\star)$. 
\end{theorem}
\begin{proof}
  Applying \cref{completion lemma}, we may assume $(R, \m)$ is local 
  and $\m = \sqrt{(\shJ_\lambda : \q)}$. Enlarging $\ideala_\star$ 
  using the previous proposition, we also assume that $\m^s \subseteq \ideala_1$ 
  for some $s \gg 0$. Since $\lambda < \infty$, we can fix $M \in \R$ 
  with $\lambda < M$, and $\lct^\q(\shD, \ideala_\star)$ is 
  unchanged by considering only those $v \in \Val_X^*$ such that 
  $v(\ideala_\star) > 0$ and
  \[ \lct^\q(v; \shD, \ideala_\star) =
  \frac{A(v; \shD) + v(\q)}{v(\ideala_\star)} \le M, \]
  or equivalently
  \[ A(v; \shD) + v(\q) \le M \, v(\ideala_\star). \]
  Because $\m^s \subseteq \ideala_1$, $\m^{st} \subseteq \ideala_t$ for all $t \ge 1$,
  and we see $0 < v(\ideala_\star) \le v(\m^s)$, i.e. $c_X(v) = \m$. 
  Then $A(v; \shD) \le A(v; \shD) + v(\q) 
  \le M v(\ideala_\star) \le M v(\m^s)$. Re-scaling $v$ if necessary, 
  we may assume $v(\m) = 1$, which now gives $A(v; \shD) \le N := Ms$. 
  Thus, $v \in \mathcal{V}_{N}$ as defined in \eqref{V_t compact}. We see that
  \[ \lct^\q(\shD, \ideala_\star) = 
  \inf_{v \in \mathcal{V}_N} \lct^\q(v; \shD, \ideala_\star), \]
  and $\lct^\q(-; \shD, \ideala_\star)$ is lsc \eqref{lct as sup 2}, so 
  compactness yields $v \in \mathcal{V}_N$ achieving this minimum. 
\end{proof}

We do not get very much information about the valuation computing 
$\lct^\q(\shD, \ideala_\star)$ above. In analogy with the discrete 
valuation case, where $\lct^\q(v; \shD, \ideala_\star) = +\infty$ for
non-divisorial discrete valuations on varieties \eqref{non-F-finite DVRs}, 
one might expect these valuations to be Abhyankar. 
Jonsson and Musta\c{t}\u{a} conjecture 
exactly this, calling Abhyankar valuations {\em quasi-monomial}, 
since their schemes are excellent over $\Q$. 

\begin{conjecture}[{cf. \cite[Conjecture 7.4]{JonssonMustata}}]
  \label{JM: Conjecture 7.4}
  Suppose $\lct^\q(\shD, \ideala_\star) = \lambda < \infty$. 
  \begin{itemize}
    \item \textbf{Weak version}: for any generic point $\m$ of an 
      irreducible component of $\V(\shJ_\lambda : \q)$ there exists 
      a Abhyankar valuation $v \in \Val_X^*$ computing 
      $\lct^\q(\shD, \ideala_\star)$ with $c_X(v) = \m$. 

    \item \textbf{Strong version}: any valuation computing 
      $\lct^\q(\shD, \ideala_\star)$ must be Abhyankar. 
  \end{itemize}
\end{conjecture}

Jonsson and Musta\c{t}\u{a} reduce this conjecture to the (potentially) 
simpler case of affine space.

\begin{conjecture}[{cf. \cite[Conjecture 7.5]{JonssonMustata}}]
  \label{JM: Conjecture 7.5}
  Suppose $X = \A^n_k$ with $n \ge 1$ and $k$ algebraically closed 
  of characteristic $p > 0$. Suppose the graded sequence of ideals 
  $\ideala_\star$ vanishes only at a closed point $x \in X$,
  and $\lct^\q(\shD, \ideala_\star) < \infty$. Then:
  \begin{itemize}
    \item \textbf{Weak version}: there is a Abhyankar 
      valuation $v$ computing $\lct^\q(\shD, \ideala_\star)$ with $c_X(v) = x$.
    \item \textbf{Strong version}: any valuation with transcendence 
      degree $0$ over $k$ computing $\lct^\q(\shD, \ideala_\star)$ 
      must be Abhyankar. 
  \end{itemize}
\end{conjecture}

As expected from \cite{JonssonMustata}, \cref{JM: Conjecture 7.4} 
may be reduced to \cref{JM: Conjecture 7.5}. 
To make reductions as they do, we need their second ``enlarging lemma.''

\begin{lemma}[{cf. \cite[Proposition 7.15]{JonssonMustata}}]
  \label{enlarging lemma 2}
  Suppose that $\lambda = \lct^\q(\shD, \ideala_\star) < \infty$ 
  and let $\m \in X$ with $\m^s \subseteq \ideala_1$. If $N > \lambda s$ 
  and $\idealr = \q + \m^N$ then $\lct^\idealr(\shD, \ideala_\star) = \lambda$. 
  Furthermore, $v \in \Val_X^*$ computes $\lct^\q(\shD, \ideala_\star)$ 
  if and only if $v$ computes $\lct^\idealr(\shD, \ideala_\star)$. 
\end{lemma}
\begin{proof}
  As noted previously, $\q \not\subseteq \shJ_\lambda$ but $\q \subseteq \shJ_t$ 
  for all $t < \lambda$. To prove $\lct^\idealr(\shD, \ideala_\star) = \lambda$, 
  it therefore suffices to show that $\m^N \subseteq \shJ_\lambda$. 
  This, of course, follows from our choice of $N > \lambda s$: we want to
  check that
  \begin{equation}\tag{$\ast$}\label{enlarging lemma 2 estimates 1}
    v(\m^N) + A(v; \shD, \ideala_\star^\lambda) > 0 
    \;\text{ for every valuation $v \in \Val_X^*$.}
  \end{equation}
  If $v(\m) = 0$, then $0 \le v(\ideala_\star) \le v(\m^s) = 0$ and
  $A(v; \shD \cdot \ideala_\star^\lambda) = A(v; \shD) > 0$ by strong 
  $F$-regularity, see \eqref{characterization of ShFP and SFR}. 
  We therefore assume $v(\m) > 0$. 
  Rescaling $v$ does not change the truth of 
  \eqref{enlarging lemma 2 estimates 1}, and
  so we assume $v(\m) = 1$. Our assumed inclusion 
  $\m^s \subseteq \ideala_1$ gives 
  $\lambda \,v(\ideala_\star) \le s\lambda < N$. Now,
  \[ v(\m^N) + A(v; \shD, \ideala_\star^\lambda) = 
  N + A(v; \shD, \ideala_\star^\lambda) = 
  N + A(v; \shD) - \lambda v(\ideala_\star) > A(v; \shD) > 0. \]
  We conclude that \eqref{enlarging lemma 2 estimates 1} holds. 

  We now prove that $v$ computes $\lct^\q(\shD, \ideala_\star)$ 
  if and only if it computes $\lct^\idealr(\shD, \ideala_\star)$. 
  The inclusion $\q \subseteq \idealr$ implies that $v(\idealr) \le v(\q)$, 
  so if $\lambda = \lct^\q(v; \shD, \ideala_\star)$, then 
  \[ \lct^\idealr(\shD, \ideala_\star) = \lambda \le 
  \lct^\idealr(v; \shD, \ideala_\star) = 
  \frac{A(v; \shD) + v(\idealr)}{v(\ideala_\star)} \le 
  \frac{A(v; \shD) + v(\q)}{v(\ideala_\star)} = \lambda \]
  and $v$ also computes $\lct^\idealr(\shD, \ideala_\star)$. 
  To finish the proof, it suffices to prove that $v(\q) = v(\idealr)$ 
  whenever $\lct^\idealr(v; \shD, \ideala_\star) = \lambda$. 
  Recalling that $v(\ideala_\star) \le v(\m^s)$, we have
  \[ v(\m) \ge \frac{v(\ideala_\star)}{s} = \frac{A(v; \shD) + 
  v(\idealr)}{\lambda s} > \frac{v(\idealr)}{N}. \]
  This implies that $v(\m^N) > v(\idealr) = \min\{v(\q), v(\m^N)\}$ 
  and completes the proof. 
\end{proof}

\begin{lemma}\label{base change of fields lemma}
  Suppose $X = \A^n_k$ with $k$ an $F$-finite field, 
  let $k \subseteq K$ be an algebraic extension, and
  let $\A^n_K \to \A^n_k$ be the associated
  morphism. Extending scalars, the Cartier subalgebra $\shD$, 
  graded sequence of ideals $\ideala_\star$, and ideal $\q$ on $\A^n_k$ 
  give $\shD', \ideala'_\star, \q'$ on $\A^n_K$.  
  For every valuation $v' \in \Val(\A^n_K)$ centered at $0$ with 
  restriction $v$ to $\A^n_\kappa$, 
  $\lct^\q(v; \shD, \ideala_\star) = \lct^{\q'}(v'; \shD', \ideala'_\star)$. 
\end{lemma}
\begin{proof}
  The proof is very similar to \eqref{completion lemma}. Specifically, 
  we know $A(v; \shD) = A(v'; \shD)$ and  
  $v(\ideala_\star) = v'(\ideala'_\star)$, so 
  $A(v; \shD \cdot \ideala_\star^t) =  A(v', \shD' \cdot (\ideala_\star')^t)$ 
  for all $t \ge 0$. As before, $v(\q) = v'(\q')$, 
  so the log canonical thresholds are equal. 
\end{proof}

\begin{lemma}\label{relative canonical lemma}
  Suppose $(R, \m)$ and $(S, \n)$ are regular $F$-finite rings, $\dim(R) \ge 2$,
  and let $\pi: Y = \Spec(S) \to X = \Spec(R)$ be a birational morphism
  with $\pi(\n) = \m$. Fix generators $\Phi_R$ and $\Phi_S$
  for $\shC^X$ and $\shC^Y$, respectively, and suppose 
  $\Phi_R = \Phi_S \cdot h_{S/R}$ for $h_{S/R} \in \Frac(R)$. 
  Then $-\div_Y(h_{S/R}) \ge 0$. 
\end{lemma}
\begin{proof}
  Consider a prime divisor $G$ on $S$.
  If $\pi_*G \ne 0$, then $\ord_G(h_{S/R}) = 0$ since $\Phi_R$
  is a unit multiple of $\Phi_S$ at the generic point of $G$. Therefore,
  we may assume $\pi_*G = 0$, and replace $R$ and $S$ with their localizations 
  at the generic points of the closure of $\pi(G)$ and $G$, respectively. 
  Then $\pi$ factors through the blow-up of $X$ at $\m$, so 
  $A(\ord_G; \shC^X) \ge A(\ord_\m; \shC^X) = \dim(R) \ge 2$. 
  Applying \eqref{RLR} to $S$,
  \[ A(\ord_G; \Phi_R) = A(\ord_G; \Phi_S \cdot h_{S/R}) 
  = A(\ord_G; \shC^Y) - \frac{\ord_G(h_{S/R})}{p-1} = 1 - \frac{\ord_G(h_{S/R})}{p-1}. \]
  Putting this all together, $-\ord_G(h_{S/R})  \ge (p-1) = (p-1)(2-1)$. 
\end{proof}

\begin{remark}
  \Cref{relative canonical lemma} generalizes the well-known effectivity of
  the relative canonical divisor for a proper birational morphism between
  smooth varieties. 
\end{remark}

\begin{theorem}[{cf. \cite[Theorem 7.5]{JonssonMustata}}]
  \label{affine conjecture implies regular conjecture}
  If the weak (resp. strong) version of \cref{JM: Conjecture 7.5} 
  holds for every $n \le N$ and $\shD = \shC^{\A^n}$, 
  then the weak (resp. strong)
  version of \cref{JM: Conjecture 7.4} holds for all $X$ 
  with $\dim(X) \le N$ and $\shD = \shC^X$. 
\end{theorem}
\begin{proof}
  Let us write $\lct^\q(X, \ideala_\star)$ for simplicity. 

  We begin with the weak versions. Suppose 
  $\lambda = \lct^\q(X, \ideala_\star) < \infty$ and let 
  $\m$ be a minimal prime of $(\shJ_\lambda : \q)$. 
  Applying Lemmas \ref{enlarging lemma 1}, 
  \ref{enlarging lemma 2}, and \ref{completion lemma}, we may 
  assume $R = k \ldb x_1, \dots, x_n \rdb$, $\m = (x_1, \dots, x_n)$,
  and that $\m^s \subset \ideala_1 \cap \q$. Since $R$ is $F$-finite,
  $k$ is also $F$-finite. 
  
  We now wish to apply \cref{completion lemma} 
  ``in reverse'' to reduce to the case of $\A^n_{\overline{k}}$. 
  Write $S = k[y_1, \dots, y_n]$, $\n = (y_1, \dots, y_n)$, and identify 
  the $\n$-adic completion of $S$ with $R$. 

  I claim that $\q$ and each ideal $\ideala_s$
  of $\ideala_\star$ has a generating set contained in $S$. 
  For any $g \in R$ there exists a sequence 
  $\{g_m\}_{m \ge 1} \subset S$ with $\lim_{m \to \infty} g_m = g$ in the $\m$-adic
  topology of $R$. We have $\m^s \subseteq \q$, so if $g$ is a 
  generator for $\q$, and $g_m - g \in \m^s$ for $g_m \in S$, then we
  can replace $g$ with $g_m$ without changing $\q$. Therefore, we can assume
  $\q$ is generated by elements of $S$. Similarly, $\m^s \subseteq \ideala_1$,
  so $\m^{st} \subseteq \ideala_t$ for all $t \ge 1$, and each ideal in 
  $\ideala_\star$ has a generating set contained in $S$. 

  Since $\q$ and $\ideala_\star$ are generated by elements of $S$, 
  we can view all of these ideals 
  as being extended from $\q^\flat \subset S$ and $\ideala^\flat_\star$
  on $\Spec(S) = \A^n_k$, where $\n^s \subseteq \q^\flat \cap \ideala^\flat_1$
  still holds. Now \cref{completion lemma} implies 
  $\lct^\q(X, \ideala_\star) = \lct^{\q^\flat}(\A^n_k, \ideala^\flat_\star)$.
  Extending scalars from $k$ to $\overline{k}$ leaves this log canonical threshold
  unchanged \eqref{base change of fields lemma}, so we are reduced
  to the setting of \cref{JM: Conjecture 7.5}. We are thereby furnished with
  a Abhyankar valuation $\overline{v}$ on $\A^n_{\overline{k}}$,
  centered at $0$ computing the log canonical threshold of interest. Restricting
  $\overline{v}$ to $S$, then extending $\n$-adically to $R$ preserves the
  Abhyankar property (cf. \eqref{quasimonomial preserved: completion},
  \eqref{quasimonomial preserved: algebraic extension}), so 
  \cref{JM: Conjecture 7.4} holds in this case. 

  We proceed to prove the implication between strong versions. 
  Our strategy will be similar, the crux being twice applying
  \cref{completion lemma} to reduce to the case of affine space. 
  Suppose $v \in \Val_X^*$ computes 
  $\lct^\q(X, \ideala_\star) = \lambda < \infty$. We wish to show that $v$ 
  is Abhyankar. If $\dim \shO_{X, c_X(v)} = 1$,
  then $\lct^\q(v; X, \ideala_\star) < \infty$ implies 
  $A(v; \shC^X, \ideala_\star) < \infty$
  so $v$ must be divisorial \eqref{non-F-finite DVRs}, hence Abhyankar. 
  Therefore, we may assume $\dim \shO_{X, c_X(v)} \ge 2$. 

  Applying \cref{TFAE}, we assume that 
  $\ideala_\star = \ideala_\star(v)$, implying $\m^s \subseteq \ideala_1$ 
  for all $s \gg 0$. Now applying \cref{enlarging lemma 2} we assume 
  that $\m^N \subseteq \q$ for some $N > \lambda s$.  
  After these replacements of $\q$ and $\ideala_\star$, $\m$ is a minimal 
  prime of $(\shJ_\lambda : \q)$. Indeed, $(\shJ_\lambda: \q) \subseteq \m$ 
  by the same argument as before, and if $\p \subsetneq \m$ then 
  $\m^a \not\subseteq \p$ for any $a > 0$. 
  This implies $\ideala_s = \ideala_s(v) \not\subseteq \p$ for any $s$, 
  so if $w \in \Val_X$ has center $\p$ then $w(\ideala_\star) = 0$,
  thus $A(w; \shC^X \cdot \ideala_\star^\lambda) = A(w; \shC^X) > 0$
  \eqref{characterization of ShFP and SFR}. We now reduce
  to the case $\trdeg_X(v) =0$. 
  
  The dimension formula 
  \cite[Theorem 15.6]{MatsumuraCommutativeRingTheory} implies that 
  $\trdeg_X(v)$ is the maximum of $\dim(\shO_{X, \m}) - \dim(\shO_{Y, c_Y(v)})$ 
  over all birational (but possibly non-proper) morphisms $Y \to X$, with $Y$ a 
  regular scheme on which $v$ is centered; fix some $Y \to X$ achieving this maximum. 
  Localizing at the centers of $v$ on $Y$ and $X$, setting $R = \shO_{X, \m}$ 
  and $S = \shO_{Y, c_Y(v)}$, we assume $\pi: Y \to X$
  corresponds to a local, birational extension $(R, \m) \subseteq (S, \n)$. 
  Fixing generators $\Phi_R$ and $\Phi_S$ for $\shC^X$ and $\shC^Y$, resp.,
  we define $h_{S/R}$ as in \eqref{relative canonical lemma}; that lemma
  implies $g := h_{S/R}\inv \in S$ since $\div_Y(g) \ge 0$ and $S$ is normal. 
  If we replace $\q \subseteq R$ by 
  $\q' = \q \cdot \shO_Y(-\frac{1}{p-1}\div_Y(g))\subseteq S$,
  and write $\ideala_\star^S$ for the graded sequence of valuation ideals
  of $v$ on $S$ (an $\n$-primary sequence, by construction of $S$),
  then a direct calculation (using \eqref{effect of right multiplication})
  proves 
  $\lct^{\q'}(\shC^S; \ideala^S_\star) = \lct^\q(\shC^R; \ideala_\star) = \lambda$,
  and $v$ computes $\lct^{\q'}(\shC^S; \ideala^S_\star)$; 
  cf. \cite[Corollary 1.8, Lemma 7.11]{JonssonMustata}. 
  
  Therefore, by replacing $X$ with $Y$, $\q$ with $\q'$, and 
  $\ideala_\star$ with $\ideala_\star^S$,
  we may assume $\trdeg_X(v) = 0$. We apply \cref{completion lemma} 
  with $\m = c_X(v)$ and reduce to the case 
  $X = \Spec(\kappa \ldb x_1, \dots, x_d \rdb)$ with $\kappa$ an $F$-finite field, 
  $v$ is a valuation centered at the closed point of $X$, computing 
  $\lct^\q(X, \ideala_\star)$, and such that $\kappa(v)$ is algebraic 
  over $\kappa$. Applying 
  \eqref{completion lemma} again in reverse, followed by 
  \eqref{base change of fields lemma} to assume $\kappa = \overline{\kappa}$,
  we conclude $v = v'|_Y$ must be Abhyankar, so $v'$ is, too. 
\end{proof}

\subsection{The Monomial Case}
We now prove the strong form of \cref{JM: Conjecture 7.5} 
when each $\ideala_s$ is generated by monomials. First, we need
a complementary result to \Cref{non-decreasing for QM}, which
we can prove using a more direct approach in this special setting. 

\begin{lemma}\label{retraction lemma}
  Suppose $k$ is algebraically closed of characteristic $p$, 
  $X = \A^n_k = \Spec(k[x_1, \dots, x_n])$, and $H = \div_X(x_1 \cdots x_n)$. 
  Then for all $v \in \Val_X^*$ centered at $\m = (x_1, \dots, x_n)$, 
  $A(r_{(X, H)}(v); \shC^X) \le A(v; \shC^X)$, with strict inequality
  when $r_{(X, H)}(v) \ne v$. 
\end{lemma}
\begin{proof}
  If $n = 1$, then $v = r_{(X, H)}(v)$, so there is nothing to show. 
  Let us then assume $n \ge 2$. Set $w = r_{(X, H)}(v)$. 

  Fixing the generator $\Phi$ for $\shC^X$ that projects onto
  $(x_1 \cdots x_n)^{p-1}$, using \cref{finite type} we 
  consider only $A(v; \Phi)$ and $A(w; \Phi)$. 
  Setting $\phi = \Phi \cdot (x_1 \cdots x_n)^{p-1}$, we see
  $A(w; \phi) = 0$ using \eqref{effect of right multiplication} and
  \eqref{RLR}. Moreover, the Cartier subalgebra $\lbrak \phi \rbrak \subseteq \shC^X$
  is sharply $F$-pure, so $A(v; \phi) \ge 0$. But $w(H) = v(H)$, 
  so using \cref{effect of right multiplication} again we see 
  $A(w; \Phi) \le A(v; \Phi)$. 
  
  Assume now $v \ne w$, and let us show 
  $A(w; \Phi) < A(v; \Phi)$. 
  Since $w$ is monomial on $(X, H)$, there exists a 
  log-smooth pair $\pi: (Y, D) \succeq (X, H)$ 
  so that the local ring at $y_0 := c_Y(w)$ has
  dimension $\ratrk(w)$, and so that
  $y_1 := c_Y(v)$ is contained in the closure $Z$
  of $\{y_0\}$ in $Y$. 
  Such pairs exist over arbitrary ground fields, 
  cf. \cite[Lemma 3.6(ii)]{JonssonMustata},
  which is a stronger, but more specialized, version
  of local monomialization (cf. \cite{KnafKuhlmann}). 

  I claim $v \ne w$ implies $c_Y(v)$ is 
  not the generic point $y_0$ of $Z$. We proceed by contradiction. 
  To see this, we note that
  $\ratrk(w) \le \ratrk(v)$, since the value
  group of $w$ is contained in the value group of $v$.
  Supposing $y_1 := c_Y(v) = y_0$, the residue field $\kappa(y_0)$ is 
  a sub-field of $\kappa(v) = \shO_v/\m_v$ (where $\shO_v$ is the 
  valuation ring of $v$ and $\m_v$ is the maximal ideal), so 
  $\trdeg(\kappa(y_0)/k) \le \trdeg_Y(v)$. But then Abhyankar's 
  inequality forces $v$ to be Abhyankar, thus equal to $w$:
  \[ n = (\ratrk(w) + \trdeg(\kappa(Z)/k)) \le 
  \ratrk(v) + \trdeg_Y(v) \le n. \]
  Since we are assuming $v \ne w$, we thus conclude $y_1 \ne y_0$. 

  Let $z \in \m_{y_1} \setminus \m_{y_1}^2$ that is a unit in 
  $\shO_{Y, y_0} = \shO_{Y, c(w)}$, so the divisor
  $\div_U(z)$ is smooth in some neighborhood $U$ of $y_1$. 
  Writing $G$ for the closure of $\div_U(z)$
  in $Y$, we can still define the retraction 
  $r_{(Y, H + G)}(v) =: w'$, since $v \in \Val_{U}^*$
  and $(H + G) \cap U$ is snc on $U$. By construction,
  \begin{equation}\label{retraction estimate}
    A(w; \shC^Y) < A(w'; \shC^Y), 
  \end{equation}
  and the same idea we used to show $A(w; \shC^X) \le A(v; \shC^X)$
  can be used to conclude $A(w'; \shC^Y) \le A(v; \shC^Y)$. 
  To complete the proof, we must only study
  the transformation of log discrepancies from $X$ to $Y$.
  Let $S = \shO_{Y, c(v)}$, $R = k[x_1, \dots, x_n]_\m$,
  fix a generator $\Phi_S$ for $\shC^S = \shC^Y_{c(v)}$, and
  define $g = h_{S/R}\inv$ via $\Phi = \Phi_S \cdot h_{S/R}$. 
  By \eqref{relative canonical lemma}, $g \in S$ (since $S$ is normal). 
  Moreover, $\div_{\Spec(S)}(g)$ is supported on
  the exceptional divisor $E$ of $\pi$, since
  it corresponds to the relative canonical divisor
  $K_{Y/\A^n} = K_Y$ near $y_1$. In particular,
  $g$ is a monomial in some regular system of parameters
  for $\shO_{Y, y_0}$, hence $w(g) = w'(g) = v(g)$. 
  Applying \eqref{effect of right multiplication}, which
  says
  \[ A(u; \shC^X) = A(u; \shC^Y) + \frac{u(g)}{p-1} \]
  for all $u \in \{w, w', v\}$, to 
  \eqref{retraction estimate} now proves
  $A(w; \shC^X) < A(v; \shC^X)$. 
\end{proof}

\begin{proposition}[cf. \cite{JonssonMustata}, Proposition 8.1]
  Suppose $k$ is algebraically closed of characteristic $p$.
  Let $\ideala_\star$ be a graded sequence of monomial ideals 
  on $X = \A^n_k = \Spec(k[x_1, \dots, x_n])$, vanishing only 
  at $\m = (x_1, \dots, x_n)$, and with 
  $\lct(\shC^X, \ideala_\star) < \infty$. 
  For any nonzero ideal $\q$ on $X$, there exists a monomial 
  valuation computing $\lct^\q(\shC^X, \ideala_\star)$, 
  and any valuation computing $\lct^\q(\shC^X, \ideala_\star)$ is monomial.
\end{proposition}
\begin{proof}
  Let $H = \div(x_1 \cdots x_n)$ and $v \in \Val_X^*$. If $c_X(v) \ne \m$,
  then $v(\ideala_\star) = 0$, hence $\lct^\q(v; \shC^X, \ideala_\star) = +\infty$
  by definition. Therefore, we restrict our attention to 
  $v \in \Val_X^* \cap c_X\inv(\m)$. The retraction 
  $\overline{v} := r_{(X, H)}(v)$ satisfies $\overline{v}(\q) \le v(\q)$ 
  and agrees with $v$ on monomials in $x_1, \dots, x_n$, so 
  $v(\ideala_\star) = \overline{v}(\ideala_\star)$. Applying
  \eqref{retraction lemma}, we see
  \begin{equation}\label{retraction bound}
      \lct^\q(\overline{v}; \shC^X, \ideala_\star) = \frac{A(\overline{v}; \shC^X) + \overline{v}(\q)}{\overline{v}(\ideala_\star)} 
  \le \frac{A(v; \shC^X) + v(\q)}{v(\ideala_\star)} = \lct^\q(v; \shC^X, \ideala_\star) 
  \end{equation}
  for every ideal $\q$, with strict inequality when $v \ne \overline{v}$. 
  Thus we see (using \cref{existence of computing valuation}) there must 
  be a monomial valuation computing
  $\lct^\q(\shC^X, \ideala_\star)$, and \eqref{retraction bound} shows
  any computing valuation is monomial. 
\end{proof}